\documentclass[a4paper,reqno]{amsart}%
\makeatletter
\RequirePackage{ifthen}

\provideboolean{omarfont}
\setboolean{omarfont}{false}

\provideboolean{usesvjour}
\setboolean{usesvjour}{false}%

\provideboolean{nohyperref}
\setboolean{nohyperref}{false}%

\provideboolean{usecleveref}
\setboolean{usecleveref}{false}%

\provideboolean{landscape}
\setboolean{landscape}{false}%

\makeatletter
      \RequirePackage{ifthen}
      \provideboolean{useutopia}
      \setboolean{useutopia}{false}%
      \provideboolean{usesvjour}
      \setboolean{usesvjour}{false}%
      \provideboolean{usesiam}
      \setboolean{usesiam}{false}%
      \provideboolean{usebeamer}%
      \setboolean{usebeamer}{false}%
      \provideboolean{usinghyperref}
      \setboolean{usinghyperref}{false}%
\makeatletter
          \RequirePackage{xparse}
          \RequirePackage{ifthen}
          \RequirePackage{iftex}
          \provideboolean{usemathrsfs}%
          \provideboolean{useutopia}
          \setboolean{useutopia}{false}%
          \provideboolean{usebeamer}%
          \setboolean{usebeamer}{false}%
          \provideboolean{isthesis}%
          \provideboolean{isamsltex}%
          \provideboolean{issiamltex}%
          \ifXeTeX%
            \RequirePackage{xltxtra}%
            \ifthenelse{2=1}{
              \RequirePackage{polyglossia}
              \setdefaultlanguage[variant=british]{english}
              \setotherlanguages{french,vietnamese,russian}%
            }{
              \RequirePackage[american,main=british,french,vietnamese]{babel}
            }
            \DeclareSymbolFont{usualmathcal}{OMS}{cmsy}{m}{n}
            \DeclareSymbolFontAlphabet{\mathcalbf}{usualmathcal}
            \ifthenelse{\boolean{useutopia}}{
              \RequirePackage[utopia,euro]{mathdesign}%
              \ifthenelse{\boolean{usebeamer}}{}{
                \setmainfont{Baskerville}
              }
              \RequirePackage[OMLmathbf,OMLmathsfit]{isomath}
            }{
              \RequirePackage{amssymb}    %
              \RequirePackage{bbold}    %
            }
            \RequirePackage{mathrsfs} %
            \setboolean{usemathrsfs}{true}%
          \else%
          \RequirePackage[utf8]{inputenc}%
          \ifthenelse{\boolean{useutopia}}{%
            \RequirePackage[utopia,euro]{mathdesign}%
            \RequirePackage[OMLmathrm,OMLmathbf,OMLmathsfit]{isomath}
            \RequirePackage{bbold}    %
            \RequirePackage{baskervald}
            \DeclareSymbolFont{usualmathcal}{OMS}{cmsy}{m}{n}
            \DeclareSymbolFontAlphabet{\mathcalbf}{usualmathcal}
            \RequirePackage{stmaryrd} %
            \providecommand{\diracdelta}[1][]{\ensuremath{\deltaup_{#1}}}
            
            \providecommand{\lap}{\ensuremath{\Deltaup}}
            \providecommand{\pic}{\ensuremath{\mathrm\pi}}
            \providecommand{\measure}[1]{\ensuremath{\mathcalbf{\uppercase{#1}}}}
          }{%
            \RequirePackage{amssymb}    %
            \RequirePackage{mathrsfs} %
            \RequirePackage{bbold}    %
            \RequirePackage{stmaryrd} %
            \setboolean{usemathrsfs}{true}%
            \providecommand{\mathcalbf}{\mathcal}
          }%
          \RequirePackage[american,main=british,vietnamese]{babel}
          \fi%
          \RequirePackage{xstring}%
          \RequirePackage{chngcntr}%
          \RequirePackage{mathtools}%
          \RequirePackage{nicefrac}
          \RequirePackage{stmaryrd} %
          \RequirePackage{amsthm}%
          \ifthenelse{\boolean{usebeamer}}{}{
            \RequirePackage[shortlabels]{enumitem}%
          }
          \RequirePackage{xspace}%
          \RequirePackage{verbatim}%
          \RequirePackage{fvextra}
          \RequirePackage{fancyvrb}
          \RequirePackage{listings}%
          \RequirePackage{xcolor}
          \RequirePackage{epsfig}
          \RequirePackage{graphicx}
          \RequirePackage[space]{grffile}%
          \RequirePackage{booktabs}%
          \RequirePackage{tikz}%
          \usetikzlibrary{calc}%
          \usetikzlibrary{fadings}
      \def\olprovideenvironment{\@star@or@long\provide@environment}
      \def\provide@environment#1{%
              \@ifundefined{#1}%
                      {\def\reserved@a{\newenvironment{#1}}}%
                      {\def\reserved@a{\renewenvironment{dummy@environ}}}%
              \reserved@a
      }
      \def\dummy@environ{}
      \colorlet{a}{magenta}
      \colorlet{b}{green!75!blue}
      \colorlet{c}{yellow!87.5!red}
      \colorlet{d}{cyan}
      \colorlet{e}{red}
      \colorlet{f}{blue}
      \colorlet{g}{white}
      \colorlet{i}{black}
      \colorlet{h}{i!50!g}
      \colorlet{j}{a!75!g}

      \ifthenelse{\boolean{usinghyperref}}{%
        \providecommand{\linkedurl}[1]{\url{1}}%
        \providecommand{\linkedemail}[1]{\href{mailto:#1}{#1}}%

      }{%
        \providecommand{\linkedurl}[1]{\texttt{#1}}%
        \providecommand{\linkedemail}[1]{\texttt{#1}}%
      }
      \providecommand{\email}[1]{{\linkedemail{#1}}}
      \providecommand{\Ignore}[1]{}
      \providecommand{\ignore}[1]{}
      \providecommand{\freeze}[1]{}%
      
      \providecommand{\crossout}[1]{{\color{i!20} #1}}
      \providecommand{\highlightcolor}{a}
      \providecommand{\highlight}[1]{{\color{\highlightcolor}#1}}

      \providecommand{\memo}[1]{%
        \ensuremath{%
          \framebox{\tiny\textbf{\kern-2pt\textsf{#1}}\kern-2pt}%
        }%
        \xspace}

      \RequirePackage{alphalph}
      \provideboolean{shownotes}
      \setboolean{shownotes}{true}%
      \newcounter{margnote}[page]
      \providecommand{\mgcolor}{a}%
      \providecommand{\mgcolorset}[1]{\renewcommand{\mgcolor}{\alphalph{#1}}}
      \providecommand{\mgcolorsetbycounter}[1]{%
        \newcount\@olmodn
        \@olmodn \value{#1}\relax
        \newcount\@olmodd
        \@olmodd 6\relax
        \newcount\@olmodq %
        \newcount\@olmodr %
        \newcount\@olmodc %
        \@olmodc\@olmodd\relax
        \multiply\@olmodc by 2\relax
        \advance\@olmodc-1\relax
        \@olmodq 0\relax
        \@olmodr\@olmodn\relax
        \ifnum\@olmodr>\@olmodc
        \loop
        \advance\@olmodr-\@olmodd\relax
        \advance\@olmodq1\relax
        \ifnum\@olmodr>\@olmodc
        \repeat
        \fi
        \setcounter{tmpcounter}{\the\@olmodr}
        \stepcounter{tmpcounter}
        \mgcolorset{\value{tmpcounter}}
      }
      \providecommand{\mgcolormake}{\mgcolorsetbycounter{margnote}}
      \providecommand{\margnotecolor}{\mgcolormake}
      \providecommand{\margnotemark}{{\colorbox{\mgcolor}{\tiny\color{g}\upshape\texttt{\arabic{page}.\arabic{margnote}}}}\,}
      \providecommand{\margnote}[2][]{%
        \ifthenelse{%
          \boolean{shownotes}%
        }{%
          \stepcounter{margnote}%
          \margnotecolor%
          \margnotemark %
          \marginpar{%
            \color{\mgcolor}%
            \texttt{\bfseries{%
              \begin{minipage}{2cm}%
                \raggedright\tiny%
                \margnotemark%
                #2%
                \\
                {\ifx|#1|{}\else{ - #1}\fi}%
              \end{minipage}%
              }%
            }%
          }%
        }{%
        }%
      }%
      \providecommand{\mathnote}[2][]{%
        \ifthenelse{%
          \boolean{shownotes}%
        }{%
          \stepcounter{margnote}%
          \margnotecolor%
          \text{%
            \colorbox{g}{%
              \color{\mgcolor}%
              \texttt{\bfseries{%
                \tiny%
                    \margnotemark\,%
                    #2%
                    \ifx|#1|{}\else{ - #1}\fi%
                }%
              }%
            }%
          }%
        }{%
        }%
      }%
      \providecommand{\textnote}[2][]{%
        \ifthenelse{%
          \boolean{shownotes}%
        }{%
          \stepcounter{margnote}%
          \margnotecolor%
          \ \\
          \text{%
              \begin{minipage}{\textwidth}
              \color{\mgcolor}%
              \texttt{%
                \margnotemark%
                #2%
                \ifx|#1|{}\else{ - #1}\fi%
              }%
              \end{minipage}
          }%
        }{}%
      }%

      \providecommand{\Todo}[2][To do:]{
        \ifthenelse{\boolean{shownotes}}{
          \begin{tikzpicture}
           \node[fill=a!17]{
             \begin{minipage}{\textwidth}
               \tiny
               \texttt{#1}
               \texttt{\bfseries{#2}}
             \end{minipage}
           };
          \end{tikzpicture}
        }{}}
      \provideboolean{showrevisions}
      \setboolean{showrevisions}{true}
      \provideboolean{emphrevisions}
      \setboolean{emphrevisions}{false}
      \newcommand{\revisionsheader}{\ \clearpage\Warning{the following part is under development/revision}}
      \newcommand{\revisionsfooter}{\ \newline\Warning{end of part under development/revision}\clearpage}

      \providecommand{\HighlightBox}[2][a!6.25]{
        \begin{center}
          \begin{tikzpicture}
            \node[fill=#1]{
              \begin{minipage}{\textwidth}
                #2
              \end{minipage}
            };
          \end{tikzpicture}
        \end{center}
      }
      \providecommand{\Warning}[1]{    
        \HighlightBox[b!25]{%
          \texttt{\bfseries{\small Warning: #1}}
        }
      }

      \provideboolean{showcomments}
      \providecommand{\margincomment}[1]{
      \ifthenelse{\boolean{showcomments}}{\marginpar{\tiny #1}}{}
      }
      \provideboolean{showchanges}
      \setboolean{showchanges}{false}
      \providecommand{\changes}[2][]{%
        \ifthenelse{\boolean{showchanges}}{{%
            \ifx|#1|{}\else\margnote{#1}\fi%
            \highlight{#2}%
        }}{%
          #2}}
      \providecommand{\mathchanges}[2][]{%
        \ifthenelse{\boolean{showchanges}}{{\ifx|#1|{}\else\mathnote{#1}\fi\highlight{#2}}}{#2}}

      \providecommand{\changefromto}[3][replace with]{%
        \ifthenelse{\boolean{showchanges}}{{%
            \crossout{#2}\margnote{#1}%
          }{%
            \highlight{#3}
          }%
        }{%
          #3\xspace%
        }%
      }
      \providecommand{\ChangePar}[3][]{%
        \ifthenelse{\boolean{showchanges}}{
          {\par\textcolor{i!20}{#2}\ifx|#1|\else\margnote{#1}\fi}{\par\textcolor{a}{#3}}
        }{%
          \par #3%
        }%
      }
      \providecommand{\InsertPar}[1]{
        \ifthenelse{\boolean{showchanges}}
        {{\par$\mapsto$ \textcolor{blue}{#1}}}
        {\par #1}
      }
      
      \providecommand{\mathchangefromto}[3][]{\crossout{#2}\ifx|#1|\else\mathnote{#1}\fi\highlight{#3}}

      \let\trueMakeUppercase\MakeUppercase
      \newcommand{\UCmath}[1]{%
        \begingroup
        \ucmathlist\trueMakeUppercase{#1}%
        \endgroup
      }
      \ifthenelse{\(\boolean{useutopia}\)\OR\(\boolean{usesvjour}\)}{
        \newcommand{\ucmathlist}{%
          \def\alpha{A}%
          \def\beta{B}%
          \let\gamma\Gamma
          \let\delta\Delta
          \def\epsilon{E}%
          \def\varepsilon{E}%
          \def\zeta{Z}%
          \def\eta{H}%
          \let\theta\Theta
          \let\vartheta\Theta
          \def\iota{I}%
          \def\kappa{K}%
          \let\lambda\Lambda
          \def\mu{M}%
          \def\nu{N}%
          \let\xi\Xi
          \def\omicron{O}
          \let\pi\Pi
          \let\varpi\Pi
          \def\rho{P}%
          \def\varrho{P}%
          \let\sigma\Sigma
          \def\varsigma{C}
          \def\tau{T}%
          \let\upsilon\Upsilon
          \let\phi\Phi
          \let\varphi\Phi
          \def\chi{X}%
          \let\psi\Psi
          \let\omega\Omega
      }}{
        \newcommand{\ucmathlist}{
          \def\alpha{\mathrm{A}}%
          \def\beta{\mathrm{B}}%
          \let\gamma\Gamma
          \let\delta\Delta
          \def\epsilon{\mathrm{E}}%
          \def\varepsilon{\mathrm{E}}%
          \def\zeta{\mathrm{Z}}%
          \def\eta{\mathrm{H}}%
          \let\theta\Theta
          \let\vartheta\Theta
          \def\iota{\mathrm{I}}%
          \def\kappa{\mathrm{K}}%
          \let\lambda\Lambda
          \def\mu{\mathrm{M}}%
          \def\nu{\mathrm{N}}%
          \let\xi\Xi
          \let\pi\Pi
          \let\varpi\Pi
          \def\rho{\mathrm{P}}%
          \def\varrho{\mathrm{P}}%
          \let\sigma\Sigma
          \def\tau{\mathrm{T}}%
          \let\upsilon\Upsilon
          \let\phi\Phi
          \let\varphi\Phi
          \def\chi{\mathrm{X}}%
          \let\psi\Psi
          \let\omega\Omega
        }
      }
      \providecommand{\mathscript}
      	       {\mathscr}

      \providecommand{\cD}{\ensuremath{\mathscript D}\xspace}

      \providecommand{\bbbold}{\mathbb}

      \providecommand{\rN}{\ensuremath{\bbbold N}\xspace}
      
      \providecommand{\rP}{\ensuremath{\bbbold P}\xspace}
      
      \providecommand{\rR}{\ensuremath{\bbbold R}\xspace}
      
      \providecommand{\rT}{\ensuremath{\bbbold T}\xspace}

      \providecommand{\Ae}[1][]{\ensuremath{\ifx|#1|{\ }\else{\:#1\text{-}}\fi\text{almost everywhere\xspace}}}

      \providecommand{\Aa}[1][]{\ensuremath{\text{ for }\ifx|#1|{}\else{\:#1\text{-}}\fi\text{almost all }}}
      \providecommand{\as}[1][]{\ensuremath{\ifx|#1|{\ }\else{#1\text{-}}\fi\text{almost surely}}\xspace}
       \providecommand{\naturals}{\ensuremath{\rN}}
       
       \providecommand{\NO}[1][]{\ensuremath{\naturals_0\ifx|#1|{}\else^{#1}\fi}}

       \providecommand{\reals}{\rR}

       \providecommand{\fieldmats}[3][F]{\csname#1\endcsname{#2\times#3}}
       
       \providecommand{\fieldtens}[3][F]{\csname#1\endcsname{{#2}_1\times\dotsb\times{#2}_{#3}}}

       \providecommand{\RO}[1][]{{\reals_{0+}\ifx|#1|{}\else^{#1}\fi}}
       \providecommand{\RP}[1][]{{\reals_+\ifx|#1|\else^{#1}\fi}}
       \providecommand{\Rneg}[1][]{{\reals_-\ifx|#1|\else^{#1}\fi}}

       \providecommand{\ring}[1][A]{\csname r#1\endcsname}
       \providecommand{\field}[1][K]{\csname r#1\endcsname}

       \providecommand{\torus}[1]{\rT\ifthenelse{\equal{#1}1}{}{^#1}}
      
       \providecommand{\one}{\ensuremath{\bbbold 1}}
       \providecommand{\zerofun}{\ensuremath{\bbbold 0}}
       \providecommand{\ones}[1][]{\one\ifx|#1|\else_{#1}\fi}
       \providecommand{\zeros}[1][]{\zerofun\ifx|#1|\else_{#1}\fi}

       \providecommand{\diracdelta}[1][]{\ensuremath{{\mathrm{\delta}}\ifx|#1|{}\else_{#1}\fi}}

       \providecommand{\pic}{\pi}%
       \providecommand{\pifracl}[2][]{\fracl{\ifx|#1|\else#1\fi\pic}{#2}}
       \providecommand{\pifrac}[2][]{\frac{\ifx|#1|\else#1\fi\pic}{#2}}

       \providecommand{\take}{\smallsetminus}
       \providecommand{\takesetof}[1]{\take\setof{#1}}
       \providecommand{\takeset}\takesetof
       \providecommand{\takeel}\oldneg%

       \providecommand{\closure}[2][]{\ifx|#1|\overline{#2}\else\operatorname{clos}_{#1}{#2}\fi}
       \providecommand{\inner}{\cdot}
       \providecommand{\vecprod}{\times}
       \providecommand{\outerp}{\wedge}

       \providecommand{\W}{\ensuremath{\varOmega}\xspace}

       \providecommand{\qp}[2][]{\ensuremath{\ifx|#1|\left(\else\csname#1\endcsname(\fi{#2}\ifx|#1|\right)\else\csname#1\endcsname)\fi}}

       \providecommand{\qpreg}[1]{\ensuremath{(#1)}}
       \providecommand{\qpbig}[1]{\qp[big]{#1}}%
       \providecommand{\qpBig}[1]{\ensuremath{\Big(#1\Big)}}
       \providecommand{\qpbigg}[1]{\ensuremath{\bigg(\!#1\!\bigg)}}
       \providecommand{\qpBigg}[1]{\ensuremath{\Bigg(\!#1\!\Bigg)}}
       \providecommand{\qb}[2][]{\ifx|#1|\left[\else\csname#1\endcsname[\fi{#2}\ifx|#1|\right]\else\csname#1\endcsname]\fi}
       \providecommand{\qc}[2][]{\ensuremath{\ifx|#1|\left\{\else\csname#1\endcsname\{%
           \fi{#2}\ifx|#1|\right\}\else\csname#1\endcsname\}\fi}}

       \providecommand{\qa}[2][]{\ifx|#1|\left\langle\else\csname#1\endcsname\langle%
         \fi{#2}\ifx|#1|\right\rangle\else\csname#1\endcsname\rangle\fi}%
       \providecommand{\qareg}[1]{\ensuremath{\langle#1\rangle}}
       \providecommand{\qabig}[1]{\ensuremath{\big\langle#1\big\rangle}}
       \providecommand{\qaBig}[1]{\ensuremath{\Big\langle#1\Big\rangle}}
       \providecommand{\qabigg}[1]{\ensuremath{\bigg\langle#1\bigg\rangle}}
       \providecommand{\qaBigg}[1]{\ensuremath{\Bigg\langle#1\Bigg\rangle}}
       \providecommand{\qv}[2][]{\ifx|#1|\left|\else\csname#1\endcsname|%
         \fi{#2}\ifx|#1|\right|\else\csname#1\endcsname|\fi}%

       \providecommand{\opinter}[2]{\ensuremath{\left(#1,#2\right)}\xspace}

       \providecommand{\clinter}[2]{\ensuremath{\left[#1,#2\right]}\xspace}
       
       \providecommand{\opclinter}[2]{\ensuremath{\left(#1,#2\right]}\xspace}
       \providecommand{\clopinter}[2]{\ensuremath{\left[#1,#2\right)}\xspace}   
       \providecommand{\opintertopinfty}[1]{\opinter{#1}\infty}
       \providecommand\optyinter\opintertopinfty
       \providecommand{\opinterbotinfty}[1]{\opinter{-\infty}{#1}}
       \providecommand\tyopinter\opinterbotinfty
       \providecommand{\clintertopinfty}[1]{\clopinter{#1}\infty}
       \providecommand{\cltyinter}\clintertopinfty
       \providecommand{\clinty}\clintertopinfty
       
       \providecommand{\clinterbotinfty}[1]{\opclinter{-\infty}{#1}}
       \providecommand{\tyclinter}\clinterbotinfty
       \providecommand{\expp}[1]{\ensuremath{\e^{#1}}}
       
       \providecommand{\compowqp}[2]{\ensuremath{\qp{\!#2\!\!}^{\kern -.4em #1}\!}}
       
       \providecommand{\powqpreg}[2]{\ensuremath{%
           \qpreg{#2}^{\kern 0em\lower .1ex\hbox{\scriptsize $#1$}}\kern-.3em}}
       \providecommand{\powqpbig}[2]{\ensuremath{%
           \qpbig{#2}^{\kern -.2em\lower .3ex\hbox{\scriptsize $#1$}}\kern-.3em}}
       \providecommand{\powqpBig}[2]{\ensuremath{%
           \qpBig{#2}^{\kern -.2em\lower .3ex\hbox{\scriptsize $#1$}}\kern-.3em}}
       \providecommand{\powqpbigg}[2]{\ensuremath{%
           \qpbigg{#2}^{\kern -.2em\lower .3ex\hbox{\scriptsize $#1$}}\kern-.3em}}
       \providecommand{\powqpBigg}[2]{\ensuremath{%
           \qpBigg{#2}^{\kern -.2em\lower .3ex\hbox{\scriptsize $#1$}}}}
       \providecommand{\powp}[3][]{{#3}\ifx|#1|^{#2}\else{#1}^{#2}\fi}%
       \providecommand{\pow}[2][]{\ifx|#1|\operatorname{pow}^{#2}\else\powp{#2}{#1}\fi}%
       \providecommand{\ppow}[3][]{\powp[#1]{#3}{#2}}

       \providecommand{\normpow}[3][]{\powp{#3}{\norm[#1]{#2}}}

       \providecommand{\norm}[2][]{\ifx|#1|\left|\else\csname#1\endcsname|\fi#2\ifx|#1|\right|\else\csname#1\endcsname|\fi}
       \providecommand{\normon}[3][]{\norm[#1]{#2}_{#3}}

       \providecommand{\abs}[2][]{\ensuremath{\ifx|#1|{\left|#2\right|}\else{\csname#1\endcsname|{#2}\csname#1\endcsname|}\fi}}

       \providecommand{\Norm}[2][]{\ifx|#1|\left\|\else\csname#1\endcsname\|\fi{#2}\ifx|#1|\right\|\else\csname#1\endcsname\|\fi}

       \providecommand{\Normon}[3][]{\Norm[#1]{#2}_{#3}}

       \providecommand{\normonsob}[5][]{\normon[#1]{#2}{\sob{#3}{#4}\if|#5|{}\else(#5)\fi}}
       \providecommand{\Normonsob}[5][]{\Normon[#1]{#2}{\sob{#3}{#4}\if|#5|{}\else(#5)\fi}}
       \providecommand{\normonsobh}[4][]{\normon[#1]{#2}{\sobh{#3}\if|#4|{}\else(#4)\fi}}
       \providecommand{\normonsobhz}[4][]{\normon[#1]{#2}{\sobhz{#3}\if|#4|{}\else(#4)\fi}}
       \providecommand{\Normonsobh}[4][]{\Normon[#1]{#2}{\sobh{#3}\if|#4|{}\else(#4)\fi}}
       \providecommand{\Normonsobhz}[4][]{\Normon[#1]{#2}{\sobhz{#3}\if|#4|{}\else(#4)\fi}}
       \providecommand{\Normonleb}[4][]{\Normon[#1]{#2}{\leb{#3}\if|#4|\else(#4)\fi}}

       \providecommand{\ltwop}[3][]{\ensuremath{\qa{#2,#3}\ifx|#1|\else_{#1}\fi}}
       \providecommand{\ltwopreg}[2]{\ensuremath{\qareg{#1,#2}\ifx|#1|\else_{#1}\fi}}
       \providecommand{\ltwopbig}[2]{\ensuremath{\qabig{#1,#2}\ifx|#1|\else_{#1}\fi}}
       \providecommand{\ltwopBig}[2]{\ensuremath{\qaBig{#1,#2}\ifx|#1|\else_{#1}\fi}}
       \providecommand{\ltwopbigg}[2]{\ensuremath{\qabigg{#1,#2}\ifx|#1|\else_{#1}\fi}}
       \providecommand{\ltwopBigg}[2]{\ensuremath{\qaBigg{#1,#2}\ifx|#1|\else_{#1}\fi}}

       \providecommand{\duality}[3][]{%
         \ifx#1\left\langle#2\middle|#3\right\rangle%
         \else%
         #1\langle%
         #2%
         #1|%
         #3%
         #1\rangle%
         \fi}%
       
       \providecommand{\average}[2][]{{\qa{#2}\ifx|#1|\else_{#1}\fi}}

       \providecommand{\ensemble}[2]{\ensuremath{\left\{ #1:\;#2 \right\}}}
       \providecommand{\setofsuch}{\ensemble}%
       \providecommand{\setof}[1]{{\qc{#1}}}

       \providecommand{\conditionalto}[1]{{\left|{#1}\right.}}

      \providecommand{\measure}[1]{\ensuremath{\mathcalbf{\MakeUppercase{#1}}}}

      \providecommand{\probmeasure}[2][]{{\measure{#2}}\ifx|#1|\else_{#1}\fi}
      \providecommand{\Prob}{}
      \renewcommand{\Prob}[1][]{\probmeasure[{#1}]{p}}

      \providecommand{\randvars}[1][\Prob]{\operatorname{RV}\ifx|#1|{}\else{(#1)}\fi}
      \providecommand{\discrandvars}[1][\Prob]{\operatorname{DRV}\ifx|#1|{}\else{({#1)}\fi}} 
      \providecommand{\contrandvars}[1][\Prob]{\ensuremath{\operatorname{CDRV}\ifx|#1|{}\else(#1)\fi}} 
       \def\env@matrix{\hskip -\arraycolsep
        \let\@ifnextchar\new@ifnextchar
        \array{*\c@MaxMatrixCols c}}
       \renewcommand*\env@matrix[1][c]{\hskip -\arraycolsep
         \let\@ifnextchar\new@ifnextchar
         \array{*\c@MaxMatrixCols #1}}
       \providecommand{\irow}[2]{#1_{#2}}%
       \providecommand{\icol}[2]{#1^{#2}}%

       \providecommand{\ijrowcol}[3]{\icol{\irow{#1}{#2}}{#3}}
       \providecommand{\entry}[1]{\qb{#1}}
       \providecommand{\vecentry}[2]{\irow{#1}{#2}}

       \providecommand{\colvecentry}\vecentry
       \providecommand{\covecentry}[2]{\icol{#1}{#2}}
       \providecommand\rowvecentry\covecentry

       \providecommand{\rowof}[1]{\qb{#1}}
       \providecommand{\disvecof}[2][r]{\begin{bmatrix}[#1]#2\end{bmatrix}}
       \providecommand{\vecof}[2][r]{\mathchoice{\disvecof[#1]{#2}}{\qp{#2}}{\qp{#2}}{\qp{#2}}}

       \providecommand{\getentryi}[2]{\irow{\entry{#1}}{#2}}
       \providecommand{\getcolentry}\getentryi
       
       \providecommand{\getvecentry}[2]{\getentryi{\vec #1}{#2}}

       \providecommand{\discolvecitwo}[1]{\discolvectwo{\vecentry{#1}1}{\vecentry{#1}2}}
       
       \providecommand{\discolvecintwo}\discolvecitwo%

       \providecommand{\matof}[1]{\qb{#1}}
       \providecommand{\dismatof}[2][r]{%
         \mathchoice{%
           \begin{bmatrix}[#1]#2\end{bmatrix}%
         }{%
           \qb{
             \begin{smallmatrix}
               #2
             \end{smallmatrix}
           }
         }{%
           \qb{
             \begin{smallmatrix}
               #2
             \end{smallmatrix}
           }
         }{%
           \qb{
             \begin{smallmatrix}
               #2
             \end{smallmatrix}
           }
         }
       }

       \providecommand{\matentry}[3]{\ijrowcol{#1}{#2}{#3}}

       \providecommand{\block}[5]{\ijrowcol{#1}{\ifx#2#3{\rowof{#2}}\else\rowof{{#2}\dotsc{#3}}\fi}{\ifx#4#5{\rowof{#4}}\else\rowof{{#4}\dotsc{#5}}\fi}}
       \providecommand{\colblock}[3]{\getvecentry{#1}{\ifx#2#3{#2}\else\fromto{#2}{#3}\fi}}

       \providecommand{\matskelijnm}[5]{\ijrowcol{\matof{#1}}{#2\integerbetween1{#4}}{#3\integerbetween1{#5}}}
       \providecommand{\sqmatskelijn}[4]{\matskelijnm{#1}{#2}{#3}{#4}{#4}}
       \providecommand{\dismatskeldots}[4]{
         \dismatof[c]{
           #1&\dotsc&#3
           \\
           \vdots & \ddots &\vdots
           \\
           #2&\dotsc&#4
         }
       }
       \providecommand{\dismatcommfromtofromto}[5]{
         \dismatskeldots{#1#2#4}{#1#3#4}{#1#2#5}{#1#3#5}
       }
       \providecommand{\dismatcustfromtofromto}[6][matentry]{
         \dismatcommfromtofromto{\csname#1\endcsname{#2}}#3#4#5#6
       }
       \providecommand{\dismatcustfromtofromto}[6][matentry]{
         \dismatskeldots{%
           \csname#1\endcsname{#2}{#3}{#4}%
         }{%
           \csname#1\endcsname{#2}{#3}{#6}%
         }{%
           \csname#1\endcsname{#2}{#5}{#4}%
         }{%
           \csname#1\endcsname{#2}{#5}{#6}%
         }%
       }%
       \providecommand{\dismatcustfromtofromto}[6][matentry]{
         \dismatof{
           \csname#1\endcsname{#2}{#3}{#4}&\dotsc&\csname#1\endcsname{#2}{#3}{#6}
           \\
           \vdots & \ddots &\vdots
           \\
           \csname#1\endcsname{#2}{#5}{#4}&\dotsc&\csname#1\endcsname{#2}{#5}{#6}
         }
       }

       \providecommand{\dissysaxbdotsnm}[5]{\begin{matrix}[r]%
           \matentry{#1}11\vecentry{#2}1&+\dotsb&+\matentry{#1}1{#5}\vecentry{#2}{#5}
           &
           =
           \ifx|#3|0\else{\vecentry {#3}1}\fi
           \\
           \dotsb
           \\
           \matentry{#1}{#4}1\vecentry{#2}1&+\dotsb&+\matentry{#1}{#4}{#5}\vecentry{#2}{#5}
           &
           =
           \ifx|#3|0\else{\vecentry {#3}{#4}}\fi
       \end{matrix}}
       \providecommand{\seqof}[1]{\qp{#1}}%
       \providecommand{\seqs}[2]{\seqof{#1}_{#2}}
       \providecommand{\sets}[2]{\setof{#1}_{#2}}%
       \providecommand{\seqi}[3][]{\seqs{#2_{#3}}{\ifx|#1|{#3}\else{{#3}\in{#1}}\fi}}%
       \providecommand{\sequ}[3][]{\seqs{#2^{#3}}{\ifx|#1|{#3}\else{{#3}\in{#1}}\fi}}%
       \providecommand{\subseqi}[4][]{\seqs{#2_{{#3}_{#4}}}{\ifx|#1|{#4}\else{{#4}\in{#1}}\fi}}%

       \providecommand{\seti}[3][]{\sets{#2_{#3}}{\ifx|#1|_{#3}\else_{{#3}\in{#1}}\fi}}%
       \providecommand{\setu}[3][]{\sets{#2^{#3}}{\ifx|#1|{#3}\else{{#3}\in{#1}}\fi}}%
       \let\liminf\relax
       \DeclareMathOperator*{\liminf}{liminf}
       \let\limsup\relax
       \DeclareMathOperator*{\limsup}{limsup}
       \providecommand{\limofat}[3][]{\ensuremath{\lim_{\ifx|#1|{}\else{#1\ni}\fi#3}{#2}}}
       \providecommand{\limsupofat}[3][]{\ensuremath{\limsup_{\ifx|#1|{}\else{#1\ni}\fi#3}{#2}}}
       \providecommand{\liminfofat}[3][]{\ensuremath{\liminf_{\ifx|#1|{}\else{#1\ni}\fi#3}{#2}}}

       \providecommand{\stringdotsfrom}[3][]{\ensuremath{#2\ifx|#1|\else#1\fi\,#3\ifx|#1|\else#1\fi\,\dotsc}}
       \providecommand{\listdotsfrom}[3][]{\ensuremath{#2\ifx|#1|\else#1\fi,#3\ifx|#1|\else#1\fi,\dotsc}}
       \providecommand{\stringdotsfromto}[3][]{\ensuremath{#2\ifx|#1|\else#1\fi\,\dotsc\,#3\ifx|#1|\else#1\fi}}
       \providecommand{\listdotsfromto}[3][]{\ensuremath{#2\ifx|#1|\else#1\fi,\dotsc,#3\ifx|#1|\else#1\fi}}
       \providecommand{\listifromto}[5][]{\ensuremath{{#2}_{#3}\ifx|#1|\else#1\fi},\text{ for }\ensuremath{\rangefromto{#3}{#4}{#5}}\xspace}
       \providecommand{\listufromto}[5][]{\ensuremath{{#2}^{#3}\ifx|#1|\else#1\fi},\text{ for }\ensuremath{\rangefromto{#3}{#4}{#5}}\xspace}
       \providecommand{\listitwo}[2][]{\ensuremath{#2_1\ifx|#1|\else#1\fi,#2_2\ifx|#1|\else#1\fi}}
       \providecommand{\listutwo}[2][]{\ensuremath{#2^1\ifx|#1|\else#1\fi,#2^2\ifx|#1|\else#1\fi}}
       \providecommand{\listithree}[2][]{\ensuremath{#2_1\ifx|#1|\else#1\fi,#2_2\ifx|#1|\else#1\fi,#2_3\ifx|#1|\else#1\fi}}
       \providecommand{\listithreez}[2][]{\ensuremath{#2_0\ifx|#1|\else#1\fi,#2_1\ifx|#1|\else#1\fi,#2_2\ifx|#1|\else#1\fi}}
       \providecommand{\listifourz}[2][]{\ensuremath{#2_0\ifx|#1|\else#1\fi,#2_1\ifx|#1|\else#1\fi,#2_2\ifx|#1|\else#1\fi,#2_3\ifx|#1|\else#1\fi}}
       \providecommand{\listuthree}[2][]{\ensuremath{#2^1\ifx|#1|\else#1\fi,#2^2\ifx|#1|\else#1\fi,#2^3\ifx|#1|\else#1\fi}}

       \providecommand{\jump}[2][]{\ensuremath{\left\llbracket\ifx|#1|{\,}\else{#1}\fi #2\right\rrbracket}}
       \providecommand{\njump}[2][]{\ensuremath{\left\llbracket\normal\ifx|#1|{\,}\else{#1}\fi #2\right\rrbracket}}
       \providecommand{\jumpnt}[2][]{\ensuremath{\left\llbracket #2 \ifx|#1|{\,}\else{#1}\fi\normal\transposed\right\rrbracket}}

       \providecommand{\fromto}[2]{\ensuremath{\setof{#1\dotsc#2}}}%

       \providecommand{\integerbetween}[2]{\ensuremath{={#1},\dotsc,{#2}}}

       \providecommand{\rangefromto}[3]{\ensuremath{#1\integerbetween{#2}{#3}}}

       \providecommand{\e}{\ensuremath{\operatorname{e}\!}\xspace}

       \providecommand{\d}{}
       \renewcommand{\d}[1][]{\ensuremath{\operatorname{d}\!\ifx|#1|\else{_{#1}}\fi}}
       
       \providecommand{\ds}[1][]{\d{\measure S}}
       \providecommand{\D}[1][]{\ensuremath{\operatorname{D}\!\ifx|#1|\else{_{#1}}\fi}}

      \providecommand{\registered}%
      {\ensuremath{^\text{\textregistered}}}

      \providecommand{\tand}{\ensuremath{\text{ and }}}

      \providecommand{\constant}[1]{\ensuremath{C_{#1}}}
      \providecommand{\constext}[2][]{\constant{\textup{#2}{\ifx|#1|{}\else{,\ensuremath{#1}}\fi}}}            %
      \providecommand{\constref}[2][]{\ensuremath{\constant{\textup{\ref{#2}{\ifx|#1|{}\else{,\ensuremath{#1}}\fi}}}}}
      \providecommand{\constdef}[2][]{\label{#2}\ensuremath{\constant{\textup{\ref{#2}{\ifx|#1|{}\else{,\ensuremath{#1}}\fi}}}}}

      \providecommand{\funkref}[3][]{\ensuremath{{#3}_{\textup{\ref{#2}{\ifx|#1|{}\else{,\ensuremath{#1}}\fi}}}}}

      \providecommand{\diam}{\operatorname{diam}}
      
      \providecommand{\curl}{\operatorname{curl}}
      \renewcommand{\curl}[1][]{\nabla\ifx|#1|{}\else\kern-2pt_{#1}\fi\kern-2pt\vecprod}
      \renewcommand{\div}[1][]{\nabla\ifx|#1|{}\else\kern-2pt_{#1}\fi\kern-1pt\inner}
      \providecommand{\divof}[2][]{\div[#1]\ifx|#2|{}\else\qb{#2}\fi}%
      \providecommand{\divideabyb}[2]{\operatorname{div}(a,b)}
      \providecommand{\grad}{}
      \renewcommand{\grad}[1][]{\nabla\ifx|#1|\else_{#1}\fi}

      \providecommand{\rot}[1][]{\nabla\ifx|#1|\else_{#1}\fi\outerp}
      
      \providecommand{\rowdiv}[1][]{\D\ifx|#1|{}\else\kern-1pt_{#1}\kern-2pt\fi\cdot}
      \providecommand{\rowdivof}[2][]{\rowdiv[#1]\ifx|#2|{}\else\qb{#2}\fi}

      \providecommand{\inv}[1][]{\operatorname{inv}\ifx|#1|\else^{#1}\fi}
      \providecommand{\ivt}[1]{\operatorname{ivt}\ifx|#1|\else^{#1}\fi}
      \providecommand\tensorinvariant\ivt

      \providecommand{\mod}{}
      \renewcommand{\mod}[1][]{\operatorname{mod}\ifx|#1|\else\kern-1pt_{#1}\fi}
      
      \let\oldfrac\frac
      \renewcommand{\frac}[3][]{\ifx|#1|\oldfrac{#2}{#3}\else\begin{array}{#1}{#2}\\\hline{#3}\end{array}\fi}
      \providecommand{\fracl}[3][]{\ifx|#1|\nicefrac{#2}{#3}\else{#2}#1/{#3}\fi}

      \providecommand{\qpfracl}[3][]{\qp{\ifx|#1|\fracl{#2}{#3}\else{#2}#1/{#3}\fi}}
      \providecommand{\qpfrac}[3][]{\qp{\ifx|#1|\frac{#2}{#3}\else{#2}#1/{#3}\fi}}

      \providecommand{\absfracl}[3][]{\abs{\ifx|#1|\fracl{#2}{#3}\else{#2}#1/{#3}\fi}}
      \providecommand{\absfrac}[3][]{\abs{\ifx|#1|\frac{#2}{#3}\else{#2}#1/{#3}\fi}}
      \providecommand{\fraclff}[3][]{\ifx|#1|{#2}/{#3}\else{#1}\fracl{#2}{#3}\fi}

      \providecommand{\eye}[1][]{\vec{\mathrm I}\ifx|#1|{}\else_{#1}\fi}%
      \providecommand{\numeye}[1][]{\boldsymbol{\mathsf{I}}\ifx|#1|{}\else_{#1}\fi}%
      \providecommand{\doteye}{{\ooalign{$\eye${\kern-0.41em\raise1.56ex\hbox{\tiny$\bullet$}}}}}%
      \providecommand{\Eye}[1]{
        \begin{bmatrix}
        \ifthenelse{#1>1}{
          \ifthenelse{#1>2}{
            \ifthenelse{#1>3}{
              \ifthenelse{#1>4}{
                1&\zeroentry&\dotso&\zeroentry
                \\
                \zeroentry&1&\dotso&\zeroentry
                \\
                \vdots&\vdots&\ddots&\vdots
                \\
                \zeroentry&\zeroentry&\dotso&1
              }{        
                1&\zeroentry&\zeroentry&\zeroentry
                \\
                \zeroentry&1&\zeroentry&\zeroentry
                \\
                \zeroentry&\zeroentry&1&\zeroentry
                \\
                \zeroentry&\zeroentry&\zeroentry&1
              }
            }{
              1&\zeroentry&\zeroentry
              \\
              \zeroentry&1&\zeroentry
              \\
              \zeroentry&\zeroentry&1
            }
          }{
            1&\zeroentry
            \\
            \zeroentry&1
          }
        }{
          1
        }
        \end{bmatrix}
      }

      \providecommand{\lebmeas}[1][]{\measure L^{#1}}     %
      \providecommand{\lebmeasof}[2][]{\ifx|#1|\left|#2\right|\else\lebmeas[#1]\qp{#2}\fi}         %
      \providecommand{\meshsize}[1][]{h\ifx|#1|\else_{#1}\fi}

      \let\oldneg\neg
      \renewcommand{\neg}[1]{\left[#1\right]_-}
      \providecommand{\dash}[1][']{\ifthenelse{\equal{#1}{'}\OR\equal{#1}{''}}{#1}{^{(#1)}}}
      
      \providecommand{\pdfrac}[2][]{\ensuremath{\frac{\partial\ifx|#1|\phantom{#2}\else{#1}\fi}{\partial{#2}}}} %
      \providecommand{\pdfracpow}[3][]{\ensuremath{\frac{\partial^{#3}\ifx|#1|\phantom{#2}\else{#1}\fi}{\partial{#2}^{#3}}}} %

      \providecommand{\pd}[2][]{\ensuremath{\partial_{#2}}{\ifx|#1|{}\else{\qb{#1}}\fi}} %

      \providecommand{\dd}[2][]{\ensuremath{\ifx|#1|\frac{\d}{\d{#2}}\else\frac[l]{\d{#1}}{\d{#2}}\fi}}    %

      \renewcommand{\Im}{\operatorname{im}}                 %
      \renewcommand{\Re}{\operatorname{re}}                 %
      \providecommand{\imaginpart}[1][]{\Im{\ifx|#1|{}\else\qp{#1}\fi}} %
      \providecommand{\realpart}[1][]{\Re{\ifx|#1|{}\else\qp{#1}\fi}} %
      \providecommand\determinant\det

      \providecommand{\transpose}{\intercal}%

      \providecommand{\transposed}{{}^\transpose}
      
      \providecommand{\Transpose}[1]{\ensuremath{{#1}^{\transpose}}}
      
      \providecommand{\Transposevec}[1]{\Transpose{\vec{#1}}}
      \providecommand{\transposevec}[1]{\Transposevec{#1}}

      \providecommand{\orthogonalto}[1][]{\ensuremath{\perp\ifx|#1|{}\else{\!_{#1}\,}\fi}}
      
      \providecommand{\rowof}[1]{\ensuremath{\mathchoice{\vecof{#1}}{\qb{#1}}{\qb{#1}}{\qb{#1}}}}

      \providecommand{\colvec}[1]{\vecof{#1}}

      \providecommand{\rowvectwo}[2]{\ensuremath{\vecof{#1,#2}}}

      \providecommand{\colvectwo}[2]{\ensuremath{%
          \mathchoice%
              {\discolvectwo{#1}{#2}}%
              {\rowvectwo{#1}{#2}\transposed}%
              {\rowvectwo{#1,#2}\transposed}%
              {\rowvectwo{#1,#2}\transposed}%
        }
      }
      \providecommand{\coltwovec}\colvectwo
      \providecommand{\colvecthree}[3]{
        \mathchoice%
            {\discolvecthree{#1}{#2}{#3}}%
            {\colvec{#1,#2,#3}}%
            {\colvec{#1,#2,#3}}%
            {\colvec{#1,#2,#3}}%
      }

      \providecommand{\colvecithree}[1]{\colvecthree{\vecentry{#1}1}{\vecentry{#1}2}{\vecentry{#1}3}}

      \providecommand{\discolvec}[2][r]{\ensuremath{\begin{bmatrix}[#1]#2\end{bmatrix}}}

      \providecommand{\discolvectwo}[3][r]{\ensuremath{\discolvec[#1]{#2\\#3}}}
      \providecommand{\discolvecthree}[4][r]{\ensuremath{\discolvec[#1]{#2\\#3\\#4}}}

      \providecommand{\discolvecitwo}[1]{\discolvectwo{\vecentry{#1}1}{\vecentry{#1}2}}

      \provideboolean{showzeroentries}%
      \setboolean{showzeroentries}{true}%
      \providecommand{\zeroentry}{\ifthenelse{\boolean{showzeroentries}}{{0}}{\phantom0}}
      \providecommand{\zeroentrywarning}{\ifthenelse{\boolean{showzeroentries}}{}{%
          \ensuremath{\text{($0$ entries omitted)\xspace}}}}

      \providecommand{\smint}{\ensuremath{{\text{\textbf{/}}}\kern-.75em\smallint}}
      \renewcommand{\smint}[1][]{\lower12.3pt\hbox{\begin{tikzpicture}\draw[line width=.75pt] (-3pt,-0.5)--(1pt,-0.5) node[pos=0.6]{$\int$};\path (3pt,-24pt)node {\scriptsize $#1$};\end{tikzpicture}}}

      \providecommand{\lap}{\ensuremath{\mathrm\Delta}}
      \providecommand{\lapin}[1][]{\lap\ifx|#1|\else_{#1}\fi}
      \providecommand{\normalsymbol}{\operatorname{\mathbf{n}}}
      \renewcommand{\normalsymbol}{\vec{\operatorname{n}}}
      \providecommand{\normal}[1][]{\normalsymbol\ifx|#1|\else_{#1}\fi}%
      \providecommand{\norm@l}[1][]{\normalsymbol\ifx|#1|\else_{#1}\fi}%
      \providecommand{\normalto}[2][]{\ensuremath{\norm@l[#2]\ifx|#1|\else\qp{#1}\fi}}
      \providecommand{\normalder}[1][]{\ensuremath{\norm@l\ifx|#1|\else\qp{#1}\fi{\inner\grad}}}

      \providecommand{\tangentialsymbol}{\operatorname{\textbf{t}}}
      \providecommand{\tangentialto}[2][]{\tangentialsymbol\ifx|#1|\else^{#1}\fi\ifx|#2|\else_{#2}\fi}

      \providecommand{\united}{\ensuremath{\cup}}

      \providecommand{\join}{\united}
      
      \providecommand{\union}[1]{\ensuremath{\bigcup\nolimits_{#1}}}

      \providecommand{\unions}[3][]{\union{#2\in{#3}\ifx|#1|\else:#1\fi}}

      \ifthenelse{\boolean{usesvjour}}{}{
        \let\vec\undefined
        \providecommand{\vec}[1]{\ensuremath{\boldsymbol{#1}}}
        \renewcommand{\vec}[1]{\ensuremath{\boldsymbol{#1}}}
      }

      \providecommand{\hatmat}[1]{\hat{\mat{#1}}}

      \providecommand{\geomat}[1]{\vec{\UCmath{#1}}}

      \providecommand{\mat}[1]{\geomat{#1}} %
      \providecommand{\Prob}[1][]{\ensuremath{\operatorname{Prob}\ifx|#1|{}\else_{#1}\fi}}

      \providecommand{\pdf}[2][]{\ensuremath{\operatorname{pdf}_{#2\ifx|#1|{}\else{\conditionalto{#1}}\fi}}\xspace}

      \providecommand{\expectation}{\ensuremath{\operatorname{E}}}
      \providecommand{\EX}[1][]{\ensuremath{\expectation\ifx|#1|{}\else_{#1}\fi}}

      \providecommand{\gausskernel}[3][x]{%
        \ensuremath{
          \exp\frac{-\if#20{#1}\else(#1-\mu)\fi^2}{%
            2\if#31{}\else\powp2{#3}\fi}%
        }%
      }
      \providecommand{\gaussdistribution}[3][x]{%
        \ensuremath{\frac1{\sqrt{2\pic}\if#31{}\else#3\fi}%
          \gausskernel[#1]{#2}{#3}
        }%
      }%

      \providecommand{\PD}[1]{\operatorname{PD}\qpreg{#1}}
      \providecommand{\pdspace}[1]{\PD{\linspace v}}
      \providecommand{\pdmats}[2][F]{\PD{\csname#1\endcsname{#2}}}
      
      \providecommand{\SPD}{\operatorname{SPD}}
      \providecommand{\spdmats}[2][F]{\SPD(\csname#1\endcsname{#2})}

       \providecommand{\Continuous}{\ensuremath{\operatorname C}\xspace}%
       \providecommand{\Hspace}{\ensuremath{\operatorname H}\xspace}
       \providecommand{\Lebesgue}{\ensuremath{\operatorname L}\xspace}
       \providecommand{\Besovspace}{\ensuremath{\operatorname B}\xspace}

       \providecommand{\Weaklyder}{\ensuremath{\operatorname W}\xspace}
       
       \providecommand{\dual}[1]{\ensuremath{{#1}'}}
       
       \providecommand{\dualspace}[2][]{\dual{\linspace{#2}\ifx|#1|\else{_{#1}}\fi}}
       \providecommand{\bidual}[1]{\ensuremath{{#1}''}}
       \providecommand{\bidualspace}[2][]{\bidual{\linspace{#2}\ifx|#1|\else{_{#1}}\fi}}

       \providecommand{\cont}[1]{\ensuremath{\Continuous^{#1}}}

       \providecommand{\diffable}[2][]{\ensuremath{\cD\ifx|#1|\else^{#1}\fi(#2)}}
       
       \providecommand{\BV}[1]{\ensuremath{\operatorname{BV}}}

       \providecommand{\leb}[1]{\ensuremath{\Lebesgue_{#1}}}

       \providecommand{\lebloc}[1]{\ensuremath{{{\Lebesgue}^{\kern-.20em\lower .1ex\hbox{\tiny\textrm{\textup{loc}}}}_{#1}}}}
       \providecommand{\lebnorm}[3][]{\ensuremath{\Norm{#2}_{\leb{#3}\ifx|#1|{}\else(#1)\fi}}}

       \providecommand{\bes}[3][]{\ensuremath{\Besovspace^{#2}_{#3\ifx|#1|\else,#1\fi}}}

       \providecommand{\sob}[2]{\ensuremath{{\smash\Weaklyder}^{#1}_{#2}}}

       \providecommand{\sobh}[1]{\ensuremath{\Hspace^{#1}}}
       \providecommand{\vecsobh}[1]{\ensuremath{\vec\Hspace^{#1}}}
       \ProvideDocumentCommand{\hdiv}{ O{} O{}}{\vecsobh{\operatorname{div}}\ifx+#1+\else_{0|#1}\fi\ifx|#2|\else(#2)\fi}
       \providecommand{\hcurl}[1][]{\vecsobh{\operatorname{curl}}\ifx|#1|\else(#1)\fi}
       
       \providecommand{\sobhz}[2][]{\sobh{#2}_{0\ifx+#1+\else|#1\fi}}

       \providecommand{\Lip}[1][]{\ensuremath{\operatorname{Lip}}\ifx|#1|{}\else{\qp{#1}}\fi}

       \ProvideDocumentCommand{\polyring}{ O{X} O{A} }{\ring[#2][#1]}
       \ProvideDocumentCommand{\polyfield}{ O{X} O{} }{\field[#2][#1]}
       \providecommand{\polyreals}[1][]{\polyfield[][R]\ifx|#1|\else^{#1}\fi}
       \providecommand{\poly}[2][]{\ensuremath{\rP\ifx#1\else_{#1}\fi^{#2}}}

       \providecommand{\Symmatrices}[2][R]{\ensuremath{\operatorname{Sym}{(\csname#1\endcsname{#2})}}}
       
       \providecommand{\SAmatrices}[2][F]{\ensuremath{\operatorname{SA}{(\csname#1\endcsname{#2})}}}
       \providecommand{\mesh}[2][]{\ensuremath{\mathcalbf{\MakeUppercase{#2}}\ifx|#1|\else_{#1}\fi}}

      \providecommand{\crouzeixraviart}[1][1]{\operatorname{CR}\ifx|#1|{}\else{^{#1}}\fi}

      \providecommand{\nodes}{\operatorname{Nodes}}
      
      \providecommand{\nodesofmesh}[2][]{\nodes\ifx|#1|\else_{#1}\fi\mesh{#2}}

      \providecommand{\linspace}[1]{\mathscript{\MakeUppercase{#1}}}

      \providecommand{\Lin}{\operatorname{Lin}}
      \providecommand{\CL}{\operatorname{CL}}
      \providecommand{\linops}[3][]{\ensuremath{\Lin\ifx|#1|\else^{#1}\fi\qp{{#2}\to{#3}}}}

      \providecommand{\clinops}[3][]{\ensuremath{\CL\ifx|#1|\else^{#1}\fi\qp{{#2}\to{#3}}}}
      
      \providecommand{\fepartition}[2][]{\mathscript{\MakeUppercase{#2}}\ifx|#1|{}\else_{#1}\fi}
      \providecommand{\fespace}[2][]{\mathbb{\MakeUppercase{#2}}\ifx|#1|{}\else_{#1}\fi}
      \providecommand{\hatfespace}[2][]{\widehat{\mathbb{\MakeUppercase{#2}}}\ifx|#1|{}\else_{#1}\fi}

      \providecommand{\vespace}[1][]{\fespace v\ifx|#1|\else_{#1}\fi}
      \providecommand{\hatvespace}[1][]{\hatfespace v\ifx|#1|\else_{#1}\fi}
      \providecommand{\fezerospace}[2][]{\ensuremath{\mathring{\fespace{#2}}\ifx|#1|{}\else_{#1}\fi}}

      \providecommand{\fe}[2][]{\ensuremath{\UCmath{#2}\ifx|#1|\else_{#1}\fi}}%

      \providecommand{\vecfe}[2][]{\ensuremath{\vec{\fe{#2}}\ifx|#1|{}\else{_{#1}}\fi}}%
      
      \providecommand{\matfe}[2][]{\ensuremath{\mat{\fe{#2}}\ifx|#1|{}\else{_{#1}}\fi}}%
      
      \providecommand{\hatmatfe}[2][]{\ensuremath{\hatmat{\UCmath{#2}}\ifx|#1|{}\else{_{#1}}\fi}}%

      \providecommand{\Foreach}[1][]{\text{\; for each \ifx|#1|\else#1\ \fi}}%

      \RequirePackage{amscd}

      \providecommand{\funk}[4][]{\ensuremath{#2:#3\ifx|#1|\else(\subseteq #1)\fi\to#4}}

      \providecommand{\isomorphicto}{\leftrightarrows}
      \providecommand\isomorphic\isomorphicto

      \providecommand{\weaklyto}{\rightharpoonup}

      \providecommand{\evalat}[3][]{\qb{#2}_{\ifx|#1|{}\else#1=\fi#3}}
      \providecommand{\evaldiff}[4][]{\qb{#2}^{\ifx|#1|{}\else#1=\fi#3}_{\ifx|#1|{}\else#1=\fi#4}}

      \providecommand{\bs}{\char '134}   %

      \providecommand{\Program}[1]{\textsf{#1}\xspace}%
      \providecommand{\Source}[1]{\nolinkurl{#1}\xspace}

      \providecommand{\matlabplot}[2][]{%
        \begin{center}
          \includegraphics[width=0.9375\linewidth,trim=64 200 64 200,clip]{#2}%
          \ifx|#1|\else\\#1\fi
        \end{center}%
      }

      \providecommand{\texcommand}[1]{\texttt{\bs{\nolinkurl{#1}}}\xspace}
      
      \providecommand{\codename}[1]{\nolinkurl{#1}\xspace}
      \providecommand\colorvar[2][a]{\colorbox{#1!6.25}{#2}}%
      \providecommand{\colorvarname}[2][a]{\colorvar[#1]{\Verb{#2}}}
      
      \providecommand{\codevarname}[1]{\colorvarname[a]{#1}}
      
      \providecommand\olco\codevarname

      \ifthenelse{\boolean{useutopia}}{%
        
      }{%
        
      }

      \providecommand{\matlab}{{\footnotesize\Program{MATLAB}}\xspace}%
      \providecommand\MATLAB\matlab

      \ProvideDocumentCommand{\codesnip}{ O{.} O{1.0} m}{%
        \newline
        \begin{minipage}{#2\linewidth}
          \lstinputlisting{#1/#3}
        \end{minipage}
      }
      \providecommand{\codeprint}[2][.]{
        \ \newline
        \begin{minipage}{\linewidth}
          \lstinputlisting{#1/#2}
          \framebox{Contents of file %
            \ifthenelse{\isundefined\pickuppath}{%
             \codename{#2}%
            }{%
              \providecommand{\fullpickuppath}{}%
              \renewcommand{\fullpickuppath}{\pickuppath/\ifx|#1|\else#1/\fi#2}%
              \href{\fullpickuppath}{\codename{#2}}%
          }}
        \end{minipage}
      }
      \providecommand{\codenoprint}[2][.]{
              \providecommand{\fullpickuppath}{}%
              \renewcommand{\fullpickuppath}{\pickuppath/\ifx|#1|\else#1/\fi#2}%
              \href{\fullpickuppath}{\codename{#2}}%
      }

      \providecommand{\indexen}[2][]{{\ifthenelse{\boolean{shownotes}}{\color b}{}#2\ifx|#1|\index{#2}\else\index{#1}\fi}}

      \providecommand{\indexma}[2][]{{\ifthenelse{\boolean{shownotes}}{\color b}{}#2\ifx|#1|\index{\(#2\)}\else\index{<#1@\(#2\)}\fi}}

      \providecommand{\ListParameters}{}
      \renewcommand{\ListParameters}%
      {
      	 \setlength{\topsep}{0pt}
      	 \setlength{\leftmargin}{0pt}
               \setlength{\itemsep}{0pt}
      	 \setlength{\parsep}{0pt}
      	 \setlength{\parskip}{0pt}
               \setlength{\labelsep}{0pt}
      	 \setlength{\itemindent}{0pt}
      }
      {%
        \begin{list}%
          {}%
          {\ListParameters%
          
      }}%
      {\end{list}}
      \newcounter{tmpcounter}
      \newcounter{LetterListItem}
      \renewcommand{\theLetterListItem}{(\alph{LetterListItem})}

      \newcounter{CapitalListItem}
      \renewcommand{\theCapitalListItem}{\Alph{CapitalListItem}.}

      \newcounter{NumberListItem}
      \renewcommand{\theNumberListItem}{\arabic{NumberListItem}}
      {
      	\begin{list}%
      	{\theNumberListItem.\ }%
      	{\usecounter{NumberListItem}%
      	 \ListParameters
      	}
      }%
      {\end{list}}
      \newcounter{QuestionListItem}
      \renewcommand{\theQuestionListItem}{\textbf{Question \arabic{QuestionListItem}}}
      {
      	\begin{list}%
      	{\theQuestionListItem.\ }%
      	{\usecounter{QuestionListItem}%
      	 \ListParameters
      	}
      }%
      {\end{list}}
      \newcounter{RomanListItem}
      \renewcommand{\theRomanListItem}{(\roman{RomanListItem})}
      {
      	\begin{list}%
      	{\theRomanListItem\ }%
      	{\usecounter{RomanListItem}
      	 \ListParameters
      	}
      }%
      {\end{list}}
      \newcounter{StepsItem}
      {
      	\begin{list}%
      	{Step \theStepsItem.\ }%
      	{\usecounter{StepsItem}%
      	 \ListParameters
      	}
      }%
      {\end{list}}
      \newcounter{CasesListItem}
      \renewcommand{\theCasesListItem}{\Alph{CasesListItem}}
      {
      	\begin{list}%
      	{\emph{Case \theCasesListItem.}\ }%
      	{\usecounter{CasesListItem}%
      	 \ListParameters
      	}
      }%
      {\end{list}}
      \newcounter{NumCasesListItem}
      \renewcommand{\theNumCasesListItem}{\arabic{NumCasesListItem}}
      {
      	\begin{list}%
      	{\emph{Case \theNumCasesListItem.}\ }%
      	{\usecounter{NumCasesListItem}%
      	 \ListParameters
      	}
      }%
      {\end{list}}
      \newcounter{QAListItem}
      \renewcommand{\theQAListItem}{Q\arabic{QAListItem}:}
      {
      	\begin{list}%
      	{\theQAListItem}%
      	{\usecounter{QAListItem}
      	 \ListParameters
      	}
      }%
      {\end{list}}

      \ifthenelse{\boolean{isthesis}}{%
        \setcounter{secnumdepth}{1}%
      }{
        \setcounter{secnumdepth}{2} %
      }
      \providecommand{\ListParameters}{}
      \renewcommand{\ListParameters}
      {
      	 \setlength{\topsep}{0em}
      	 \setlength{\leftmargin}{0em}
               \setlength{\itemsep}{0ex}
      	 \setlength{\parsep}{.5ex}
      	 \setlength{\itemindent}{\labelsep}
      	 \addtolength{\itemindent}{\labelwidth}
      }

        \providecommand{\ObsName}{Remark}%
        \providecommand{\RemName}{Remark}%
        \providecommand{\NotName}{Notation}%
        \providecommand{\BFNName}{Big~Fantastic~Note}%
        \providecommand{\DefName}{Definition}%
        \providecommand{\ExaName}{Example}%
        \providecommand{\TheName}{Theorem}%
        \providecommand{\LemName}{Lemma}%
        \providecommand{\ProName}{Proposition}%
        \providecommand{\CorName}{Corollary}%
        \providecommand{\PbmName}{Problem}%
        \providecommand{\HypName}{Hypothesis}%
        \providecommand{\AlgName}{Algorithm}%
        \providecommand{\ExeName}{Exercise}%
        \providecommand{\SolName}{Solution}%
        \providecommand{\ClaName}{Claim}%
        \providecommand{\EsyName}{Essay}%
        \providecommand{\Proofname}{Proof}%
        \providecommand{\Derivename}{Derivation}%
      
      \ifthenelse{\boolean{isthesis}}{%
        \providecommand{\Thecounter}{The}
      }{%
        \providecommand{\Thecounter}{subsection}
      }
      \newcommand{\oltikzgetxy}[3]{%
        \tikz@scan@one@point\pgfutil@firstofone#1\relax
        \edef#2{\the\pgf@x}%
        \edef#3{\the\pgf@y}%
      }
      \providecommand{\pdfformat}[1]{
         \provideboolean{pdfoutput}
         \setboolean{pdfoutput}{#1}%
        \ifthenelse{\boolean{pdfoutput}}{
          \typeout{using pdf}
\makeatletter
\usepackage{pdfsync}
          \providecommand{\graphext}{pdf}
          \renewcommand{\graphext}{pdf}
          \providecommand{\graphextex}{pdf_t}
          \renewcommand{\graphextex}{pdf_t}
        }{
          \typeout{using eps}
          \RequirePackage[dvips]{graphicx,xcolor}
          \providecommand{\graphext}{eps}
          \renewcommand{\graphext}{eps}
          \providecommand{\graphextex}{eps_t}
          \renewcommand{\graphextex}{eps_t}
        }
        \RequirePackage{epsfig}
        \RequirePackage{tikz}
        \RequirePackage{rotating}
\makeatletter
        \RequirePackage{graphicx}
        \RequirePackage{xcolor}
        \provideboolean{darkcolortheme}
        \definecolor{SussexFlint}{rgb}{.00,.19,.21}
        \definecolor{SussexGrey}{rgb}{.51,.58,.49}
        \definecolor{SussexOrange}{rgb}{.94,.29,.00}
        \definecolor{SussexYellow}{rgb}{1.00,.73,.00}
        \definecolor{SussexRed}{rgb}{.94,.01,.49}
        \definecolor{SussexPurple}{rgb}{.48,.06,.44}
        \definecolor{SussexGreen}{rgb}{.00,.58,.46}
        \definecolor{OmarGreen}{rgb}{.00,.68,.36}
        \definecolor{SussexBlue}{rgb}{.00,.58,.65}
        \definecolor{OmarBlue}{rgb}{.00,.38,.65}
        \colorlet{olgreen}{SussexGreen}
        \colorlet{olblue}{OmarBlue}
        \colorlet{a}{OmarBlue}%
        \colorlet{b}{SussexOrange}
        \colorlet{c}{SussexGreen}
        \colorlet{d}{SussexPurple}%
        \colorlet{e}{SussexRed}
        \colorlet{f}{SussexYellow}
        \colorlet{g}{white}%
        \colorlet{h}{SussexGrey}%
        \colorlet{i}{black}%
        \colorlet{j}{SussexFlint}
        \colorlet{colora}{a}
        \colorlet{colorb}{b}
        \colorlet{colorc}{c}
        \colorlet{colord}{d}
        \colorlet{colore}{e}
        \colorlet{colorf}{f}
        \colorlet{colorg}{g}
        \colorlet{colorh}{h}
        \colorlet{colori}{i}
        \colorlet{colorj}{j}
        \newcommand{\mausDarkColorTheme}{
          \colorlet{a}{SussexYellow!50!yellow}
          \colorlet{b}{SussexBlue}%
          \colorlet{c}{SussexRed!50!red}
          \colorlet{d}{SussexOrange!50!yellow}
          \colorlet{e}{SussexGreen!50!green}
          \colorlet{f}{SussexPurple!50!magenta}
          \colorlet{g}{black}%
          \colorlet{h}{SussexFlint!50!black}
          \colorlet{i}{white}%
          \colorlet{j}{SussexGrey}
        }
        \ifthenelse{\boolean{darkcolortheme}}{\mausDarkColorTheme}{}
\makeatletter
      }
      \providecommand{\solution}{\textbf{\SolName.}\xspace}

      \newcounter{phantomedinput}
      \newcounter{phantombox}
      \counterwithin*{phantombox}{phantomedinput}%
      \provideboolean{showphantoms}
      \renewcommand{\thephantombox}{\Alph{phantombox}}%
      \providecommand{\phantombox}[2][]{\stepcounter{phantombox}%
        \ensuremath{
          \boxed{%
            \ifthenelse{\boolean{showphantoms}}{#2}{\phantom{#2}}
            \texttt{\tiny\ \colorbox{i!50}{\color g\thephantombox}}
          }%
        }%
      }%

      \provideboolean{hidesolution}
      \newcommand{\consolution}[2][]{
        \ifthenelse{\boolean{hidesolution}}{#1\setboolean{showphantoms}{false}}{%
          {\setboolean{showphantoms}{true}\color{i!50}\par \small {\solution}\ #2\par\ \\[5pt]}}
      }
      \provideboolean{showmarks}
      \providecommand{\showmarks}[1]{%
        \ifthenelse{%
          \boolean{showmarks}}{%
          \!\,\marginpar{%
            \tiny [$#1$ mark\ifthenelse{\equal{#1}1}{\phantom{s}}s]}%
        }{}}%

      \newcommand{\condibreak}{\ifthenelse{\boolean{hidesolution}}{\clearpage}{}}
      
      \newcommand{\solutibreak}{\ifthenelse{\boolean{hidesolution}}{}{\clearpage}}
      \newcommand{\questionly}[1]{\ifthenelse{\boolean{hidesolution}}{#1}{}}
      \newcommand{\solutionly}[1]{\ifthenelse{\boolean{hidesolution}}{}{#1}}

       \providecommand{\qeyword}[1]{\index{#1}\ifthenelse{\boolean{shownotes}}{{\tiny\color e\colorbox{e!6.25}{#1}}}{}}
       \providecommand{\pathwordbase}{file:///Users/ol20/Sokratex}
       \providecommand{\pathword}[2][]{%
         \label{#2}%
         \ifthenelse{\boolean{shownotes}}{%
           \ \\\index{#2@\tiny\codevarname{#2}}{%
             \tiny{\color f\href{\pathwordbase/#2}{\colorvarname[f]{#2}%
           }}}\\
         }{}%
       }
       \providecommand{\targword}[2][]{%
         \label{#2}%
         \ifthenelse{\boolean{shownotes}}{%
           \index{#2@\tiny\codevarname{#2}}{\ensuremath{\tiny\color d-> \href{\pathwordbase/#2}{\colorvarname[d]{#2}\ifx|#1|\else\colorvarname[d]{[#1]}\fi}}}\\
         }{}%
       }
       \providecommand{\sourceurl}[2][]{%
         \ifthenelse{\boolean{shownotes}}{{\ \newline\tiny\colorbox{d!6.25}{\color d\texttt{source: \ifx|#1|\url{#2}\else\href{#2}{#1}\fi}}\newline}}{}}
       \providecommand{\sourcecite}[2][]{%
         \ifthenelse{\boolean{shownotes}}{{\ \\\tiny\colorbox{d!6.25}{\color{d}\texttt{source: \citet[#1]{#2}}}}}{%
       }}
       \providecommand{\conword}[2][]{\ifthenelse{\boolean{shownotes}}{#2}{#1}}
       \providecommand{\solword}[2][]{\ifthenelse{\boolean{hidesolution}}{#1}{#2}}
       \providecommand{\solghost}[1]{\ifthenelse{\boolean{showphantoms}}{#1}{\phantom{#1}}}

      \RequirePackage{lineno}
      \ifthenelse{\boolean{showchanges}}{
        \newcommand{\llabel}[1]{\hypertarget{llineno:#1}{\linelabel{#1}}}
        \newcommand{\lref}[1]{\hyperlink{llineno:#1}{\ref*{#1}}}
      }{
        \newcommand\llabel[1]{}
        \newcommand\lref[1]{}
      }
      \provideboolean{includeresponses}
      \setboolean{includeresponses}{false}
      \providecommand{\mailto}[1]{\href{mailto:#1}{\nolinkurl{#1}}}
      \provideboolean{showoldetails}
      \setboolean{showoldetails}{true}
      \providecommand{\oldetails}[2]{\ifthenelse{\boolean{showoldetails}}{#1}{#2}}
      \RequirePackage{hyphenat}
   \ifthenelse{\boolean{nohyperref}}{
     \PackageWarning{omarstyle}{package hyperref not loaded, please load manually if needed}
   }{
     \RequirePackage[obeyspaces,hyphens]{url}
     \RequirePackage[bookmarks,hypertexnames=false,debug,pdfpagelabels]{hyperref}
     \RequirePackage{bookmark}
   }

   \newtheoremstyle{plain}%
     {}%
     {}%
     {\mdseries\slshape}%
     {\parindent}%
     {\bfseries}%
     {.}%
     {.5em}%
     {}%
   
   \newtheoremstyle{note}%
     {}%
     {}%
     {}%
     {\parindent}%
     {\bfseries}%
     {.}%
     {.5em}%
     {}%
   
   \newtheoremstyle{claim}%
     {}%
     {}%
     {\mdseries\slshape}%
     {}%
     {\bfseries}%
     {}%
     {.5em}%
     {}%
   
   \newtheoremstyle{exercise}%
     {}%
     {}%
     {}%
     {}%
     {\bfseries}%
     {.}%
     {1em}%
     {}%
   
   \newtheoremstyle{break}%
     {}%
     {}%
     {}%
     {}%
     {\bfseries}%
     {.}%
     {\newline}%
     {}%
   
   \swapnumbers{
     \theoremstyle{plain}
     \ifthenelse
         {\boolean{isthesis}}
         {\newtheorem{The}{\TheName}[section]}%
         {
         }%
   {
      \theoremstyle{plain}
   
      \renewcommand{\Thecounter}{subsection}

      \newtheorem*{The*}{\TheName}
      \newtheorem*{Lem*}{\LemName}
      \newtheorem*{Pro*}{\ProName}
      \newtheorem*{Cor*}{\CorName}
      \newtheorem*{Pbm*}{\PbmName}
      \newtheorem*{Hyp*}{\HypName}
      \newtheorem*{Exe*}{\ExeName}
      \newtheorem*{Txx*}{\ExeName} %
      \newtheorem*{Con*}{Conclusion}
      \newtheorem*{Sum*}{Summary}
    }
    {
      \theoremstyle{claim}

    }
    {
      \theoremstyle{note}

      \newtheorem*{Obs*}{\ObsName}

      \newtheorem*{Def*}{\DefName}
      \newtheorem*{Exa*}{\ExaName}
      \newtheorem*{Alg*}{\AlgName}
    }
   
    {
      \theoremstyle{break}
    }
   }

   \ifthenelse{\boolean{usecleveref}}{
     \usepackage{aliascnt}
     \usepackage{cleveref}
     \providecommand{\@oltocline}{tocline}
     \newtheorem{OLThe}[subsection]{Theorem}%
     \crefname{OLThe}{theorem}{theorems}
     \newenvironment{The}[1][]{%
       \begin{OLThe}[#1]%
         \addcontentsline{toc}{subsection}{\thesubsection. \TheName\ifx|#1|\else\ - #1\fi}
     }{\end{OLThe}}
     \newaliascnt{Lem}{subsection}%
     \newtheorem{OLLem}[Lem]{Lemma}
     \crefname{OLLem}{lemma}{lemmata}
     \newenvironment{Lem}[1][]{%
       \begin{OLLem}[#1]%
         \addcontentsline{toc}{subsection}{\thesubsection. \LemName\ifx|#1|\else\ - #1\fi}
     }{\end{OLLem}}  
     \newaliascnt{Pro}{subsection}
     \newtheorem{OLPro}[Pro]{Proposition}
     \crefname{OLPro}{proposition}{propositions}
     \newenvironment{Pro}[1][]{%
       \begin{OLPro}[#1]%
         \addcontentsline{toc}{subsection}{\thesubsection. \ProName\ifx|#1|\else\ - #1\fi}
     }{\end{OLPro}}  
     \newaliascnt{Cor}{subsection}
     \newtheorem{OLCor}[Cor]{Corollary}
     \crefname{OLCor}{corollary}{corollaries}
     \newenvironment{Cor}[1][]{%
       \begin{OLCor}[#1]%
         \addcontentsline{toc}{subsection}{\thesubsection. \CorName\ifx|#1|\else\ - #1\fi}
     }{\end{OLCor}}  
     \newaliascnt{Def}{subsection}
     \theoremstyle{note}
     \newtheorem{OLDef}[Def]{Definition}
     \crefname{OLDef}{definition}{definitions}
     \newenvironment{Def}[1][]{%
       \begin{OLDef}[#1]%
         \addcontentsline{toc}{subsection}{\thesubsection. \DefName\ifx|#1|\else\ of #1\fi}
     }{\end{OLDef}}
     \newaliascnt{Obs}{subsection}
     \theoremstyle{note}
     \newtheorem{OLObs}[Obs]{Remark}
     \crefname{OLObs}{remark}{remarks}
     \newenvironment{Obs}[1][]{%
       \begin{OLObs}[#1]%
         \addcontentsline{toc}{subsection}{\thesubsection. \ObsName\ifx|#1|\else\ on #1\fi}
     }{\end{OLObs}}
     \newaliascnt{Sol}{subsection}
     \theoremstyle{note}
     \newtheorem{OLSol}[Sol]{Solution}
     \crefname{OLSol}{solution}{solutions}
     \newenvironment{Sol}[1][]{%
       \begin{OLSol}[#1]%
         \addcontentsline{toc}{subsection}{\thesubsection. \SolName\ifx|#1|\else\ on #1\fi}
     }{\end{OLSol}}
     \newaliascnt{Txx}{subsection}
     \theoremstyle{note}
     \newtheorem{OLTxx}[Txx]{\ExeName}
     \crefname{OLTxx}{exercise}{exercises}
     \newenvironment{Txx}[1][]{%
       \begin{OLTxx}[#1]%
         \addcontentsline{toc}{subsection}{\thesubsection. \ExeName\ifx|#1|\else\ - #1\fi}
     }{\end{OLTxx}}
     \newaliascnt{Exa}{subsection}
     \theoremstyle{note}
     \newtheorem{OLExa}[Exa]{\ExaName}
     \crefname{OLExa}{example}{examples}
     \newenvironment{Exa}[1][]{%
       \begin{OLExa}[#1]%
         \ifx|#1|%
         \addcontentsline{toc}{subsection}{\thesubsection. \ExaName}%
         \else%
         \addcontentsline{toc}{subsection}{\thesubsection. \ExaName\ of #1}%
         \fi
     }{\end{OLExa}}
     \newaliascnt{Alg}{subsection}
     \theoremstyle{note}
     \newtheorem{OLAlg}[Alg]{\AlgName}
     \crefname{OLExe}{exercise}{exercises}
     \newenvironment{Alg}[1][]{%
       \begin{OLAlg}[#1]%
         \ifx|#1|%
         \addcontentsline{toc}{subsection}{\thesubsection. \AlgName}%
         \else%
         \addcontentsline{toc}{subsection}{\thesubsection. \AlgName: #1}%
         \fi
     }{\end{OLAlg}}
     \newaliascnt{Pbm}{subsection}
     \theoremstyle{note}
     \newtheorem{OLPbm}[Pbm]{\PbmName}
     \crefname{OLPbm}{problem}{problems}
     \newenvironment{Pbm}[1][]{%
       \begin{OLPbm}[#1]%
         \addcontentsline{toc}{subsection}{\thesubsection. \PbmName\ifx|#1|\else\ - #1\fi}
     }{\end{OLPbm}}
     \aliascntresetthe{Lem}
     \aliascntresetthe{Pro}
     \aliascntresetthe{Def}
     \aliascntresetthe{Obs}
     \aliascntresetthe{Sol}
     \aliascntresetthe{Txx}
     \aliascntresetthe{Exa}
     \aliascntresetthe{Pbm}
     \crefname{Lem}{lemma}{lemmata}
     \crefname{Pro}{proposition}{propositions}
     \crefname{Def}{definition}{definitions}
     \crefname{Obs}{remark}{remarks}
     \crefname{Sol}{solution}{solutions}
     \crefname{Txx}{exercise}{exercises}
     \crefname{Exa}{example}{examples}
     \crefname{Pbm}{problem}{problems}
   }{
     \newenvironment{The}[1][]{%
       \ifx&#1&%
       \subsection{\TheName\xspace}%
       \else%
       \subsection[\MakeUppercase#1 theorem]{\TheName\ (#1)}%
       \fi%
       \slshape}{%
       \upshape}
     \newenvironment{Pro}[1][]{\subsection{\ProName\xspace{\ifx&#1&{}\else{ (#1)}\fi}}\slshape}{\upshape}
     \newenvironment{Lem}[1][]{\subsection{\LemName\xspace{\ifx&#1&{}\else{ (#1)}\fi}}\slshape}{\upshape}
     \newenvironment{Cor}[1][]{\subsection{\CorName\xspace{\ifx&#1&{}\else{ (#1)}\fi}}\slshape}{\upshape}
     \newenvironment{Txx}[1][]{\subsection{\ExeName\xspace{\ifx&#1&{}\else{ (#1)}\fi}}\slshape}{\upshape}
     \newenvironment{Pbm}[1][]{\subsection{\PbmName\xspace{\ifx&#1&{}\else{ (#1)}\fi}}\slshape}{\upshape}
     \newenvironment{Def}[1][]{\subsection{\DefName\xspace{\ifx&#1&{}\else{ of \indexen{#1}}\fi}}}{}
     \newenvironment{Obs}[1][]{\subsection{\ObsName\xspace{\ifx&#1&{}\else{ (#1)}\fi}}}{}
     \newenvironment{Exa}[1][]{\subsection{\ExaName\xspace{\ifx&#1&{}\else{ (#1)}\fi}}}{}
     \newenvironment{Alg}[1][]{\subsection{\AlgName\xspace{\ifx&#1&{}\else{ (#1)}\fi}}}{}
   }
   \providecommand{\qed}{\vrule height 5pt depth 0pt width 3pt}
   \providecommand{\qqed}{{\raggedright{\ \hfill\qed}}}
   
   \newcounter{passo}

   \newenvironment{Proof}[1][]%
   {\par\noindent{\bf \Proofname\ifx|#1|.\ \else\ #1.\ \fi}\setcounter{passo}{0}}%
   {\qqed\par}
   {\par\noindent{\bf \Derivename\ #1}\setcounter{passo}{0}}%
   {\qqed\par}
   \newenvironment{Proof*}[1][{}]%
   {\subsection{\Proofname\ #1}\setcounter{passo}{0}}
   {\qqed\par}

\usepackage[utf8]{inputenc}
\usepackage[T1]{fontenc}
\usepackage{CJK}
\usepackage{graphicx}
\usepackage{grffile}
\usepackage{longtable}
\usepackage{wrapfig}
\usepackage{rotating}
\usepackage[normalem]{ulem}
\usepackage{amsmath,stackengine}
\usepackage{amsthm}
\usepackage{textcomp}
\usepackage{amssymb}
\usepackage{capt-of}
\usepackage{natbib}
\makeatletter
\@ifpackageloaded{biblatex}{
  \newbibmacro*{bbx:parunit}{%
    \ifbibliography
        {\setunit{\bibpagerefpunct}\newblock
          \usebibmacro{pageref}%
          \clearlist{pageref}%
          \setunit{\adddot\newline\nobreak}}
        {}}
  
  \renewbibmacro*{doi+eprint+url}{%
    \usebibmacro{bbx:parunit}%
    \iftoggle{bbx:doi}
             {\printfield{doi}}
             {}%
             \iftoggle{bbx:eprint}
                      {\usebibmacro{eprint}}
                      {}%
                      \iftoggle{bbx:url}
                               {\usebibmacro{url}}
                               {}}
  
  \renewbibmacro*{eprint}{%
    \usebibmacro{bbx:parunit}%
    \iffieldundef{eprinttype}
                 {\printfield{eprint}}
                 {\printfield[eprint:\strfield{eprinttype}]{eprint}}}
  
  \renewbibmacro*{url+urldate}{%
    \usebibmacro{bbx:parunit}%
    \printfield{url}%
    \iffieldundef{urlyear}
                 {}
                 {\setunit*{\addspace}%
                   \printtext[urldate]{\printurldate}}}
  
  \renewbibmacro*{url}{%
    \usebibmacro{bbx:parunit}%
    \iffieldundef{doi}{%
      \printfield{url}%
    }
  }

}{
  \RequirePackage{doi}
}
\usepackage{tikz}
\usepackage{pdfpages}
\usepackage{chemformula}
\usepackage{subcaption}
\usepackage{float}

\usepackage{chemformula}
\usepackage{subcaption}
\usepackage{float}
\usepackage{dutchcal}
\numberwithin{equation}{section}
\usepackage{cleveref}

\newcommand{\normalom}{\mathbf{n}}
\renewcommand\ltwop[3][]{\ensuremath{\qp{#2,#3}\ifx|#1|\else_{#1}\fi}}

\providecommand{\normalsymbol}{}
\renewcommand{\normalsymbol}{n}

\providecommand{\mmmathcal}{\mathcal{H}}
\providecommand{\mcf}{{\mathcal{F}}(c_\tau^-(t),u_\tau^-(t),u_\tau^+(t))}
\providecommand{\mcp}{(m(c^{k-1}))^+}
\providecommand{\mcm}{(m(c^{k-1}))^-}
\providecommand{\mcppwt}{(m(c_\tau^-(t)))^+}
\providecommand{\mcpp}{(m(c_\tau^-))^+}

\providecommand{\mc}{m(c_\tau^-)}

\providecommand{\upt}{u_\tau^{+}}
\providecommand{\umt}{u_\tau^{-}}
\providecommand{\cmt}{c_\tau^{-}}
\providecommand{\upwt}{u_\tau^{+}(t)}
\providecommand{\umwt}{u_\tau^{-}(t)}

\providecommand{\hop}{{\textsl{H}}_1(u^k,u^{k-1})}
\providecommand{\htp}{{\textsl{H}}_2(u^k,u^{k-1})}

\providecommand{\hopwo}{{\textsl{H}}_1(w_1,u^{k-1})}
\providecommand{\htpwo}{{\textsl{H}}_2(w_1,u^{k-1})}
\providecommand{\hopwt}{{\textsl{H}}_1(w_2,u^{k-1})}
\providecommand{\htpwt}{{\textsl{H}}_2(w_2,u^{k-1})}
\providecommand{\hopw}{{\textsl{H}}_1(w,u^{k-1})}
\providecommand{\htpw}{{\textsl{H}}_2(w,u^{k-1})}

\providecommand{\how}{h_1(w,u^{k-1})}
\providecommand{\htw}{h_2(w,u^{k-1})}
\providecommand{\hox}{h_1(\xi,u^{k-1})}
\providecommand{\htx}{h_2(\xi,u^{k-1})}
\providecommand{\hopa}{30u^{k-1}u^k(u^k-1)^2}
\providecommand{\htpa}{30(u^{k})^2(u^k-1)(u^{k-1}-1)}
\providecommand{\hoe}{h_1}
\providecommand{\hte}{h_2}
\providecommand{\hoa}{15r\Big(\frac{(s-1)^4}{4} + \frac{(s-1)^3}{3}\Big) + \frac{5r}{4}}
\providecommand{\hta}{15(r-1)\Big(\frac{s^4}{4} + \frac{s^3}{3}\Big) - \frac{5r}{4}}
\providecommand{\hope}{\textsl{H}_1}
\providecommand{\htpe}{\textsl{H}_2}

\newenvironment{theorem}[1][]{\begin{The}[#1]}{\end{The}}
\newenvironment{lemma}[1][]{\begin{Lem}[#1]}{\end{Lem}}

\newenvironment{remark}[1][]{\begin{Obs}[#1]}{\end{Obs}}

\providecommand{\normalsymbol}{}
\renewcommand{\normalsymbol}{n}

\providecommand{\mcf}{{\mathcal{F}}(c_\tau^-(t),u_\tau^-(t),u_\tau^+(t))}
\providecommand{\mcp}{(m(c^{k-1}))^+}
\providecommand{\mcm}{(m(c^{k-1}))^-}
\providecommand{\mcppwt}{(m(c_\tau^-(t)))^+}
\providecommand{\mcpp}{(m(c_\tau^-))^+}

\providecommand{\mc}{m(c_\tau^-)}

\providecommand{\upt}{u_\tau^{+}}
\providecommand{\umt}{u_\tau^{-}}
\providecommand{\cmt}{c_\tau^{-}}
\providecommand{\upwt}{u_\tau^{+}(t)}
\providecommand{\umwt}{u_\tau^{-}(t)}

\providecommand{\hopa}{30u^{k-1}u^k(u^k-1)^2}
\providecommand{\htpa}{30(u^{k})^2(u^k-1)(u^{k-1}-1)}

\providecommand{\cstar}{C^{\ast}}
\providecommand{\cstartwo}{C^\star}

\author{Alexandros Skouras}
\address{Institut Montpelliérain Alexander Grothendieck, Université de Montpellier, France}
\email{alskouras@gmail.com}
\author{Omar Lakkis}
\address{Department of Mathematics, University of Sussex, Brighton UK}
\email{lakkis.o.math@gmail.com}
\author{Vanessa Styles}
\address{Department of Mathematics, University of Sussex, Brighton UK}
\email{v.styles@sussex.ac.uk}
\title[PDE models of lithium batteries]{%
  Existence of solution to a system of PDEs modeling the crystal
  growth inside lithium batteries}
\date{\today}
\hypersetup{
  pdfauthor={Alexandros Skouras},
  pdftitle={Existence of solution to a system of PDEs modeling the crystal growth inside lithium batteries},
  pdfkeywords={},
  pdfsubject={},
  pdfcreator={Emacs 27.2 (Org mode 9.4.4)}, 
  pdflang={English}}
\setboolean{showchanges}{false}%
\setboolean{shownotes}{false}%
\renewcommand{\normalsymbol}{\operatorname{\mathbf{n}}}
\makeatother
\begin{document}
\maketitle
\begin{abstract}
  We study a model for lithium (Li) electrodeposition on Li-metal electrodes that leads to dendritic pattern formation. The model comprises of a system of three coupled PDEs, taking the form of an Allen--Cahn equation, a Nernst--Planck equation and a Poisson equation. We prove existence of a weak solution and stability
  results for this system and present numerical simulations resulting from a finite element approximation of the system, which illustrate the dendritic nature of solutions to the model.. 
\end{abstract}

\section{Introduction}
\label{sec1}
As humanity aims at reaching ``net-zero'' by 2050, replacing fossil
fuels with renewable alternatives has spurred research in storage
technologies, with lithium batteries being one of the favoured
technologies.  Electrodeposition in lithium batteries is a phenomenon
that plays an important role in studying lithium batteries, in that it
may lead to solid dendritic structures which, whilst remaining
attached to the electrode, grow into the electrolyte solution. Over
time these structures can seriously compromise the performance, the
lifespan as well as the safety of such batteries, see for example
\citet{okajimaPhasefieldModelElectrode2010%
  ,akolkarMathematicalModelDendritic2013%
  ,liangNonlinearPhaseField2014%
  ,chenModulationDendriticPatterns2015%
  ,yurkivPhasefieldModelingSolid2018}
and the references therein.

In this paper we analyse a phase-field model for lithium (Li) electrodeposition on
Li-metal electrodes in high-capacity lithium-oxygen and
lithium-sulphur batteries. The model we consider is a variant of the
thermodynamically consistent model introduced in
\citet{chenModulationDendriticPatterns2015} to predict dendritic
patterns which form during an electrochemical process using
Li-electrodeposition. For ease of presentation we consider a
simplified version in which, where possible, physical parameters have
been set to $1$;
\changes{namely, we look for functions $u,c,\phi$
  of $(x,t)=(x_1,\dotsc,x_d,t)\in\W\times T, d=2,3$
  such that }
\begin{equation}
  \begin{aligned}\label{m1}
    \partial_t u = \nabla\cdot [\mat A(\nabla u)\nabla u] - g'(u) + m(c) h'(u) ~~~ \text{ in } \Omega\times [0,T],
  \end{aligned}
\end{equation}
\begin{equation}
  \begin{aligned}\label{m2}
    \partial_t c = \nabla\cdot \big [D(u)\nabla c + D_1(u,c)\nabla\phi \big ] - \partial_t u ~~~ \text{ in } \Omega\times [0,T],  
  \end{aligned}
\end{equation}
\begin{equation}
  \begin{aligned}\label{m3}
    \nabla\cdot \big [\sigma(u)\nabla\phi \big ]
    =  \partial_t u  -
    \sigma'(u)\changes{\partial_{x_1} u}  ~~~
    \text{ in } \Omega\times [0,T], 
  \end{aligned}
\end{equation}
together with appropriate boundary and initial data.  Here $c$
represents the ion concentration of \ch{Li+}, $\phi$ denotes the
electrostatic potential and $u$ is a phase field variable that has a
physical correspondence to the concentration of Li atoms. The specific
forms of the anisotropic tensor $\mat A(u)$ and the functions $g(u)$,
$m(c)$, etc. are defined in Section \ref{sec:NPPAC}.

An early application of phase-field modelling to electrochemistry was
introduced in \citet{guyerModelElectrochemicalDouble2002} in which a
new model was proposed and derived using a free energy functional that
includes the electrostatic effect of charged particles leading to rich
interactions between concentration, electrostatic potential and phase
stability. This model was further studied by the same authors in
\citet{guyerPhaseFieldModeling2004a,
  guyerPhaseFieldModeling2004}. There is extensive literature on
extensions of the model proposed in
\citet{guyerModelElectrochemicalDouble2002} to models taking similar
forms to \eqref{m1}--\eqref{m3}, namely a Nernst--Planck--Poisson
system coupled with an anisotropic phase-field equation,
e.g. \citet{okajimaPhasefieldModelElectrode2010,
  liangNonlinearPhasefieldModel2012,akolkarMathematicalModelDendritic2013,
  chenModulationDendriticPatterns2015,
  yurkivPhasefieldModelingSolid2018,
  muNumericalSimulationFactors2019}. In these works numerical
computations are presented that simulate the evolution of dendritic
structures in \ch{Li}-ion or \ch{Li}-metal batteries.  Works
presenting numerical simulations of anisotropic phase field models go
back many decades, see for example
\citet{kobayashiNumericalApproachThreeDimensional1994%
  ,karmaQuantitativePhasefieldModeling1998%
  ,mcfaddenPhasefieldModelsAnisotropic1993%
  ,wheelerXivectorFormulationAnisotropic1996%
  ,graserTimeDiscretizationsAnisotropic2013}.%
While mathematical analysis on the anisotropic phase-field model was
developed in
\citet{elliottLimitAnisotropicDoubleobstacle1996%
  ,elliottLimitFullyAnisotropic1997%
  ,taylorDiffuseInterfacesSharp1998,graserTimeDiscretizationsAnisotropic2013}.

In this paper we adapt techniques from
\citet{burmanExistenceSolutionsAnisotropic2003} in which the authors
establish an existence result for weak solutions to a simplified
version of the model (\ref{m1})--(\ref{m3}). Specifically, the model
studied in \citet{burmanExistenceSolutionsAnisotropic2003}, which does
not include an equation of type \eqref{m3}, is a system of two
equations that take similar forms to \eqref{m1} and \eqref{m2}. There
are two key differences between the forms of \eqref{m1} and \eqref{m2}
in this paper and their respective forms in
\citet{burmanExistenceSolutionsAnisotropic2003}, firstly, in this work
the $\partial_t u$ term in \eqref{m2} is not included in
\citet{burmanExistenceSolutionsAnisotropic2003} and secondly, the
$\grad \phi$ advection term in \eqref{m2} in this work is replaced by
a $\grad u$ advection term in
\citet{burmanExistenceSolutionsAnisotropic2003}.

The remainder of the paper is organised as follows. In Section
\ref{sec2} we present the full model together with a weak formulation,
then we present the main result, which is an existence result, using
Rothe's method, for solutions to the weak formulation.  We prove the
existence result by considering, in Section \ref{sec3}, a time
discretization of the weak formulation and, in Section \ref{sec4},
taking the limit as the time step $\tau$, of the dicretization, tends
to zero.  We conclude with Section \ref{sec5} in which we present a
finite element discretisation of the model together with some
numerical examples that illustrate the dendritic nature of solutions
to the model.

We conclude this section with some comments on notation. We denote the
$\leb 2$-inner product on $\Omega$ by $(\cdot,\cdot)$ and adopt the
standard notation for Sobolev spaces, denoting the norm of $\sob lp
(\Omega)$ $(l\in\mathbb{N},p\in[1,\infty])$ by $\|\cdot\|_{l,p}$ and
the semi-norm by$|\cdot |_{l,p}$.  For $p=2$, $\sob l 2(\Omega)$ will
be denoted by $\sobh l(\Omega)$ with the associated norm and semi-norm
written, as respectively, $\|\cdot\|_{l}$ and $|\cdot |_{l}$. We
denote the dual space of $\sobh 1(\Omega)$ by $(\sobh 1(\Omega))'$. In
addition, we adopt the standard notation $\sob lp(a,b;X), (l\in
\mathbb{N}, p\in[1,\infty], (a,b)$ an interval in $\mathbb{R}$, $X$ a
Banach space) for time dependent spaces with norm $\|\cdot\|_{\sob
  lp(a,b;X)}$. Once again, we write $\sobh l(a,b;X)$ if $p=2$. For a
Hilbert Sobolev space with Dirichlet boundary condition on
$\Sigma\subseteq\partial Q$ we write $\sobhz[\Sigma] m(Q)$ and $\sobhz
m(Q)$ if $\Sigma=\partial Q$.  Lastly, $C$ denotes a generic constant
that may change from line to line.

\section{Weak formulation of the model}
\label{sec2}
\subsection{The model}%
\label{sec:NPPAC}
The model, \eqref{m1}--\eqref{m3}, takes the form of three interacting PDEs:
\begin{enumerate}[(a)\ ]
\item[\eqref{m1}]
  a nonlinear anisotropic {Allen--Cahn} type equation, for a 
  {phase-field} variable $u$, with a forcing term, $m(c)h'(u)$, that accounts
  for the Butler--Volmer electrochemical reaction kinetics, which
  model the concentration of lithium atoms;
\item[\eqref{m2}]
  a reaction--diffusion {Nernst--Planck} type equation, which
  describes the dynamics of the concentration, $c$, of  lithium ions;
\item[\eqref{m3}]
  a {Poisson} type equation that describes the electro static 
  potential, $\phi$, that drives the dynamics of the lithium ions.
\end{enumerate}
For completeness we now rewrite the model together with the boundary and initial data:
\begin{equation}
  \begin{aligned}\label{orderparameterrescaled}
    \partial_t u = \nabla\cdot [\mat A(\nabla u)\nabla u] - g'(u) + m(c) h'(u) ~~~ \text{ in } \Omega\times [0,T],
  \end{aligned}
\end{equation}
\begin{equation}
  \begin{aligned}\label{concentrationrescaled}
    \partial_t c = \nabla\cdot \big [D(u)\nabla c + D_1(u,c)\nabla\phi \big ] - \partial_t u ~~~ \text{ in } \Omega\times [0,T],  
  \end{aligned}
\end{equation}
\begin{equation}
  \begin{aligned}\label{electricpotentialrescaled}
    \nabla\cdot \big [\sigma(u)\nabla\phi \big ]
    =  \partial_t u  -
    \sigma'(u)\changes{\partial_{x_1} u}  ~~~
    \text{ in } \Omega\times [0,T], 
  \end{aligned}
\end{equation}
\begin{equation}\label{boundary_data}
  \mat A(\nabla u)\nabla u\cdot \normalom= 0,~~ D(u)\nabla c \cdot \normalom= D_1(u,c)\grad\phi \cdot \normalom~~~ \text{ on } \partial\Omega\times [0,T],
\end{equation}
\begin{equation}\label{boundary_data+phi}
  \sigma (u)\nabla\phi \cdot \normalom=  0  \text{ on } \Gamma_3\times [0,T],  ~  \phi = \phi_-
  \text{ on } \Gamma_1\times [0,T],~
  \phi = 0   \text{ on } \Gamma_2\times [0,T], 
\end{equation}
\begin{equation}\label{initial_data}
  u(0) = u_0,~~c(0) = c_0 \text{ in } \Omega.
\end{equation}
Here $[0,T]$ is the time interval, the domain $\Omega = [0,L_1]\times [0,L_2]^{d-1}\subset\mathbb{R}^d$, $d=2,3$,  is bounded and convex with Lipschitz boundary $\partial\Omega = \Gamma_1\cup\Gamma_2\cup\Gamma_3$, such that
\begin{gather}
  \Gamma_1:=\setofsuch{\vec{x}\in\partial\Omega}{x_1= 0},~
  \Gamma_2:=\setofsuch{\vec{x}\in\partial\Omega}{x_1 = L_1}
  \tand
  \Gamma_3 =
  \partial\Omega\take(\Gamma_1\join\Gamma_2),
\end{gather}

\begin{equation}\label{doublewell}
  g(u) :=
  u^2(1-u)^2/4, ~h(u) := u^3(6u^2-15u+10),
\end{equation}
\begin{equation}\label{Dsigma}
  D(u) := D^\text{e}h(u) + D^\text{s}(1-h(u)),~~  \sigma(u) =  \sigma^\text{e}h(u)  +
  \sigma^\text{s}(1-h(u)),~
  \text{ for }  u\in\mathbb{R},
\end{equation}
\begin{equation}\label{diffusionoftransport}
  m(c) := \left\{ \begin{array}{ll}
    C_1 - C_2,& c > 1,\\
    C_1 - cC_2,& 0\leq c\leq 1,\\
    C_1,& c < 0,
  \end{array}
  \right. 
  ~~D_1(u,c) := \left\{
  \begin{array}{ll}
    D(u), & c > 1, \\
    D(u)c, & 0\leq c\leq 1,\\
    0, & c< 0,
  \end{array}
  \right.
\end{equation}
with $\D^e,D^s,\sigma^e,\sigma^s\in \mathbb{R}$ and
$0<D^{\text{s}}<D^{\text{e}}$,
$0<\sigma^{\text{e}}<\sigma^{\text{s}}$, \changes{and given constants
  $C_1,C_2 > 0$}.  The anisotropic diffusion tensor-valued field $\mat
A(\vec{p})$ of $\vec{p}\in\mathbb{R}^d\takeset{\vec0}$ is defined in
Section \ref{sec:anisotropictensor}.

\changes{Note that the term
  $\sigma'(u)\partial_{x_1}u$ in~(\ref{electricpotentialrescaled})
  appears from incorporating the boundary conditions for the potential
  $\phi$ in the equation.  Keeping this term in (\ref{electricpotentialrescaled}) allows for
  simpler boundary conditions in the analysis, but in the numerical
  simulations, see Section \ref{sec5}, it will be added to the boundary
  conditions.  }

\begin{remark}[cut-off of constitutive functions]
  In \citet{chenModulationDendriticPatterns2015} the constitutive functions
  $D_1(u,c)$ and $m(c)$ are defined, for all $u,c\in \mathbb{R}$, to
  be $D_1(u,c)=D(u)c$ and $m(c) = C_1-cC_2$ respectively.  However to
  establish the analytical results in this paper, we consider their
  cut off variants defined in \ref{diffusionoftransport}.
\end{remark}

\subsection{The weak formulation}\label{sec:weakformulation}

A weak formulation of  \eqref{orderparameterrescaled}-\eqref{initial_data} is given as follows: 

Let $u(0) = u_0 \in \sobh 1(\Omega)$ and $c(0) = c_0 \in \leb 2(\Omega)$. Then, find $(u,c,\phi)$ 

such that for all $(\chi, \eta)\in \sobh 1(\Omega)\times \sobhz[\Gamma_1\cup\Gamma_2] 1(\Omega)$, 
\begin{subequations}
  \begin{align}
    \label{orderparameterfinalform}
    (\partial_t u,\chi) + (\mat A(\nabla u)\nabla u,\nabla \chi) +( g'(u),\chi) = ( m(c) h'(u),\chi),\\
    \label{concentrationfinalform}
    (\partial_t c,\chi) + (D(u)\nabla c,\nabla \chi) + (D_1(u,c)\nabla\phi,\nabla \chi = (\partial_t u,\chi),\\
    \label{electricpotentialfinalform}
    (\sigma(u)\nabla\phi,\nabla \eta) = ( - \partial_t u + \sigma'(u)\changes{\partial_{x_1} u},\eta),
  \end{align}
\end{subequations}
where 
\begin{equation}\label{primes}
  g'(s) = 
  s(s-1)(s-\frac{1}{2}),~
  h'(s) = 
  30s^2(s-1)^2
\end{equation}
and
\begin{equation}\label{sigmaprime}
  \sigma'(s) = \
  30(\sigma_e-\sigma_s)s^2(s-1)^2.
\end{equation}

\subsection{Anisotropic diffusion tensor}\label{sec:anisotropictensor}
We define the tensor-valued field $\mat A(\vec{p})$, $\vec{p}\in\mathbb{R}^d\takeset{\vec0}$ in \eqref{m1}, which models the anisotropy of the growing crystals due to the cubic crystalline structure of lithium, by
\begin{multline}\label{anisotropy_tensor}
  \mat A(\vec p) = a(\vec{p})
  \nabla a(\vec{p})\transposevec{p}
  +
  a^2(\vec{p})\eye
  \\
  =16\, \delta \, a(\vec p)\normpow{\vec p}{-6}
  \sqmatskelijn{p_i p_j({p_i}^2\normpow{\vec p}2-\sum_{k=1}^d {p_k}^4)}ijd
  + a^2(\vec{p})\eye,
\end{multline}
where 
\begin{equation}\label{anisotropy_function}
  a(\vec{p}) = a_0 (1-3\delta) \qp{1 + \frac{4\delta}{1-3\delta}
    \frac{\sum_{j=1}^d\ppow{\vecentry pj}4}{\normpow{\vec p}4}},
\end{equation} 
for some material dependent parameters $a_0, \delta > 0$. Here $\mat
A(\vec{p})$, $\vec{p}\in\mathbb{R}^d\takeset{\vec0}$ arises from the
(Fréchet) derivative, $J_a'$, of the anisotropic Dirichlet energy,
\begin{equation}
  \label{anisotropicdirichletenergy}
  J_a(w)
  :=
  \frac12\int_{\Omega}
  a^2(\nabla w)|\nabla w|^2
\end{equation} 
such that 
$J_a': \sobh 1(\Omega)\to(\sobh 1(\Omega))', ~\forall v\in\sobh 1(\Omega)$, satisfies 
\begin{equation}\label{anisotropicdirichletenergyderivative}
  J_a'(w)[v] = \int_{\Omega} \mat A(\nabla w)\nabla w\cdot \nabla v,~~\forall v\in\sobh 1(\Omega).
\end{equation}

\begin{lemma}[convexity of the anisotropic Dirichlet energy]
  \label{lem:hessianpositive}
  Recalling \eqref{anisotropy_function}, for $\delta < 1/15$ , the
  functional
  \begin{equation}
    j_a(\vec p):= \frac12\int_{\Omega} a^2(\vec p)|\vec p|^2
  \end{equation}
  is strictly convex in $\vec p$ for all $\vec
  p\in (\leb 2 (\Omega))^d$.
  It follows that the anisotropic Dirichlet energy
  \(
  J_a(w)
  \) defined in (\ref{anisotropicdirichletenergy})
  is strictly convex in $w\in\sobh1(\Omega)$.
\end{lemma}
\begin{proof}
  The result is proved in
  \citet{burmanExistenceSolutionsAnisotropic2003} for $d=2$. Here we
  extend it to the case $d=3$.  A sufficient condition for a
  functional to be strictly convex is the positivity of the Hessian
  matrix, \citet{dacorognaDirectMethodsCalculus2008}, i.e. we require
  \begin{equation*}
    \mmmathcal(\vec p) = \nabla^2 \bar{a}^2(\vec p) > 0,~~\forall \vec p \in \operatorname{Dom}\bar{a}^2,
  \end{equation*}
  where $\bar{a}:\mathbb{R}^d\to\mathbb{R}$ is given by
  \begin{equation}\label{anisotropy_function777}
    \bar{a}(\vec p) := \frac{\sqrt{2} a(\vec p)}{2}|\vec p|
  \end{equation} 
  such that 
  \begin{equation}\label{anisotropicdirichletenergy777}
    j_a(\vec p)
    =
    \frac12\int_{\Omega}
    a^2(\vec p)|\vec p|^2
    =
    \int_{\Omega}
    \bar{a}^2(\vec p).
  \end{equation} 
  By Sylvester's criterion, a Hessian matrix is positive definite if
  all of the principal minors are positive.  By symmetry in
  $\colvecithree p$, it is sufficient to check that
  $\mmmathcal_{11}(\vec p)$, the upper $2\times 2$ matrix $\mmmathcal_{1:2,1:2}(\vec p)$ and $\mmmathcal(\vec p)$ itself have a positive
  determinant.
  
  With the use of software Wolfram Mathematica \citep{Mathematica}, our calculations show that $\mmmathcal_{11}$ is positive if and only if
  \begin{equation}
    \begin{aligned}
      1 + 10\delta - 87\delta^2 > 0, ~~1 + 6\delta - 27\delta^2 > 0, ~~ 1 - 10\delta + 21\delta^2 > 0,\\
      1 - 14\delta - 15\delta^2 > 0, ~~3 - 26\delta + 3\delta^2 > 0, ~~ 5 + 22\delta - 15\delta^2 > 0.
    \end{aligned}
  \end{equation}
  The above inequalities hold for $\delta < 1/15$, while the
  determinant of the upper $2\times 2$ matrix is positive if and only
  if
  \begin{equation}
    1 - 15\delta > 0, ~~1 + 6\delta - 315\delta^2 > 0, ~~ 1 +18\delta -15\delta^2 > 0,
  \end{equation}
  which are statisfied if and only if $\delta < 1/15$. Lastly, for the
  determinant of the Hessian matrix we have
  \begin{equation}
    \det (\mmmathcal(\vec{p})) > 0 \Leftrightarrow 
    1 + 6\delta - 315\delta^2 > 0 
    \text{ and }
    1 + 18\delta - 15\delta^2 > 0,
  \end{equation}
  which also holds for $\delta < 1/15$.	

  The convexity of $J_a$ follows from restricting $j_a$ to the linear
  subspace of $\leb2(\W)$ comprising all gradients of $\sobh1(\Omega)$
  functions.
\end{proof}

\begin{remark}[monotonicity]
  \label{monotone}
  Since the anisotropic Dirichlet energy
  \eqref{anisotropicdirichletenergy} is Fréchet differentiable and
  strictly convex it follows \citep[from Pro
    7.4][e.g.]{showalterMonotoneOperatorsBanach2013} that the
  mapping $w\in\sobh 1(\Omega)\mapsto J_a'(w)\in(\sobh 1(\Omega))'$ is
  monotone and hemicontinuous. Thus we have
  \begin{equation}\label{convexineq}
    J_a(w) - J_a(v) \leq J_a'(w)[w - v]
    =
    \int_{\Omega}
    \changes{\nabla w \cdot \mat A(\nabla w)
    \nabla (w - v)}.
  \end{equation}
\end{remark}

\begin{remark}[ellipticity]
  Noting \citet[Lem 2.2]{graserConvexMinimizationPhase} we have that 
  \begin{equation}
    (\mat A(\vec p)\vec p,\vec p) = (\nabla(\bar{a}^2(\vec p)),\vec p) = 2\int_\Omega\bar{a}^2(\vec p),~~\forall \vec p \in\mathbb{R}^d
  \end{equation}
  thus from  \eqref{anisotropy_function} and \eqref{anisotropy_function777} we
  derive the bounds
  \begin{equation}\label{anisotropicbounds}
    c_A|w|^2_{\sobh 1(\Omega)}
    \leq
    \int_{\Omega}\norm{\mat A(\nabla w)}|\nabla w|^2
    \leq C_A|w|^2_{\sobh 1(\Omega)}
  \end{equation}
  with
  \begin{equation}\label{abounds}
    c_A:= \min_{\vec{p}\in \mathbb{R}^d} a^2(\vec{p})~~\text{and}~~C_A:=\max_{\vec{p}\in \mathbb{R}^d}  a^2(\vec{p}).
  \end{equation}
\end{remark}

\begin{theorem}[existence of solutions]
  \label{thmexistence}
  For $\delta<1/15$, recall \eqref{anisotropy_function}, initial data $(u_0,c_0)\in \sobh 1(\Omega)\times \leb 2(\Omega)$ and $T>0$, there exist 
  \begin{equation*}
    u\in \leb 2([0,T];\sobh 1(\Omega))\cap \sobh 1([0,T];\leb 2(\Omega))\cap \leb {\infty}([0,T];\leb \infty(\Omega)),
  \end{equation*}
  \begin{equation*}
    c\in \leb 2([0,T];\sobh 1(\Omega))\cap \sobh 1([0,T];(\sobh 1(\Omega))')\cap \leb {\infty}([0,T];\leb 2(\Omega))
  \end{equation*}
  and
  \begin{equation*}
    \phi\in \leb 2([0,T]; \sobhz[\Gamma_1\cup\Gamma_2]1(\Omega))
  \end{equation*}
  such that $(u,c,\phi)$ satisfies $u(0)=u_0$, $c(0) = c_0$ and
  \eqref{orderparameterfinalform}--\eqref{electricpotentialfinalform}
  for almost every $t\in[0,T]$ and every
  $(\chi,\eta)\in\sobh1(\Omega)\times\sobhz[\Gamma_1\cup\Gamma_2]1(\Omega)$.
\end{theorem}

We split the proof of Theorem \ref{thmexistence} in two parts. The
first part, which is presented in Section \ref{sec3}, consists of
proving existence of a solution to a semi-implicit Euler time
discretization of
\eqref{orderparameterfinalform}--\eqref{electricpotentialfinalform}
and deriving stability bounds on this solution. In the second part,
presented in Section \ref{sec4}, we show that in the limit
as the time step $\tau$ of the semi-discrete approximation tends to
zero, there exists a solution that solves the weak formulation
(\ref{orderparameterfinalform})--(\ref{electricpotentialfinalform}).

\section{Time discretization}\label{sec3}

In this section we study a time discretization of
\eqref{orderparameterfinalform}--\eqref{electricpotentialfinalform}.
To this end, we partition the time interval $[0,T]$ into $K$
equidistant steps, $0 = t_0 < t_1 < \ldots < t_{K-1} < t_K = T$, so
that for $k=1,\ldots,K$, $t_k = k \tau$.

We use the backward difference quotient,
$d_{\tau} w^k =(w^k-w^{k-1})/\tau$ to obtain the following discretization of
\eqref{orderparameterfinalform}--\eqref{electricpotentialfinalform} in
which we have linearized the convection term in
\eqref{concentrationfinalform} and derived a specific form of the
right hand side of \eqref{orderparameterfinalform} that lends itself
to proving the boundedness of $u^k$ (see Lemma
\ref{lemmaxprin}).\\ For all $(\chi,\eta)\in \sobh 1(\Omega)\times
\sobhz[\Gamma_1\cup\Gamma_2]1(\Omega)$ and for all $k\geq
1,k\in\mathbb{N}$
\begin{subequations}
  \begin{align}
    \label{discretisedorderparameter}
    ( d_{\tau} u^k,\chi) + (\mat A(\nabla u^k)\nabla u^k,\nabla \chi) +( g'(u^{k}),\eta) = (\mcm\hop,\chi) \nonumber\\
    +(\mcp\htp,\chi),\\
    \label{discretisedelectricpotential}
    (\sigma(u^{k})\nabla\phi^k,\nabla \eta) = ( -d_{\tau} u^{k} + \sigma'(u^{k})\changes{\partial_{x_1} u^{k}},\eta),\\
    \label{discretisedconcentration}
    ( d_{\tau} c^k,\chi) + (D(u^{k})\nabla c^k,\nabla \chi) = - (d_{\tau} u^{k},\chi) - (D_1(u^{k},c^{k-1})\nabla\phi^{k},\nabla \chi), 
  \end{align}
\end{subequations}
with $u^0 = u(0)=u_0$, $c^0 = c(0)=c_0$, $(\xi)^- := \min (0,\xi)$, $(\xi)^+ := \max (0,\xi)$, 
\begin{equation}\label{hhhh}
  \changes{\hop:=\hopa,~~\htp:=\htpa.}
\end{equation}

\changes{
	\begin{remark}[Right-hand side of \eqref{discretisedorderparameter}]
		We can retrieve $h'(s)$ as defined in \eqref{primes} by setting $r=s$ in $\hope(s,r)$ and $\htpe(s,r)$ such that  
		\begin{equation}
			h'(s) = \frac12\big(\hope(s,s) + \htpe(s,s)\big).
		\end{equation}
		Associated with $\hope(s,r)$ and $\htpe(s,r)$ are two functions $\hoe,\hte:\mathbb{R}^2\to\mathbb{R}$, which are antiderivatives in $s$ of $\hope$ and $\htpe$, such that %
		\begin{equation}\label{eqdfg}
		\begin{gathered}
			\hoe(s,r) := \hoa\\
			\hte(s,r) := \hta %
			 \end{gathered}
		\end{equation}
		where, noting \eqref{doublewell}, we have
		\begin{equation}
			h(s) = \hoe(s,s) + \hte(s,s).
		\end{equation}
	\end{remark}
	} 

Note that the equations in the system \eqref{discretisedorderparameter}--\eqref{discretisedconcentration} are not simultaneous, but will be solved in sequential order, first the phase-field equation \eqref{discretisedorderparameter} is solved with given $c^{k-1}$ to obtain $u^k$, then \eqref{discretisedelectricpotential} is solved with given $u^k$ to obtain $\phi^k$ and finally the convection-diffusion equation \eqref{discretisedconcentration} is solved with given $u^k$ and $\phi^k$ to obtain $c^k$.\\

We make the following assumptions on the time step $\tau$, %
anisotropy parameter $\delta$, defined in \eqref{anisotropy_function},
and the initial data $u_0,c_0$,
\begin{equation}
  \label{assump}
  \begin{gathered}
    \displaystyle{0<\tau<\min\left(\fracl12, \fracl1{\cstartwo},\fracl1{4\cstar}\right)},
    \\
    \delta<\fracl1{15},
    \\
    0 \leq u_0 \leq 1, u_{0}\in \sobh 1(\Omega), c_{0}\in \leb 2 (\Omega).
  \end{gathered}
\end{equation}
Here $\cstartwo$ and $\cstar$ are defined by the bounds in
\eqref{cstartwo} and \eqref{cstar} respectively.

\begin{theorem}[existence and uniqueness of a time-discrete solution]
  \label{thmdiscreteexistence}
  Let assumptions \eqref{assump} hold, then 
  there exists a unique triple 
  \begin{equation}
    (u^k,c^k,\phi^k)
    \in
    (\sobh 1 (\Omega))^2\times\sobhz[\Gamma_1\cup\Gamma_2] 1(\Omega),
  \end{equation}
  with $k=1,...,K= \lceil T/\tau \rceil$, satisfying the $k$-th step of \eqref{discretisedorderparameter}--\eqref{discretisedconcentration}.
\end{theorem}

We split the proof of Theorem \ref{thmdiscreteexistence} into several
parts consisting of Lemmas \ref{lemsemidiscreteAC} --
\ref{lemexistenceconcentration}, in which we prove, in turn, the
existence of unique solutions, $u^k$, $\phi^k$ and $c^k$ to
\eqref{discretisedorderparameter},
\eqref{discretisedelectricpotential} and
\eqref{discretisedconcentration} respectively.  In order to show the
existence of a unique solution, $u^k$ to
\eqref{discretisedorderparameter}, we first show that $u^k$ remains
valued in $\clinter01$ for all $k$.
\begin{lemma}[stability of the time-discrete order parameter]
  \label{lemmaxprin}
  Under assumptions \eqref{assump} and  $0 \leq u^{k-1} \leq 1$, 
  we have
  \begin{equation}\label{boundsorderparameter}
    0 \leq u^k \leq 1.
  \end{equation}
\end{lemma}
\begin{proof}
  Setting $\chi = (u^k)^-$, in \eqref{discretisedorderparameter} gives 
  \begin{multline}
    (u^k,(u^k)^-)
    +
    \tau(\mat A(\nabla u^k)\nabla u^k,\nabla (u^k)^-)
    -
    \tau (\mcm\hop,(u^k)^-)
    \\
    -
    \tau (\mcp\htp,(u^k)^-)
    +
    \tau  (g'(u^k),(u^k)^-)= (u^{k-1},(u^k)^-)
  \end{multline}
  then noting \eqref{anisotropicbounds} we have
  \begin{multline}
    \|(u^k)^-\|_{\leb 2(\Omega)}^2
    +
    \tau c_A |(u^k)^-|^2_{\sobh 1(\Omega)}
    \\
    \leq
    (u^{k-1},(u^k)^-)
    +
    \tau (\mcm\hop,(u^k)^-)
    \\
    +
    \tau (\mcp\htp,(u^k)^-)
    -
    \tau ( g'(u^k),(u^k)^-).
  \end{multline}
  By assumption $(u^{k-1})^-=0$ and noting from \eqref{primes} and \eqref{hhhh} that $(g'(u^k),(u^k)^-)\geq 0$, 
  $(\mcm\hop,(u^k)^-)\leq 0$ and %
  $(\mcp\htp,(u^k)^-)\leq 0$
  it follows that 
  \begin{equation}\label{lowerboundop}
    \|(u^k)^-\|^2_{\leb 2 (\Omega)} \leq 0~~\Rightarrow~~u^k \geq 0.
  \end{equation}
  Working similarly with $\eta = (u^k - 1)^+$ we obtain $u^k \leq 1$, which concludes the proof. 
\end{proof}
\begin{remark}[useful inequalities]
  Combining the result in Lemma \ref{lemmaxprin} with \eqref{Dsigma}, \eqref{diffusionoftransport} and \eqref{primes}, we have 
  \begin{equation}\label{dbounded}
    0<\min(D^e,D^s)=:D_{\min}\leq D(u^k) \leq D_{\max}:=\max(D^e,D^s),
  \end{equation}
  \begin{equation}\label{sigmabounded}
    0<\min(\sigma^e,\sigma^s)=:\sigma_{\min}\leq\sigma(u^k)\leq \sigma_{\max}:=\max(\sigma^e,\sigma^s),
  \end{equation}
  \begin{equation}\label{d1bounded}
    \abs{g'(u)}+\abs{h'(u)}+\abs{m(c)} + \abs{D_1(u,c)}\leq C,
  \end{equation}
  \begin{equation}
    \abs{ \left(g'(u^k), \chi\right)}
    \leq C \|u^k\|_{\leb 2(\Omega)}\|\chi\|_{\leb 2(\Omega)}, 
    \label{gg}
  \end{equation}
  \begin{equation}
    \abs{\left(\mcm\hop + \mcp\htp, \chi\right)}
    \leq C \|u^k\|_{\leb 2(\Omega)}\|\chi\|_{\leb 2(\Omega)}.
    \label{hh}
  \end{equation}
\end{remark}

\begin{lemma}[existence and uniqueness the time-discrete order parameter]
  \label{lemsemidiscreteAC}
  Suppose assumptions \eqref{assump} hold and $u^{k-1}\in\leb
  2(\Omega)$. Then there exists a unique solution $u^k\in\sobh
  1(\Omega)$ that satisfies the $k$-th step of
  \eqref{discretisedorderparameter}.
\end{lemma}
\begin{proof}
  We rewrite \eqref{discretisedorderparameter} as follows
  \begin{equation}\label{gradfloworderparameter}
    (d_{\tau} u^k,\chi) + I'(u^k)\chi = 0 ,~~\forall \chi \in \sobh 1(\Omega), 
  \end{equation}
  where $I': \sobh 1(\Omega)\to (\sobh 1(\Omega))'$ is given by
  \begin{equation}\label{frechetderivative}
    \begin{split}
      I'(w)\chi
      &
      =
      \left(\mat A(\nabla w)\nabla w,\nabla \chi\right) + \left(g'(w) - \mcm\hopw, \chi\right)
      \\
      &\phantom=
      -\left(\mcp\htpw,\chi\right)
    \end{split}
  \end{equation} 
  is the (Fréchet) derivative of the energy functional $I: \sobh 1(\Omega)\to\mathbb{R}$ with, \changes{recalling \eqref{eqdfg},}
  \begin{equation}\label{energyfunctionalorderparameter}
    \begin{split}
      I(w)
      &
      =
      \int_{\Omega}
      \Big(\frac{a(\nabla w)^2}{2}|\nabla w|^2 + g(w) - \changes{\mcm\how}
      \\&
      \phantom=
      -
      \changes{ \mcp\htw}\Big). 
    \end{split}
  \end{equation}
  We now seek a minimiser to the functional $I_k : \sobh 1(\Omega)\to\mathbb{R}$
  \begin{equation}\label{discretefunctional}
    I_k(w) = \frac{1}{2\tau} \|w-u^{k-1}\|^2_{\leb 2(\Omega)} + I(w)
    \changes{\ = \int_{\Omega}F(x,w,\nabla w)},
  \end{equation}
  \changes{%
    where 
    \begin{multline}
      \label{integranddiscrete}
      F(x,\xi,\vec p) := \frac{a(\vec{p})^2}{2}|\vec p|^2
      + \frac{1}{2\tau}(\xi-u^{k-1})^2
      \\
      + \changes{g(\xi) - \mcm\hox - \mcp\htx}.%
    \end{multline}
  Note that $I_k$ is the functional of which 
  \eqref{gradfloworderparameter} is the Euler--Lagrange equation.
  }%
  From Theorem 3.30 in \citet{dacorognaDirectMethodsCalculus2008}, in
  order to show the existence of a minimiser of
  \eqref{discretefunctional} it is sufficient to exhibit that the
  function $F$ is a Carathéodory function, coercive and satisfies a
  suitable growth condition. From
  \citep[Remark~2.4(i)]{bartelsNumericalMethodsNonlinear2015} it follows that $F$
  is a Carathédory function.  From \eqref{d1bounded} we deduce that
  there are constants $\gamma_1,\gamma_2 >0$ such that
  \begin{equation}\label{polynomialbounds}
    -\gamma_1 \leq \changes{g(\xi) - \mcm\hox - \mcp\htx} \leq \gamma_2.
  \end{equation} 
  Using \eqref{abounds} and 
  \eqref{polynomialbounds} 
  we establish the following bounds on $F$:
  \begin{align}
    F(x,\xi,\vec p) &
    =
    \frac{a(\vec{p})^2}{2}|\vec p|^2
    +
    \frac{1}{2\tau}(\xi-u^{k-1})^2
    +
    \changes{g(\xi) - \mcm\hox} \nonumber \\
    & \qquad \changes{- \mcp\htx}
    \nonumber  \\
    &
    \geq \frac{c_A}{2} \abs{\vec p}^2
    +
    \frac{1}{4\tau}\xi^2
    -
    \frac{1}{2\tau}(u^{k-1})^2
    -
    \gamma_1
  \end{align}
  and
  \begin{align}
    F(x,\xi,\vec p)
    &
    =
    \frac{a(\vec{p})^2}{2}|\vec p|^2
    +
    \frac{1}{2\tau}(\xi-u^{k-1})^2
    +
    \changes{g(\xi) - \mcm\hox \nonumber} \\
    &\qquad \changes{- \mcp\htx \nonumber}\\
    &
    \leq
    \frac{C_A}{2} \abs{\vec p}^2 + \frac{1}{\tau} \xi^2
    +
    \frac{1}{\tau} (u^{k-1})^2 + \gamma_2.
  \end{align}
  It only remains to show the uniqueness of the minimizer. Assume
  $w_1$ and $w_2$ are minimizers. From \eqref{gradfloworderparameter}
  it follows that
  \begin{multline}
    \label{inteq1}
    \int_{\Omega} \frac{1}{\tau} (w_1 - w_2)v
    +\int_{\Omega} (\mat A(\nabla w_1)\nabla w_1 - \mat A(\nabla w_2)\nabla w_2)\cdot \nabla v
    \\
    +\,
    \int_{\Omega} \Big(g'(w_1) - g'(w_2)\Big)v  
    -
    \int_{\Omega} \mcp\Big(\hopwo-\hopwt\Big)v\\
    -
    \int_{\Omega} \mcm\Big(\htpwo-\htpwt \Big)v = 0.
  \end{multline}

  Choosing $v = w_1 - w_2$ in \eqref{inteq1}, using the
  Lipschitz continuity of $g'$, $\hop$ and $\htp$,
  \eqref{d1bounded}, and noting from Remark \ref{monotone} the
  monotonicity of the mapping $w\mapsto J_a'(w)$, we obtain
  \changes{%
    \begin{multline}
      \label{cstartwo}
      \frac1{\tau} \|w_1 - w_2\|^2_{\leb 2(\Omega)}
      \leq
      \int_{\Omega} |(g'(w_1) - g'(w_2))(w_1-w_2)|
      \\
      + C
      \int_{\Omega}
      |\hopwo - \hopwt||w_1-w_2|
      \\
      \leq
      \cstartwo
      \|w_1 - w_2\|^2_{\leb 2(\Omega)}
    \end{multline}
  }%
  and hence 
  \begin{equation}
    (\frac{1}{\tau}-\cstartwo) \|w_1 - w_2\|^2_{\leb 2(\Omega)} \leq 0,
  \end{equation}
  where $\cstartwo$ is independent of $\tau$. Thanks to
  assumption \eqref{assump} we have $\tau<1/\cstar$ and hence
  the uniqueness of the minimizer.
\end{proof}

\begin{lemma}[energy estimates for the order parameter]
  \label{lemenergyestimatesorder}
  Under assumptions \eqref{assump} we have 
  \begin{equation}
    \label{energyestimatesorder1}
    \max_{l=0,\dotsc,K}
    \|u^l\|_{\leb 2(\Omega)}^2
    +
    \tau \sum_{k=1}^K \|u^k\|^2_{\sobh 1(\Omega)}
    +
    \tau \sum_{k=1}^K \|d_{\tau} u^k\|^2_{\leb 2(\Omega)}
    \leq C.
  \end{equation}
\end{lemma}

\begin{proof}
  Setting $\chi=u^k$ in \eqref{gradfloworderparameter} and noting \eqref{anisotropicbounds}, \eqref{frechetderivative}, \eqref{hh}, \eqref{gg} and 
  \begin{equation}\label{discretetimederformula}
    ( d_{\tau} u^k,u^k) = \frac{\tau}{2} \|d_{\tau} u^k\|^2_{\leb 2(\Omega)} + \frac{1}{2}d_{\tau}\|u^k\|^2_{\leb 2(\Omega)} 
  \end{equation}
  we have
  \begin{multline}\label{cstar}
    \frac{\tau}{2}
    \|d_{\tau} u^k\|^2_{\leb 2(\Omega)}
    +
    \frac{1}{2}d_{\tau}\|u^k\|^2_{\leb 2(\Omega)}
    + c_A |u^k|^2_{\sobh 1(\Omega)}
    \\
    \leq 
    \abs{\left(g'(u^k), u^k\right)}%
    \\
    +\abs{ \left(\mcm\hop, u^k\right)} + \abs{\left(\mcp\htp, u^k\right)}\\%\nonumber\\ 
    \qquad \qquad\qquad \qquad \qquad\qquad \leq \cstar \|u^k\|^2_{\leb 2(\Omega)},
  \end{multline}
  where $\cstar$ is independent of $\tau$. 
  Multiplying by $\tau$ and summing from $k=1,2,...,l$, yields for any $l\leq K$ that 
  \begin{equation}
    \left(\frac{1}{2}-\tau \cstar\right) \|u^l\|^2_{\leb 2(\Omega)} + \tau c_A \sum_{k=1}^l |u^k|^2_{\sobh 1(\Omega)} \leq \frac{1}{2} \|u^0\|^2_{\leb 2(\Omega)} + \tau C \sum_{k=1}^{l-1} \|u^k\|^2_{\leb 2(\Omega)}.
    \label{ppp}
  \end{equation}
  Owing to \eqref{assump} we have $\tau<1/(4\cstartwo)$. Hence using 
  a generalized discretized Gronwall lemma
  \citep[Lem 2.2]{bartelsNumericalMethodsNonlinear2015} we obtain
  \begin{equation}
    \label{1stenergyestimates}
    \frac{1}{4} \max_{l\integerbetween1K}
    \|u^l\|^2_{\leb 2(\Omega)}
    +
    \tau c_A
    \sum_{k=1}^K
    |u^k|_{\sobh 1(\Omega)}^2
    \leq
    \frac{1}{2}
    \|u^0\|^2_{\leb 2(\Omega)}
    \expp {C T}
    \leq
    C
    ,
  \end{equation}
  which yields the first two bounds in \eqref{energyestimatesorder1}. 
  To prove the third bound we set $\chi=d_{\tau}u^k$ in  \eqref{gradfloworderparameter}, multiply the resulting equation by $\tau$ and sum over  $k=1,2,...,l$, to obtain for any $l\leq K$ that
  \begin{multline}
    \tau  \sum_{k=1}^l \|d_{\tau}u^k\|^2_{\leb 2(\Omega)} + \sum_{k=1}^l (\mat A(\nabla u^k)\nabla u^k,\nabla[u^k - u^{k-1}]) + \tau \sum_{k=1}^l (g'(u^k),d_{\tau}u^k) \\
    \qquad = \tau\sum_{k=1}^l (\mcm\hop,d_{\tau}u^k) + \tau\sum_{k=1}^l (\mcp\htp,d_{\tau}u^k).
  \end{multline} 
  Using 
  \eqref{convexineq}, \eqref{hh} and \eqref{gg} we have
  \begin{equation}
    \tau\sum_{k=1}^l\|d_{\tau}u^k\|^2_{\leb 2(\Omega)} + \sum_{k=1}^l(J_a(u^k) - J_a(u^{k-1})) \leq C\tau\sum_{k=1}^l  \|u^k\|^2_{\leb 2(\Omega)} 
    + \frac{\tau}{2}\sum_{k=1}^l \|d_{\tau}u^k\|^2_{\leb 2(\Omega)}.
  \end{equation}
  Hence, noting \eqref{1stenergyestimates}, \eqref{assump} and
  $u_0\in \sobh 1(\Omega)$, we have
  \begin{equation}\label{3stenergyestimates}
    \frac{\tau}{2}\sum_{k=1}^l \|d_{\tau}u^k\|^2_{\leb 2(\Omega)} + J_a(u^l) \leq J_a(u^0) +  \frac{\tau C}{4}\sum_{k=1}^l \|u^k\|^2_{\leb 2(\Omega)} \leq C,
  \end{equation}
  which completes the proof. 
\end{proof}
\begin{lemma}[existence and uniqueness of the time-discrete potential]
  \label{lemexistencepoisson}
  Let assumptions \eqref{assump} hold. Then there exists
  a unique solution $\phi^k\in \sobhz[\Gamma_1\cup\Gamma_2] 1(\Omega)$, 
  that satisfies the $k$-th step of \eqref{discretisedelectricpotential}.
\end{lemma}
\begin{proof}
  We introduce the form $b_1:(\sobh 1(\Omega))^2\to\mathbb{R}$ 
  \begin{equation}\label{bilinearformelectric}
    b_1(u^k;v,w) = \int_{\Omega}\sigma(u^{k})\nabla v\cdot \nabla w,
  \end{equation}
  which is bilinear in respect to $v$ and $w$.
  Using \eqref{bilinearformelectric} we can rewrite \eqref{discretisedelectricpotential} as
  \begin{equation}\label{hjk}
    b_1(u^k;\phi^k,\eta) = -( d_{\tau} u^{k}+\sigma'(u^k)\changes{\partial_{x_1} u}^k,\eta).
  \end{equation}
  We now prove existence of a weak solution using the Lax--Milgram Theorem. 
  The coercivity and boundedness of $b_1$ \nocite{evansPartialDifferentialEquations1998} follow from \eqref{sigmabounded} and the equivalence of the $\sobh 1$ semi-norm and the $\sobh 1$ norm under a homogeneous Dirichlet boundary condition:
  \begin{equation}
    b_1(u^k;v,v) = \int_{\Omega} \sigma(u^{k})\abs{\nabla v}^2 \geq \sigma_{\min}\abs{v}^2_{\sobh 1(\Omega)},
  \end{equation}
  \begin{equation}
    \abs{b_1(u^k;v,w)} \leq \sigma_{\max}\abs{v}_{\sobh 1(\Omega)}\abs{w}_{\sobh 1(\Omega)}.
  \end{equation} 
  From Lemma \ref{lemenergyestimatesorder}  we have that 
  \begin{equation}
    \|u^k\|^2_{\sobh 1(\Omega)} +  \|d_{\tau} u^k\|^2_{\leb 2(\Omega)} \leq \frac{C}{\tau}
  \end{equation} 
  which together with \eqref{sigmaprime}, imply 
  \begin{equation}
    \|d_{\tau} u^{k}+\sigma'(u^k)\changes{\partial_{x_1} u}^k\|^2_{\leb 2(\Omega)}\leq \frac{C}{\tau}.
  \end{equation} 
  Thus, by 
  the Lax--Milgram Theorem, \eqref{discretisedelectricpotential} has a unique solution at the $k$-th step $\phi^k \in \sobhz[\Gamma_1\cup\Gamma_2]1(\Omega)$, with 
  \begin{equation}\label{ccck}
    \|\phi^k\|_{\sobh 1(\Omega)}  \leq \frac{C}{\tau}.
  \end{equation} 
\end{proof}

\begin{lemma}[existence and uniqueness of the time-discrete concentration]
  \label{lemexistenceconcentration}
  Under assumptions \eqref{assump}, if $ \|c^{k-1}\|_{\leb
    2(\Omega)}\leq C$, then there exists a unique solution
  $c^k\in\sobh 1(\Omega)$ that satisfies the $k$-th step of
  \eqref{discretisedconcentration}.
\end{lemma}

\begin{proof}
  We introduce the form $b_2: (\sobh 1(\Omega))^2\to\mathbb{R}$
  \begin{equation}\label{linearizedbilinearform}
    b_2(u^k;w,v):= \int_{\Omega} D(u^{k})\nabla w\cdot\nabla v + \frac{1}{\tau}wv,
  \end{equation}
  which is bilinear for $w$ and $v$. Using \eqref{linearizedbilinearform} we can rewrite \eqref{discretisedconcentration} as
  \begin{equation}\label{pppp}
    b_2(u^k;c^k,\chi) = -( d_{\tau} u^{k},\chi) - (D_1(u^{k},c^{k-1})\nabla\phi^{k},\nabla \chi) + \frac{1}{\tau}(c^{k-1},\chi).
  \end{equation}
  Again we prove existence of a weak solution using the Lax--Milgram Theorem. Coercivity and boundedness of $b_2$ follow from \eqref{dbounded}: 	\begin{multline}
    b_2(u^k;w,w) = \int_{\Omega}D(u^{k})\abs{\nabla w}^2 + \frac{1}{\tau}\abs{w}^2 \geq D_{\min}|w|^2_{\sobh 1(\Omega)} + \frac{1}{\tau}\|w\|^2_{\leb 2(\Omega)}\\
    \geq D_{\min}\|w\|^2_{\sobh 1(\Omega)},
  \end{multline}
  \begin{multline}
    \abs{b_2(u^k;w,v)} = \abs{\int_{\Omega} D(u^{k})\nabla w\cdot\nabla v + \frac{1}{\tau}wv} \leq \int_{\Omega} |D(u^{k})||\nabla w||\nabla v| + \frac{1}{\tau}|w||v|\\
    \leq D_{\max} |w|_{\sobh 1(\Omega)}|v|_{\sobh 1(\Omega)} + \frac{1}{\tau}\|w\|_{\leb 2(\Omega)}\|v\|_{\leb 2(\Omega)} \leq \frac{1}{\tau} \|w\|_{\sobh 1(\Omega)}\|v\|_{\sobh 1(\Omega)}.
  \end{multline}
  From Lemma \ref{lemenergyestimatesorder}, the assumption that $ \|c^{k-1}\|_{\leb 2(\Omega)}\leq C $, \eqref{d1bounded} and \eqref{ccck} we have 
  \begin{equation}
    \|d_{\tau} u^k\|_{\leb 2(\Omega)} + \frac{1}{\tau}\|c^{k-1}\|_{\leb 2(\Omega)} +
    \|D_1(u^{k},c^{k-1})\nabla\phi^{k}\|^2_{\leb 2(\Omega)}\leq \frac{C}{\tau}.
  \end{equation} 
  Thus the right hand side of \eqref{pppp} is bounded, dependent on $\tau$,  in $(\sobh 1(\Omega))'$ and by the Lax--Milgram Theorem, there exists a unique solution $c^k \in \sobh 1(\Omega)$ that satisfies the $k$-th step of \eqref{discretisedconcentration}. 
\end{proof}

\begin{lemma}[energy stability of the potential]
  \label{lemenergyestimatespoisson}
  Let assumptions \eqref{assump} hold and let $\phi^k$, be the
  solution to the $k$-th step of
  \eqref{discretisedelectricpotential}. Then we have
  \begin{equation}\label{energyestimateselectric}
    \tau \sum_{k=1}^K \|\phi^k\|^2_{\sobh 1(\Omega)} \leq C.
  \end{equation}
\end{lemma}

\begin{proof}
  We set $\eta = \phi^k$ in  \eqref{discretisedelectricpotential} to give
  \begin{equation*}
    \int_{\Omega}\sigma(u^{k})|\nabla\phi^k|^2 = -\int_{\Omega}d_{\tau} u^{k}\phi^k + \int_{\Omega} \sigma'(u^{k})\changes{\partial_{x_1} u^{k}}\phi^k.
  \end{equation*}
  Noting \eqref{sigmabounded} we have
  \begin{equation}\label{hgf}
    ~\sigma_{\min}|\phi^k|^2_{\sobh 1(\Omega)} \leq \|d_{\tau} u^{k}\|_{\leb 2(\Omega)}\|\phi^k\|_{\leb 2 (\Omega)} + \sigma'_{\max} \|\changes{\partial_{x_1} u^{k}}\|_{\leb 2(\Omega)}\|\phi^k\|_{\leb 2(\Omega)},
  \end{equation} 
  where $\sigma_{\max}':=\max_{s\in\mathbb{R}}\sigma'(s)$. Hence from the Friedrichs inequality we have 
  \begin{equation}\label{nnb}\begin{split}
      \frac{\sigma_{\min}}{2C_f}\|\phi^k\|_{\leb 2(\Omega)}^2 + \frac{\sigma_{\min}}{2}|\phi^k|^2_{\sobh 1(\Omega)} \leq C\|d_{\tau} u^{k}\|_{\leb 2(\Omega)}^2 + \frac{\sigma_{\min}}{8C_f}\|\phi^k\|^2_{\leb 2(\Omega)} \\
      + C \|\changes{\partial_{x_1} u^{k}}\|_{\leb 2(\Omega)}^2 + \frac{\sigma_{\min}}{8C_f}\|\phi^k\|^2_{\leb 2(\Omega)}
    \end{split}
  \end{equation} 
  where $C_f$ is the constant in the Friedrichs inequality, which is associated with the diameter of $\Omega$. 
  After rearranging \eqref{nnb}, multiplying the resulting inequality by $\tau$ and summing over $k=1,2,...,l$, for any $l\leq K$ we have
  \begin{equation}
    \tau\sum_{k=1}^l\|\phi^k\|^2_{\sobh 1(\Omega)} \leq C\Big (\tau\sum_{k=1}^l\|d_{\tau}u^k\|^2_{\leb 2(\Omega)} + \tau\sum_{k=1}^l\|\changes{\partial_{x_1} u^k}\|^2_{\leb 2(\Omega)} \Big)
  \end{equation}
  and the result follows from Lemma \ref{lemenergyestimatesorder}.
\end{proof}
\begin{lemma}[energy estimates for the concentration]
  \label{lemenergyestimatesconcentration}
  Let assumptions \eqref{assump} hold. Then 
  \begin{equation}\label{energyestimatesconc1}
    \max_{l=0,...,K}\|c^l\|^2_{\leb 2(\Omega)} + \tau \sum_{k=1}^K \|c^k\|^2_{\sobh 1(\Omega)} + \tau \sum_{k=1}^K \|d_{\tau} c^k\|^2_{(\sobh 1(\Omega))'} \leq C.
  \end{equation}
\end{lemma}
\begin{proof}
  Setting $\chi = c^k$ in \eqref{discretisedconcentration} 
  gives
  \begin{equation}
    ( d_{\tau} c^k,c^k) + (D(u^k)\nabla c^k,\nabla c^k)
    =
    -
    (d_{\tau}u^k,c^k)
    -
    (D_1(u^k,c^{k-1})\nabla \phi^k,\nabla c^k)
  \end{equation}
  and noting \eqref{discretetimederformula}, \eqref{dbounded} and \eqref{d1bounded}, we obtain
  \begin{multline}
    \frac{\tau}{2} \|d_{\tau} c^k\|^2_{\leb 2 (\Omega)}
    +
    \frac{1}{2}d_{\tau}\|c^k\|^2_{\leb 2(\Omega)}
    +
    D_{\min} |c^k|^2_{\sobh 1(\Omega)}
    \\
    \leq
    \frac{1}{2} \|d_{\tau}u^k\|^2_{\leb 2(\Omega)}
    + \frac{1}{2}\|c^k\|^2_{\leb 2(\Omega)}
    + C|\phi^k|^2_{\sobh 1(\Omega)}
    +
    \frac{D_{\min}}{2}|c^k|^2_{\sobh 1(\Omega)}.\label{dbc}
  \end{multline}
  After rearranging \eqref{dbc}, multiplying the resulting inequality by $\tau$ and summing over $k=1,2,...,l$, 
  for any $l\leq K$, we have
  \begin{multline*}
    \frac{\tau^2}{2} \sum_{k=1}^l \|d_{\tau}c^k\|^2_{\leb 2(\Omega)} + \frac{1}{2}\|c^l\|^2_{\leb 2 (\Omega)} + \tau\frac{D_{\min}}{2}\sum_{k=1}^l |c^k|^2_{\sobh 1(\Omega)}\\
    \leq \frac{1}{2}\|c^0\|^2_{\leb 2(\Omega)} + \frac{\tau}{2} \sum_{k=1}^l \|d_{\tau}u^k\|^2_{\leb 2(\Omega)} +  \frac{\tau}{2} \sum_{k=1}^l \|c^k\|^2_{\leb 2(\Omega)}+ \tau C\sum_{k=1}^l |\phi^k|^2_{\sobh 1(\Omega)}.
  \end{multline*}
  Thus, noting from assumption \eqref{assump} that $c_0\in \leb
  2(\Omega)$, we have
  \begin{multline}
    \label{mmn}
    \frac{1}{2}\left( 1 - \tau\right)\|c^l\|^2_{\leb 2(\Omega)}
    +
    \tau\frac{D_{\min}}{2}
    \sum_{k=1}^l |c^k|^2_{\sobh 1(\Omega)}
    \leq
    C\Big(1+ \tau\sum_{k=1}^l \|d_{\tau}u^k\|^2_{\leb 2(\Omega)}
    \\
    +
    \tau\sum_{k=1}^l |\phi^k|^2_{\sobh 1(\Omega)}
    +
    \tau \sum_{k=1}^{l-1} \|c^k\|^2_{\leb 2(\Omega)}\Big).
  \end{multline}
  Since \eqref{assump} holds we have $\tau<1/2$ and hence using the discrete Gronwall Lemma, we obtain
  \begin{equation}\begin{split}
      \frac{1}{4}\max_{l=1,...,K}\|c^l\|^2_{\leb 2(\Omega)} + \tau \frac{D_{\min}}{2}\ \sum_{k=1}^K |c^k|^2_{\sobh 1(\Omega)}~~~~~~~~~~~~~~~~~~~~~\\
      \leq C\Big(1+ \tau\sum_{k=1}^K \|d_{\tau}u^k\|^2_{\leb 2(\Omega)} + \tau \sum_{k=1}^K |\phi^k|^2_{\sobh 1(\Omega)} \Big)\expp {\frac{T}{2}}
    \end{split}
  \end{equation}
  and from Lemmas \ref{lemenergyestimatesorder} and \ref{lemenergyestimatespoisson} we conclude
  \begin{equation}
    \frac{1}{4}\max_{l=0,...,K}\|c^l\|^2_{\leb 2(\Omega)} + \tau \sum_{k=1}^K |c^k|^2_{\sobh 1(\Omega)}\leq C.
    \label{jkl}
  \end{equation}
  To prove the last estimate in \eqref{energyestimatesconc1}, we note
  from \eqref{discretisedconcentration}, \eqref{dbounded} and
  \eqref{d1bounded} that
  \begin{multline}
    (d_{\tau}c^k,\chi)_{((\sobh 1(\Omega))',\sobh 1(\Omega))}
    \leq
    \|D(u^k)\nabla c^k \|_{\leb 2(\Omega)}\|\nabla \chi\|_{\leb 2(\Omega)}
    \\
    +\| d_{\tau}u^k\|_{\leb 2(\Omega)}\| \chi\|_{\leb 2(\Omega)}
    + \|D_1(u^k,c^{k-1})\nabla\phi^k\|_{\leb 2(\Omega)}\|\nabla \chi\|_{\leb 2(\Omega)}\\
    \leq C (|c^k|_{\sobh 1(\Omega)} + |\phi^k|_{\sobh 1(\Omega)} + \|d_{\tau}u^k\|_{\leb 2(\Omega)})\|\chi\|_{\sobh 1(\Omega)} 
  \end{multline}
  and so we conclude 
  \begin{equation}
    \|d_{\tau}c^k\|_{(\sobh 1(\Omega))'}^2 
    \leq C(|c^k|^2_{\sobh 1(\Omega)} + |\phi^k|^2_{\sobh 1(\Omega)} + \|d_{\tau}u^k\|^2_{\leb 2(\Omega)}).
  \end{equation}
  The required result follows after multiplying by $\tau$, summing over $k=1,2,...,K$ and noting \eqref{jkl} and Lemmas \ref{lemenergyestimatesorder} and \ref{lemenergyestimatespoisson}. 
\end{proof}

\begin{Proof}[ of Theorem \ref{thmdiscreteexistence}]
  Theorem \ref{thmdiscreteexistence} is a direct consequence of Lemmas \ref{lemsemidiscreteAC}, \ref{lemexistencepoisson} and \ref{lemexistenceconcentration}.
\end{Proof}
\section{Weak convergence of the limits}\label{sec4}
In this section, we prove that as $\tau\to 0$ subsequences of the
solution to
\eqref{discretisedorderparameter}--\eqref{discretisedconcentration}
converge to a solution of
\eqref{orderparameterfinalform}--\eqref{electricpotentialfinalform}.

Introduce the notation
\begin{equation}
  u^-_{\tau}(t) = u^{k-1},
  \quad
  u^+_{\tau}(t)
  =
  u^k\text{ for }t_{k-1}<t\leq t_k.
\end{equation}
\changes{We introduce the sequence $u^k$'s piecewise linear interpolant in time}
\begin{equation}\label{interpolantsintime}
  \hat{u}_{\tau}(t)
  =
  \frac{t-t_{k-1}}{\tau}u^k
  +
  \frac{t_k-t}{\tau}u^{k-1},
  \text{ for }t_{k-1}\leq t\leq t_k,
  k\integerbetween0K.
\end{equation}
\changes{We similarly define
$c^\pm_{\tau}(t),\hat{c}_{\tau}(t),\phi^\pm_{\tau}(t)$, and $\hat{\phi}_{\tau}(t)$.}
Next recalling \eqref{hhhh}, the terms on the left hand side of
\eqref{discretisedorderparameter} can be reformulated as
\begin{equation}
  \begin{split}
    (\mcm\hop,\chi)&+(\mcp\htp,\chi) 
    \\
    &=
    30\left(m(c^{k-1})u^{k-1}u^k(u^k-1)^2,\chi\right)
    \\
    &\phantom=
    +30\left((m(c^{k-1}))^+(u^k-1)u^k(u^{k-1}-u^{k}),\chi\right).
  \end{split}
\end{equation}
Thus rewriting
\eqref{discretisedorderparameter}--\eqref{discretisedelectricpotential}
and using the above notation, gives, for $t^{k-1}<t\leq t^k$,
$k=1,\ldots, K$, and for all
$(\chi,\eta)\in \sobh 1(\Omega)\times\sobhz[\Gamma_1\cup\Gamma_2]1(\Omega)$
\begin{subequations}
  \begin{multline}
    \label{approxorder}
    (\partial_t \hat{u}_{\tau},\chi)
    +
    (\mat A(\nabla u^+_{\tau}(t))\nabla u^+_{\tau}(t),\nabla \chi)
    +
    ( g'(u^+_{\tau}(t)),\chi) 
    \\
    = (\mcf,\chi)
  \end{multline}
  \begin{multline}
    \label{approxconcentration}
    (\partial_t \hat{c}_{\tau},\chi) + (D(u^+_{\tau}(t))\nabla c^+_{\tau}(t),\nabla \chi)
    \\
    = -(\hat{u}_{\tau}'(t),\chi) - (D_1(u^+_{\tau}(t),c^-_{\tau}(t))\nabla \phi^+_{\tau}(t), \nabla \chi)
  \end{multline}
  and
  \begin{equation}
    \label{approxelectricpotential}
    (\sigma(u^+_{\tau}(t))\nabla\phi^+_{\tau}(t),\nabla \eta) = -(\hat{u}_{\tau}'(t) + \sigma'(u_{\tau}^+(t))\changes{\partial_{x_1} u_{\tau}^+(t)},\eta), 
  \end{equation}
\end{subequations}
where
\begin{equation}
  \label{mcf}
  \begin{split}
    (\mcf,\chi):=30\left(\mc \umwt \upwt (\upwt-1)^2,\chi\right)
    \\
    +30\left(\mcppwt(\upwt-1)\upwt(\umwt-\upwt),\chi\right).
  \end{split}
\end{equation}
\begin{lemma}[stability bounds]
  \label{lemapproxsystem}
  Assuming \eqref{assump}, we have 
  \begin{equation}
    \label{approxestimatesorder}
    \|u_{\tau}^+\|_{\leb {\infty}([0,T];\leb \infty(\Omega))}
    +
    \|u_{\tau}^+\|_{\leb 2([0,T];\sobh 1(\Omega))}
    +
    \|\partial_t \hat{u}_{\tau}\|_{\leb 2([0,T]; \leb 2(\Omega))}
    \leq
    C
    ,
  \end{equation}
  \begin{equation}
    \label{approxestimatesconcentration}
    \|c_{\tau}^+\|_{\leb {\infty}([0,T];\leb 2(\Omega))}
    +
    \|c_{\tau}^+\|_{\leb 2([0,T];\sobh 1(\Omega))}
    +
    \|
    \partial_t
    \hat{c}_{\tau}
    \|_{\leb 2([0,T]; (\sobh 1(\Omega))')} \leq C,
  \end{equation}
  \begin{equation}
    \label{approxestimateselectric}
    \|\phi_{\tau}^+\|_{\leb 2([0,T];\sobh 1(\Omega))} \leq C. 
  \end{equation}
\end{lemma}

\begin{proof}
  The first and second bounds in \eqref{approxestimatesorder}
  follow from Lemmas \ref{lemmaxprin} and
  \ref{lemenergyestimatesorder}.  For the last bound in
  \eqref{approxestimatesorder}, we note from
  \eqref{interpolantsintime} that $\partial_t \hat{u}_{\tau}(t) =
  d_{\tau}u^k$ on $(t_{k-1},t_k)$, $k=1,2,...,K$ and thus the
  result follows from Lemma \ref{lemenergyestimatesorder}.
  
  The bounds in \eqref{approxestimatesconcentration} and
  \eqref{approxestimateselectric} follow using similar arguments
  and the results of Lemmas \ref{lemenergyestimatesconcentration}
  and \ref{lemenergyestimatespoisson} respectively.
\end{proof}
\begin{lemma}[weak compactness]
  \label{lemselectionofthelimits}
  There exist $(u,c,\phi)$ and $(\partial_t u,\partial_t c)$ 
  such that for a sequence $(\tau_n)_{n\in\mathbb{N}}$ of positive numbers with $\tau_n\to 0$ as $n\to\infty$, we have the following for $1\leq p < \infty$ if $\dim\Omega = 2$ and $1\leq p < 6$ if $\dim\Omega = 3$, up to a subsequence, which for simplicity of presentation we again denote by $\tau_n$,
  \begin{subequations}
    \begin{align}
      \label{LinftyLinftyconvorder}
      \hat{u}_{\tau_n},u_{\tau_n}^{\pm} \overset{\ast}{\rightharpoonup} u \text{ in }  \leb {\infty}([0,T];\leb \infty(\Omega)),\\
      \label{L2H1convorder}
      \hat{u}_{\tau_n},u_{\tau_n}^{\pm} \rightharpoonup u \text{ in } \leb 2([0,T];\sobh 1(\Omega)),\\
      \label{L2L2convtimeorder}
      \partial_t \hat{u}_{\tau_n}\rightharpoonup \partial_tu  \text{ in }  \leb 2([0,T]; \leb 2(\Omega)),\\
      \label{LinftyL2convconce}
      \hat{c}_{\tau_n},c_{\tau_n}^{\pm} \overset{\ast}{\rightharpoonup} c  \text{ in }  \leb {\infty}([0,T];\leb 2(\Omega)),\\
      \label{L2H1convconce}
      \hat{c}_{\tau_n},c_{\tau_n}^{\pm} \rightharpoonup c  \text{ in } \leb 2([0,T];\sobh 1(\Omega)), \\
      \label{L2H1'convconce}
      \partial_t \hat{c}_{\tau_n} \rightharpoonup \partial_tc  \text{ in }  \leb 2([0,T]; (\sobh 1(\Omega))'),\\
      \label{L2H2convelectric}
      \phi_{\tau_n}^{\pm} \rightharpoonup \phi  \text{ in }  \leb 2([0,T];\sobh 1(\Omega)),\\
      \label{strongu}
      u_{\tau_n}^{\pm} \rightarrow u \text{ in } \leb 2([0,T];\leb p(\Omega)),\\
      \label{strongc}
      c_{\tau_n}^{\pm} \rightarrow c \text{ in } \leb 2([0,T];\leb p(\Omega)).
    \end{align}
  \end{subequations}
\end{lemma}

\begin{proof}
  From Lemma \ref{lemapproxsystem}
  it follows that there are weakly convergent subsequences that converge to the appropriate limits in  \eqref{LinftyLinftyconvorder}-\eqref{L2H2convelectric}. 
  To establish \eqref{strongu} and \eqref{strongc} we use \eqref{L2H1convorder}, \eqref{L2L2convtimeorder}, \eqref{L2H1convconce}, \eqref{L2H1'convconce} and the Aubin-Lions Lemma \citep[e.g.][]{roubicekNonlinearPartialDifferential2005}. 
\end{proof}

\begin{lemma}[strong convergence]
  \label{lemstrongconvergencediffcond}
  For $u\in \leb \infty([0,T];\leb \infty(\Omega))$ and 
  \changes{$c\in \leb 2([0,T];\leb p(\Omega))$, with $1\leq p < \infty$ if $\dim\Omega = 2$ and $1\leq p < 6$ if $\dim\Omega = 3$, we have the following as $\tau_n\to 0$,}
  \begin{subequations}
    \begin{align}
      \label{strongconvdoublewellprime}
      g'(u_{\tau_n}^+)\to g'(u) \text{ in } \leb 2([0,T];\leb p(\Omega)),\\
      \label{sabc}
      \umt \upt (\upt-1)^2 \to  u^2(u-1)^2 \text{ in } \leb 2([0,T];\leb p(\Omega)),\\
      \label{strongconvergencediffusion}
      D(u_{\tau_n}^+) \rightarrow D(u) \text{ in } \leb 2([0,T];\leb p(\Omega)),\\
      \label{strongconvergenceconductivity}
      \sigma(u_{\tau_n}^+) \rightarrow \sigma(u)  \text{ in }  \leb 2([0,T];\leb p(\Omega)),\\
      \label{strongconvergenceforcing}
      m(c^{-}_{\tau_n}) \rightarrow m(c) \text{ in } \leb 2([0,T];\leb p(\Omega)),\\
      \label{strongconvergenceconvection}
      D_1(u^+_{\tau_n},c^-_{\tau_n}) \to D_1(u,c) \text{ in }
      \changes{\leb 2([0,T];\leb p(\Omega))},
      \\
      \label{strongconvergencehprime}
      30\,m(\cmt)\umt \upt (\upt-1)^2 \to  m(c)h'(u) \text{ in }
      \changes{\leb 2([0,T];\leb p(\Omega))},
      \\
      \label{convtozero}
      \mcpp(\upt-1)\upt(\umt-\upt)\to 0 \text{ in } \leb 2([0,T];\leb 2(\Omega)).
    \end{align}
  \end{subequations}
\end{lemma}
\begin{proof}
  Noting that $g'(u),D(u),\sigma(u)$ are polynomials in $u$ and $m(c)$
  is linear in $c$, the proofs of
  \eqref{strongconvdoublewellprime}--\eqref{strongconvergenceforcing} follow from Lemma
  \ref{lemapproxsystem}, \eqref{LinftyLinftyconvorder},
  \eqref{LinftyL2convconce}, \eqref{strongu} and \changes{\eqref{strongc}}.  
  \changes{Noting, from definition \eqref{diffusionoftransport},
    that $D_1$ is a bounded Lipschitz function}
  we obtain
  \changes{
  \begin{multline}
  \|D_1(\upt,\cmt) - D_1(u,c)\|_{\leb 2([0,T];\leb p(\Omega))}     \\
        \leq C(\|\upt-u\|_{\changes{\leb 2}([0,T];\leb p(\Omega))}
    +    \|\cmt-c\|_{\changes{\leb 2}([0,T];\leb p(\Omega))})
  \end{multline}
  and \eqref{strongconvergenceconvection} then follows by 
  using \eqref{strongu} and
  \eqref{strongc}.}  Adopting similar techniques and noting from
  \eqref{primes} that $h'(u)=30u^2(u-1)^2$, from
  \eqref{d1bounded}, \eqref{approxestimatesorder}, \eqref{sabc}
  and \eqref{strongconvergenceforcing} we conclude
  \eqref{strongconvergencehprime}.
  
  Noting from
  \eqref{interpolantsintime} that $\tau \partial_t \hat{u}_\tau =
  (\umt-\upt)$ for $t\in (t^{k-1},t^k)$, from \eqref{d1bounded}
  and \eqref{approxestimatesorder} we have
  \begin{multline}
    \|\mcpp(\upt-1)\upt(\umt-\upt)\|_{\leb 2([0,T];\leb 2(\Omega))}
    \\
    \leq C\|\upt-\umt\|_{\leb 2([0,T];\leb 2(\Omega))} = C\tau
  \end{multline}
  which concludes the proof. 
\end{proof}
\begin{lemma}[weak convergence]
  \label{lemweakconvergence}
     \changes{For some $1\leq r\leq
    2p/(p+2)$, with $1\leq p < \infty$ if $\dim\Omega = 2$ and $1\leq p < 6$ if $\dim\Omega = 3$, we have the following}
  as $\tau_n\to 0$
  \begin{subequations}
    \begin{align}
      \label{weakconvergencediffusionterm}
      D(u^+_{\tau_n})\nabla c^+_{\tau_n} \rightharpoonup D(u)\nabla c
      \text{ in }
      \changes{\leb 1}([0,T]; (\changes{\leb r}(\Omega))^d),
      \\
      \label{weakconvergencetransportterm}
      D_1(u^+_{\tau_n},c^-_{\tau_n})\nabla\phi^+_{\tau_n} \rightharpoonup D_1(u,c)\nabla\phi
      \text{ in }
      \changes{\leb 1}([0,T]; (\changes{\leb r}(\Omega))^d),
      \\
      \label{weakconvergenceconductivity}
      \sigma(u^+_{\tau_n})\nabla\phi_{\tau_n}^+ \rightharpoonup \sigma(u)\nabla\phi
      \text{ in }
      \changes{\leb 1}([0,T]; (\changes{\leb r}(\Omega))^d),
      \\
      \label{weakconvergenceanisotropy}
      \mat A(\nabla u^+_{\tau_n})\nabla u^+_{\tau_n} \rightharpoonup \mat A(\nabla u)\nabla u \text{ in } \leb 2([0,T];(\leb 2(\Omega))^d).
    \end{align}
  \end{subequations}
\end{lemma}

\begin{proof}
\changes{%
    From Lemmas \ref{lemstrongconvergencediffcond} and
    \ref{lemselectionofthelimits}, for $1\leq p < \infty$ if $\dim\Omega = 2$ and $1\leq p < 6$ if $\dim\Omega = 3$, we have 
    \begin{equation}
      \begin{split}
        D(u_{\tau_n}^+)\to D(u)\text{ in }\leb2(0,T;\leb p(\W))
        \\
        \grad c_{\tau_n}^+\weaklyto\grad c
        \text{ in }\leb2(0,T;(\leb2(\W))^d).
      \end{split}
    \end{equation}
    From the extended result of \citet[Prop. II.2.30 on p64 and p92--93]{boyer_mathematical_2013}
    we have that
    \begin{equation}
      D(u_{\tau_n}^+)\grad c_{\tau_n}^+
      \weaklyto
      D(u)\grad c
      \text{ in }
      \leb1(0,T;(\leb r(\W))^d)
      \text{ with }
      r=\frac{2p}{p+2},
    \end{equation}
    which establishes (\ref{weakconvergencediffusionterm}).
    Similar arguments lead to (\ref{weakconvergenceconductivity}) and (\ref{weakconvergencetransportterm}). 
    Relation \eqref{weakconvergenceanisotropy} follows immediately from \citep[Lemma 5.4,
      p1152]{burmanExistenceSolutionsAnisotropic2003}.
  }
\end{proof}

\subsection{Proof of Theorem \ref{thmexistence}}\label{proofmainresult}
Now we can prove our main result, Theorem \ref{thmexistence}.

\begin{proof}
  In Theorem \ref{thmdiscreteexistence} we established the existence of a solution $(u^k,c^k,\phi^k)\in(\sobh 1(\Omega))^2\times\sobhz[\Gamma_1\cup\Gamma_2]1(\Omega)$ of \eqref{discretisedorderparameter}--\eqref{discretisedconcentration}. From Lemmas \ref{lemselectionofthelimits}--\ref{lemweakconvergence} we conclude the existence of a solution $(u,c,\phi)\in(\sobh 1(\Omega))^2\times\sobhz[\Gamma_1\cup\Gamma_2]1(\Omega)$ of \eqref{orderparameterfinalform}--\eqref{electricpotentialfinalform}, with $(u(0),c(0))=(u_0,c_0)\in\sobh 1(\Omega)\times\leb 2(\Omega)$ and 
  $u\in\leb 2([0,T];\sobh 1(\Omega))\cap\sobh 1([0,T];\leb 2(\Omega))\cap\leb {\infty}([0,T];\leb \infty(\Omega))$, $c\in\leb 2([0,T];\sobh 1(\Omega))\cap\sobh 1([0,T];(\sobh 1(\Omega))')\cap\leb {\infty}([0,T];\leb 2(\Omega))$ and $\phi\in\leb 2([0,T];\sobhz[\Gamma_1\cup\Gamma_2]1(\Omega))$. 
\end{proof}

\section{Numerical Simulations}\label{sec5}

In this section we present a finite element discretization of \eqref{orderparameterfinalform}-\eqref{electricpotentialfinalform} together with numerical examples that illustrate how dendritic structures can arise from the model. 

\subsection{Finite element approximation}\label{sec5.1}

Until now, for ease of presentation, we have considered a simplified version of the model from \citet{chenModulationDendriticPatterns2015} in which, where possible, all physical parameters were set to $1$. However in this section, in order to produce computational simulations in which dendritic patterns arise, we consider the original model from \citet{chenModulationDendriticPatterns2015} that has the physical parameters reinserted.

As in Section \ref{sec3}, we partition the time interval $[0,T]$
into $K$ equidistant steps, $0 = t_0 < t_1 < \ldots < t_{K-1} < t_K =
T$, so that for $k=1,\ldots,K$, $t_k := k\tau$.  We let
$\mathscr{T}_h$ be a regular triangulation of $\Omega$ into disjoint
open simplices $\mathcal{T}$ and we define $h :=
\max_{\mathcal{T}\in\mathscr{T}_h} \diam {\mathcal{T}}$ the maximal
element size of $\mathscr{T}_h$.

Associated with ${\mathscr{T}}_h$ are the finite element spaces
\begin{equation*}
  \begin{aligned}
    V_h = \{v\in\cont 0(\bar{\Omega});~  v|_{\mathcal{T}}\in P_2(\mathcal{T}),~\forall~ \mathcal{T}\in\mathscr{T}_h\}\subset \sobh 1(\Omega),
    \\
    V_{0,h} = \{v\in\cont 0(\bar{\Omega});~ v = 0 \text{ on } \Sigma\subseteq\partial\Omega;~  v|_{\mathcal{T}}\in P_2(\mathcal{T}),~\forall~ \mathcal{T}\in\mathscr{T}_h\}\subset\sobhz[\Sigma] 1(\Omega),
  \end{aligned}
\end{equation*}
where $\Sigma = \Gamma_1\cup\Gamma_2$ and $P_2(\mathcal{T})$ denotes the set of polynomials of maximum degree $2$  on $\mathcal{T}$.
\changes{$P_1(\mathcal{T})$ would be an alternative choice, but $P_2(\mathcal T)$ elements lead to smoother visual results at interfacial layers.}

Using a standard Galerkin finite element method in space and a semi-implicit Euler discretization in time we obtain the following discretisation of a variant of \eqref{orderparameterfinalform}-\eqref{electricpotentialfinalform} in which the physical parameters $\epsilon, \gamma, \nu, \mu$ from \citet{chenModulationDendriticPatterns2015} have been reinstated, the boundedness imposed on the $D_1(u,c)$, see \eqref{diffusionoftransport}, is removed and the electrochemical reaction rate $m(c)$ is replaced with a modified version $\tilde{m}(c)$ to highlight its exponential structure, see  \citep[Appendix B]{chenModulationDendriticPatterns2015}:

Set $u^0_h = u_0(\vec x)$ and $c^0_h = c_0(\vec x)$, then for $k=1,2\ldots,K$, given $(u^{k-1}_h,c^{k-1}_h)\in [V_h]^2$ find $(u^k_h,c^k_h,\phi^k_h)\in [V_h]^2\times V_{0,h}$ such that for all $(\chi_h,\eta_h)\in V_h\times V_{0,h}$
\begin{subequations}
  \begin{gather}
    \label{fullydiscretisedorderparameter}
    \begin{split}
      \epsilon^{2}( d_{\tau} u^k_h,\chi_h)
      +
      (\mat A(\nabla u^{k-1}_h)\nabla u^k_h,\nabla \chi_h)
      +
      \gamma(g'(u^{k-1}_h),\chi_h)
      \\
      = ( \tilde{m}(c^{k-1}_h)h'(u^{k-1}_h),\chi_h),
    \end{split}
    \\
    \label{fullydiscretisedelectricpotential}
    (\sigma(u^{k}_h)\nabla\phi^k_h,\nabla \eta_h)
    =
    -\nu(d_{\tau} u^{k}_h,\eta_h),
    \\
    \label{fullydiscretisedconcentration}
    ( d_{\tau} c^k_h,\chi_h) + (D(u^{k}_h)\nabla c^k,\nabla \chi_h)
    +
    (D(u^{k}_h) c^{k-1}_h\nabla\phi^{k}_h,\nabla \chi_h)
    =
    -\mu( d_{\tau} u^{k}_h,\chi_h),
  \end{gather}
\end{subequations}
where $\tilde{m}(c)=C_1(\expp{-C_2}-c\expp{C_2})$.  The physical
parameters are defined, together with their values in the
computations, in Table \ref{table1}. We note that
\eqref{fullydiscretisedorderparameter}--\eqref{fullydiscretisedconcentration}
are linear in $u^k$, $\phi^k$ and $c^k$ respectively.
\changes{Also, the term $\sigma'(u^k_h)\partial_{x_1}u^k_h$ is removed from
  \eqref{fullydiscretisedelectricpotential} and incorporated in the boundary
  conditions for $\phi^k_h$.}

The computational results we present are produced by implementing a standard adaptive strategy, in particular, we use a gradient-based indicator which involves the gradient of the solution of the phase-field equation and the reaction-diffusion equation, given by
\begin{equation}
  \eta_{\mathcal{T}}:= \sqrt{\abs{\nabla u^k_h}^2_{\mathcal{T}} + \abs{\nabla c^k_h}^2_{\mathcal{T}}},
\end{equation}
such that simplicies on which $\eta_{\mathcal{T}}\geq 160$ are refined and simplices on which $\eta_{\mathcal{T}}\leq 80$ are coarsened.

\subsection{Numerical examples}\label{sec5.2}
In this section we present numerical simulations of dendritic crystal growth. 

We set $d=2$ and take $\Omega = [0,9]\times[-9,9]$ and $T=0.5$. For an
initial set-up we place a pill-shaped solid seed
\citep[cf.][Fig 1]{akolkarModelingDendriteGrowth2014}
at the mid-point of the left-hand boundary of $\Omega$ and set the ion
concentration of \ch{Li+} to be $0$ along the left hand boundary of
$\Omega$, which represents the right most edge of the electrode, and
$1$ in the reminder of the domain.  Thus giving rise to the following
initial data for the order parameter and the concentration:
\begin{equation}
  u_0(x_1,x_2) = \Bigg\{
  \begin{array}{ll}
    1, & (x_1 - r_0 - \sqrt{r_0^2 - x_2^2})(x_1 + r_0 + \sqrt{r_0^2 - x_2^2})(x_2^2 - r_0^2)\geq 0,\\
    0, & \text{elsewhere}.
  \end{array}
\end{equation}
and
\begin{equation}\label{c0}
  c_0(x_1,x_2) = \Bigg\{
  \begin{array}{ll}
    0, & x_1 = 0,\\
    1, & \text{elsewhere.}
  \end{array}
\end{equation}
In the simulations presented below we set $r_0=0.1$.

For the implementation of the numerical scheme we used the DUNE Python
module, \citet{dednerDunePythonModule2018}, and the DUNE Alugrid
module, \citet{alkamperDUNEALUGridModule2015}.  Since the solution is
symmetric in the $x_2$-direction, for the sake of computational cost,
we reduced our computations from the rectangular domain
$\Omega=[0,9]\times [-9,9]$ to the square domain
$\Omega'=[0,9]^2$. The values of the parameters used in the
simulations are presented in Table \ref{table1}.

We consider the polar form of the anisotropy function \eqref{anisotropy_function}:
\begin{equation}\label{anisotropy_function_polar}
  \tilde{a}(\theta) = a_0(1 + 4\delta \cos(4\theta)),
\end{equation}
with $\theta = \arctan(\partial_{x_2} u/\changes{\partial_{x_1} u})$, %
which yields the anisotropy tensor 
\begin{equation}
  \mat A(\theta)
  =
  \begin{bmatrix}
    \tilde{a}^2(\theta) & -\tilde{a}(\theta)\tilde{a}'(\theta)
    \\
    \tilde{a}(\theta)\tilde{a}'(\theta) & \tilde{a}^2(\theta)
  \end{bmatrix}.
\end{equation}
We present four sets of results, Figure \ref{systemfigure} -- Figure
\ref{fignoise}. In Figure \ref{systemfigure} we show a simulation of
the evolution of a dendritic structure, while in Figures
\ref{comparisondeltas} and \ref{comparisonbetac} we examine how the
dendritic structure is affected by variations in the anisotropy
strength $\delta$ and in the parameter $\mu$ respectively. In Figure
\ref{fignoise} we investigate the effect adding stochastic noise has
on the dendritic structure.

Figure \ref{systemfigure} depicts the numerical approximations of the
order parameter, $u_h^k$ (top row), the concentration of lithium
atoms, $c_h^k$ (middle row), and the electrostatic potential,
$\phi_h^k$ (bottom row), each displayed at $t_k = 0.061$ (left) $t_k =
0.244$ (middle) and $t_k = 0.427$ (right). In these results we set
$\delta = 0.05$ and $\mu = 2$. From the results for the order
parameter we clearly see the formation of a dendritic structure, while
a more diffuse form of this structure is mirrored in the concentration
results and a very diffuse underlying dendritic structure can also be
observed in the electric potential results, with the region, in which
$\abs{\phi^k_h}$ attains its maximum value, growing as the dendrite
grows, thus eventually causing the battery to fail.  When comparing
the dendritic structure that forms in Figure \ref{systemfigure} to the
dendritic structures that are displayed in Figure 4 in
\citet{chenModulationDendriticPatterns2015} and Figure 3a in
\citet{muNumericalSimulationFactors2019} we see that in Figure
\ref{systemfigure} the structure is smooth and does not form the side
branches perpendicular to the growth direction of the dendritic
structure, that can be observed in
\citet{chenModulationDendriticPatterns2015} and
\citet{muNumericalSimulationFactors2019}. In Figures
\ref{comparisondeltas} and \ref{fignoise} below, we see that dendritic
structures with side branches can be observed by increasing the
anisotropy strength, $\delta$, or adding stochastic noise.

In Figure \ref{comparisondeltas} we examine how the anisotropy
strength, $\delta$, affects the formation of dendritic crystals. To
this end we recall the bound $\delta<1/15 \approx 0.067$ that we
impose in assumptions \eqref{assump}, recall \eqref{assump}, in order
for our theoretical results to hold. The specific nature of the bound
can be understood from the proof of Lemma \ref{lem:hessianpositive}.
Figure \ref{comparisondeltas} displays plots of the numerical
approximation of the order parameter, $u_h^k$, at $t_k=0.366$ for $\mu
= 2$ and six values of $\delta$. From this figure we see that the
anisotropy strength has a notable impact on the rate at which the
dendritic structure grows, in particular the speed of growth increases
as $\delta$ decreases. In addition, we see that the value $\delta =
1/15 \approx 0.067$ plays a significant role; in the simulations where
$\delta \ll 1/15 \approx 0.067$, i.e. Figure
\ref{orderparameterdelta0} -- Figure \ref{orderparameterdelta0.05} in
which $\delta = 0,0.01,0.02,0.05$ respectively, the geometry of the
dendritic structure is smooth, while in Figure
\ref{orderparameterdelta0.065}, where $\delta =0.065$ and hence the
condition $\delta <1/15$ is only just satisfied, the structure is no
longer smooth and small instabilities, perpendicular to the growth of
the dendritic structure, have started to form. The instabilities are
more pronounced in Figure \ref{orderparameterdelta0.1}, which shows
$\delta = 0.1$ which is significantly greater than $\delta =
1/15$. %
$\delta \geq 1/15$, the numerical scheme still produces solutions for
values of $\delta>1/15$.\\

In Figure \ref{comparisonbetac} we examine how the evolution of the order parameter, $u$, varies with the parameter $\mu$, which is related to the ratio of site density of Li and the standard bulk concentration of electrolyte solution as described in \citet{chenModulationDendriticPatterns2015}. We display plots of the numerical approximation of the order parameter, $u_h^k$, at $t_k=0.122$ for $\delta = 0.05$ and five values of $\mu$, together with, for comparison, a plot of $u_h^k$ resulting from solving the system \eqref{fullydiscretisedorderparameter}--\eqref{fullydiscretisedconcentration} with $\delta =0$ and $\mu=0$.
From this figure we see that for small values of $\mu$, i.e. $\mu=0$ and $\mu=0.5$, no dendritic structures arise and instead the geometry takes a form that is relatively radially symmetric, while for larger values of $\mu$, namely $\mu=1,1.5,2$ dendritic patterns arise, with the dendritic structure becoming more pronounced, yet growing more slowly, as $\mu$ increases. 

As observed above, in Figure \ref{systemfigure} no side branches are formed perpendicular to the growth direction of the dendritic structure, whereas in the results in \citet{chenModulationDendriticPatterns2015} and \citet{muNumericalSimulationFactors2019} side branches do form. From Figure \ref{comparisondeltas} we see that side branches are observed if the anisotropy strength, $\delta$, is taken large enough. An alternative approach to yield the formation of side branches is to add stochastic noise to the model, to this end we follow the authors in \citet{kobayashiModelingNumericalSimulations1993} and add a stochastic term, $\beta u_h^k(1-u^k_h)X$, to \eqref{fullydiscretisedorderparameter}, so that it becomes
$$
\epsilon^{2}( d_{\tau} u^k_h,\chi_h) + (\mat A(\nabla u^{k-1}_h)\nabla u^k_h,\nabla \chi_h) + \gamma(g'(u^{k-1}_h),\chi_h) = ( \tilde{m}(c^{k-1}_h)h'(u^{k-1}_h),\chi_h)
$$
\begin{equation}\label{acnoise}
  \hspace{3cm}-(\beta u_h^k(1-u^k_h)X,\chi_h), ~~\forall~\chi_h\in V_h,
\end{equation}
where $X\sim\mathcal{U}(-1/2,1/2)$ and $\beta > 0$  is a scalar that controls the noise strength. Results obtained from the model with \eqref{fullydiscretisedorderparameter} replaced by \eqref{acnoise}, and $\delta=0.05$ and $\mu=2$, are shown in Figure \ref{fignoise}, which displays plots of the numerical approximation of the order parameter, $u_h^k$, at $t_k=0.305$ for four values of $\beta$. From this figure we see that when $\beta=0.03$ side branches have started to form and that the length of these branches increases as $\beta$ increases.

\newpage

\begin{table}[h]
  \caption{Dimensionless parameters.}\label{table1}
  \begin{tabular}{ p{63mm}p{13mm}p{49mm}  }
    \hline
    Parameter & Notation & Value \\
    \hline
    Diffusion coefficient in the electrolyte  & $D^{\text{e}}$ & $3.0$  \\
    Diffusion coefficient in the electrode &  $D^{\text{s}}$ & $2.5$ \\
    Conductivity coefficient in the electrolyte & $\sigma^{\text{e}}$ & $100.0$ \\
    Conductivity coefficient in the electrode & $\sigma^{\text{s}}$ & $10^7$\\
    Interfacial mobility & $\gamma$ & $1.0$\\
    Characteristic time scale & $\epsilon^2$ & $0.0003$\\
    Anisotropy strength& $\delta$ & $0.05$ (unless otherwise specified)\\
    Interfacial thickness & $a_0$ & $0.01$\\
    Electrochemical reaction rate coefficients& $C_1$ \& $C_2$ & $1/15$ \& $1.9345$\\
    Coefficient multiplying $d_\tau u_h^k$ in \eqref{fullydiscretisedelectricpotential} & $\nu$ & $76.4$\\
    Coefficient multiplying $d_\tau u_h^k$ in  \eqref{fullydiscretisedconcentration}  & 
    $\mu$ & $2.0$ (unless otherwise specified)\\
    \hline
  \end{tabular}
  
\end{table}

\begin{figure}[H]
  \centering
  \includegraphics[width=0.275\linewidth,height=0.275\linewidth]{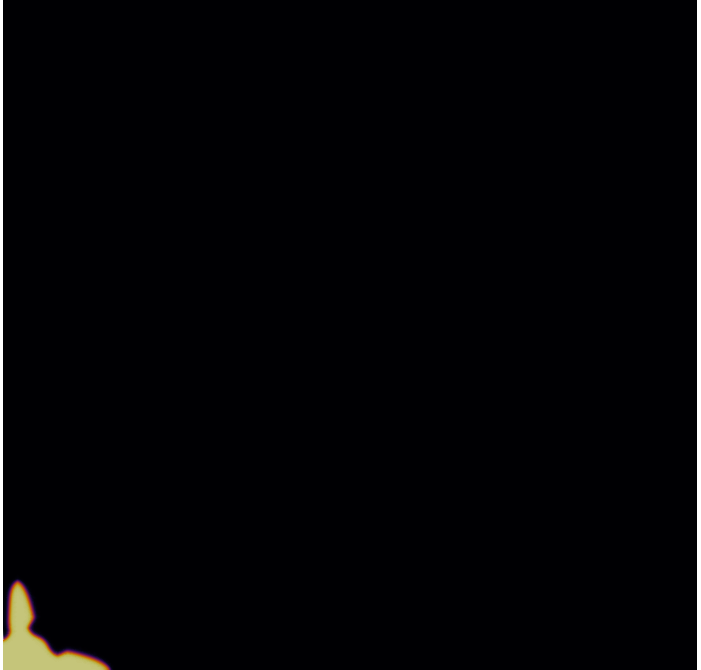}%
  \hfill %
  \includegraphics[width=0.275\linewidth,height=0.275\linewidth]{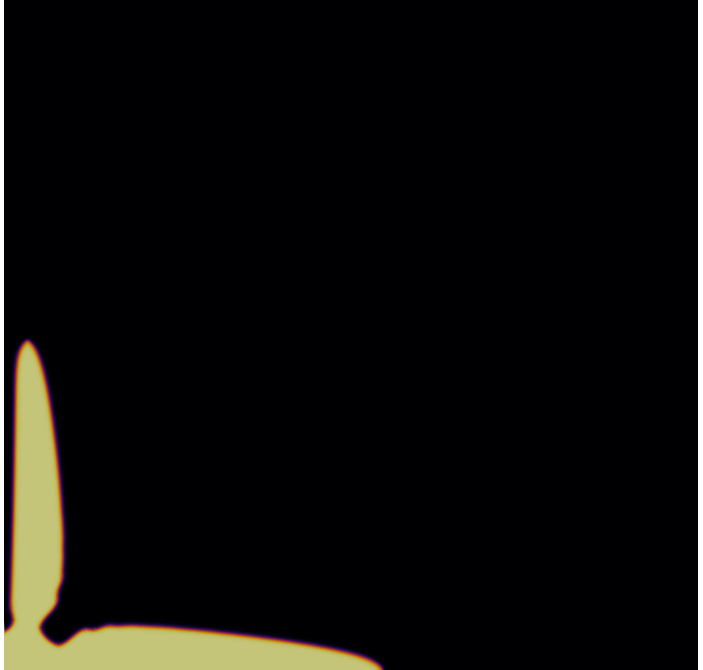}%
  \hfill %
  \includegraphics[width=0.275\linewidth,height=0.275\linewidth]{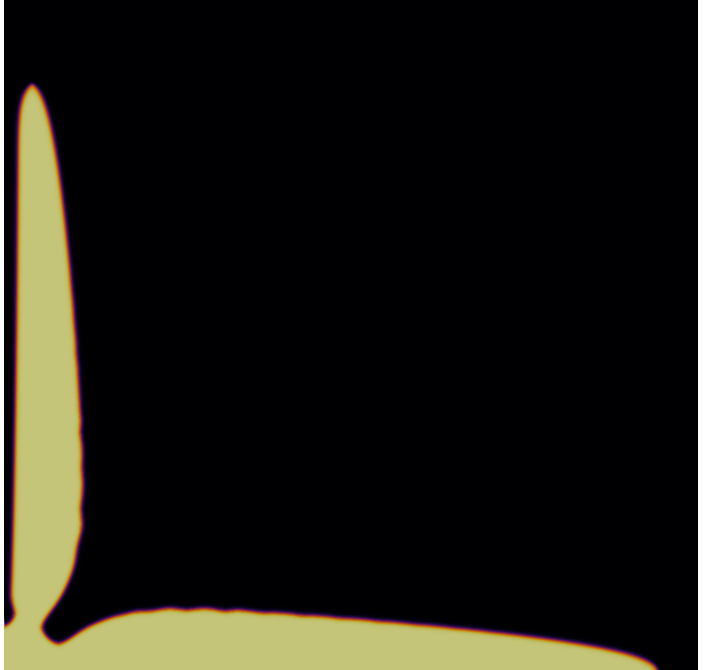}
  \hfill
  \includegraphics[width=0.13\linewidth,height=0.275\linewidth]{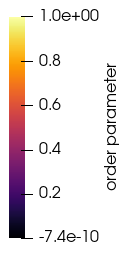}%
  \vskip\baselineskip %
  \includegraphics[width=0.275\linewidth,height=0.275\linewidth]{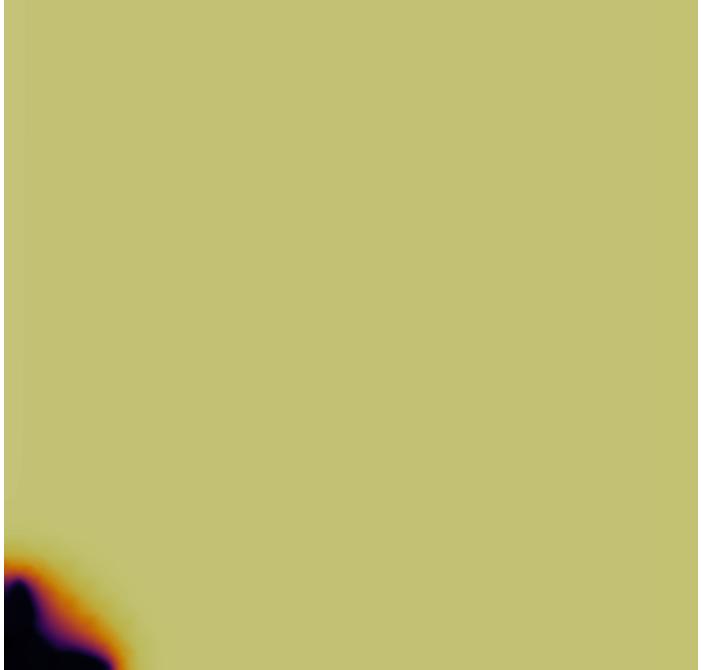}%
  \hfill %
  \includegraphics[width=0.275\linewidth,height=0.275\linewidth]{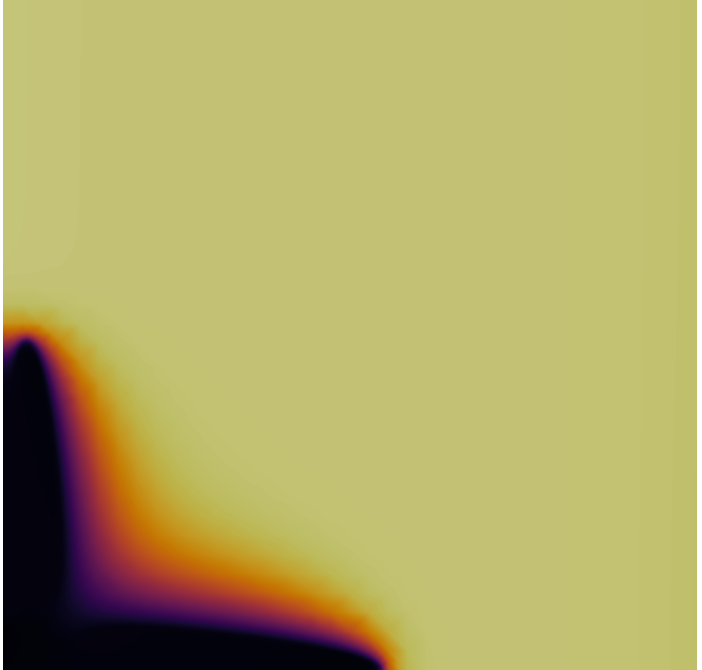}%
  \hfill %
  \includegraphics[width=0.275\linewidth,height=0.275\linewidth]{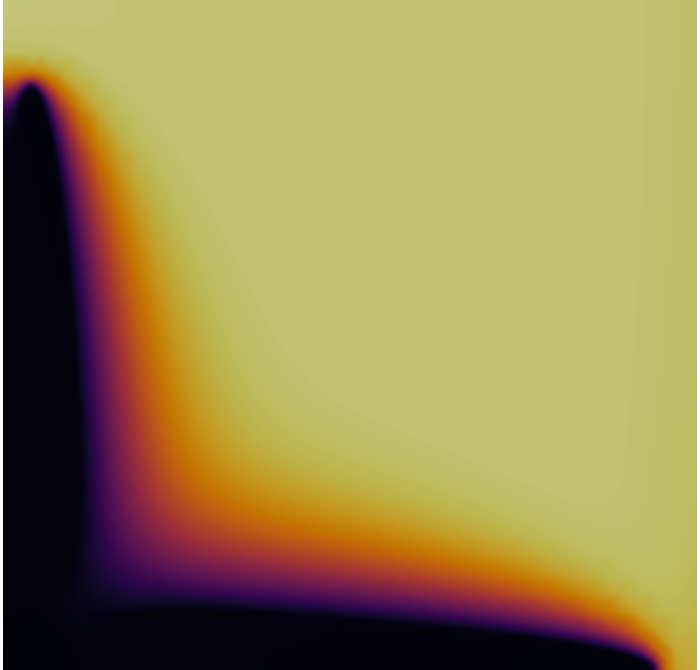}
  \hfill
  \includegraphics[width=0.13\linewidth,height=0.275\linewidth]{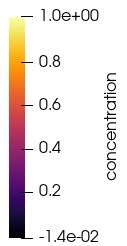}%
  \vskip\baselineskip
  \includegraphics[width=0.275\linewidth,height=0.275\linewidth]{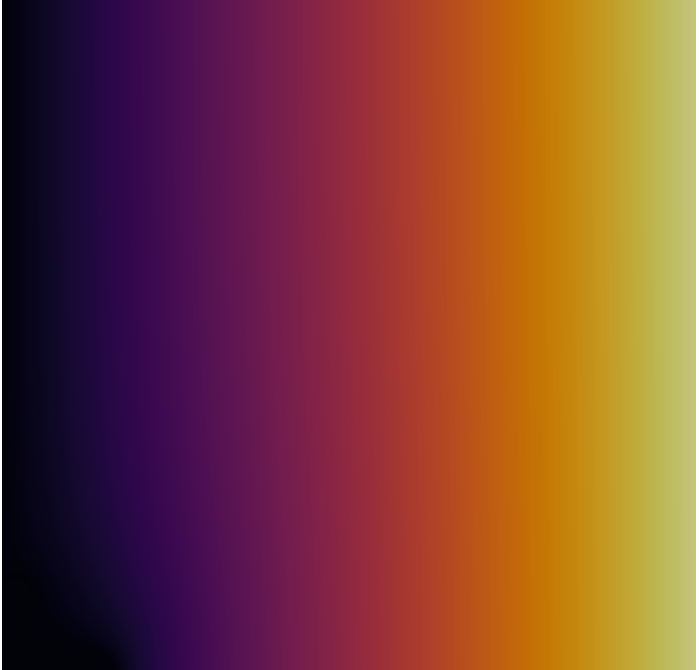}%
  \hfill %
  \includegraphics[width=0.275\linewidth,height=0.275\linewidth]{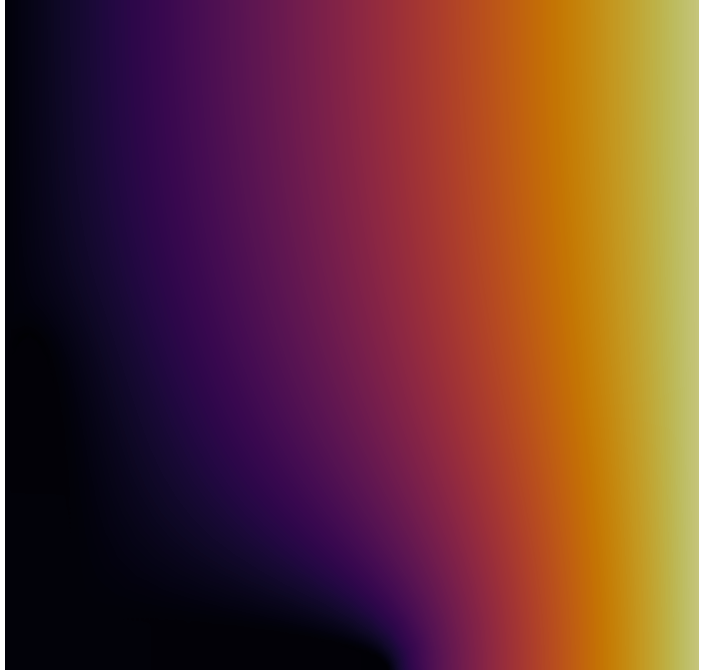}%
  \hfill %
  \includegraphics[width=0.275\linewidth,height=0.275\linewidth]{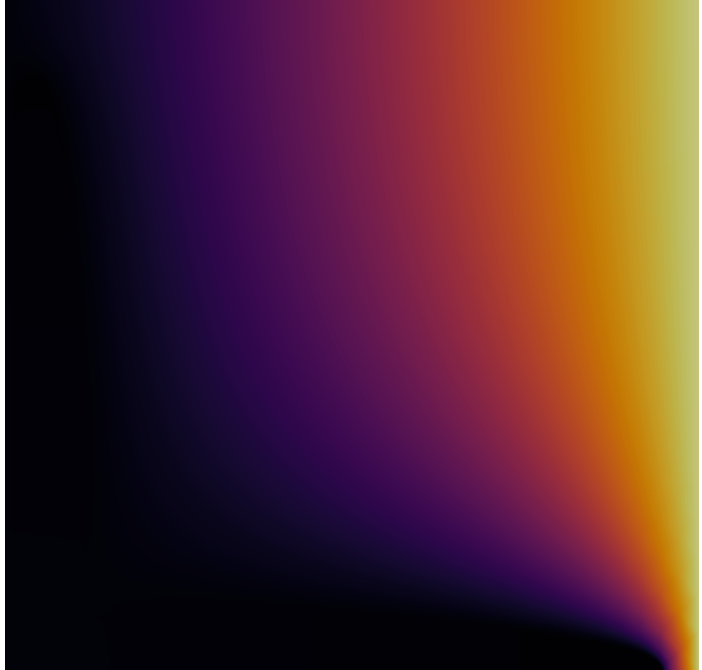}
  \hfill
  \includegraphics[width=0.13\linewidth,height=0.275\linewidth]{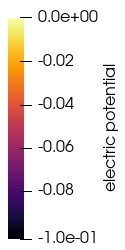}%
  \hfill %
  \caption{Dendritic formation with $\delta = 0.05$ and $\mu = 2$: $u_h^k$ (top row), $c_h^k$ (middle row) and $\phi_h^k$ (bottom row) displayed at $t_k = 0.061$ (left) $t_k = 0.244$ (middle) and $t_k = 0.427$ (right)} \label{systemfigure}
\end{figure}

\begin{figure}[H]
  \centering
  \subcaptionbox{$\delta = 0$\label{orderparameterdelta0}}{
    \includegraphics[width=0.3\linewidth,height=0.3\linewidth]{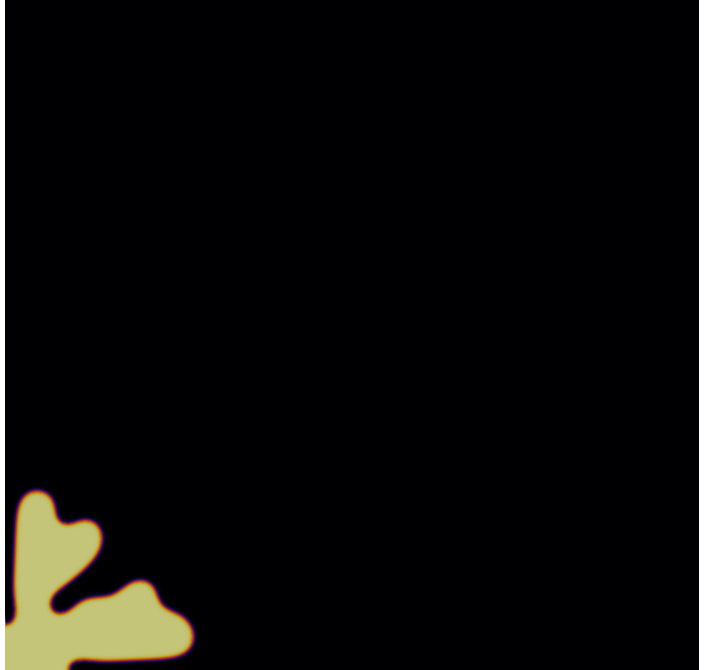}}%
  \quad %
  \subcaptionbox{$\delta = 0.01$\label{orderparameterdelta0.01}}{
    \includegraphics[width=0.3\linewidth,height=0.3\linewidth]{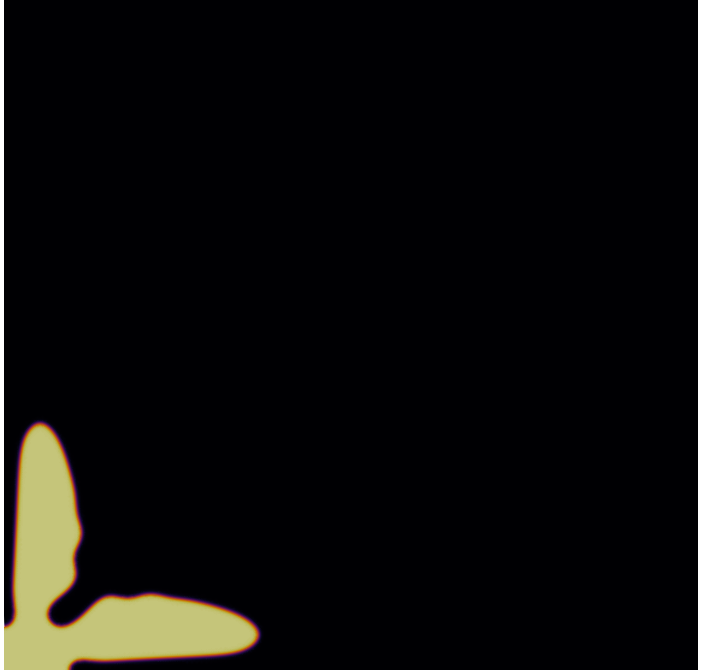}}%
  \quad%
  \subcaptionbox{$\delta = 0.02$\label{orderparameterdelta0.02}}{
    \includegraphics[width=0.3\linewidth,height=0.3\linewidth]{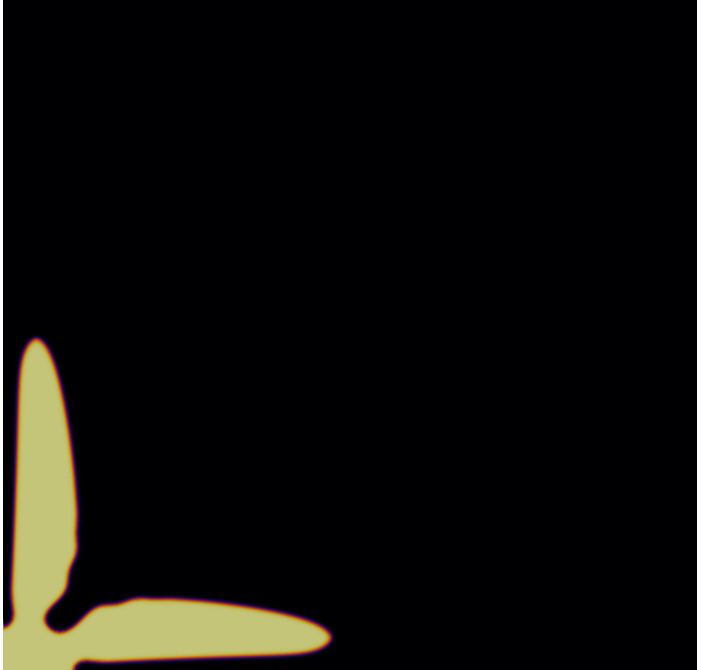}}%
  \vskip\baselineskip
  \subcaptionbox{$\delta = 0.05$\label{orderparameterdelta0.05}}{
    \includegraphics[width=0.3\linewidth,height=0.3\linewidth]{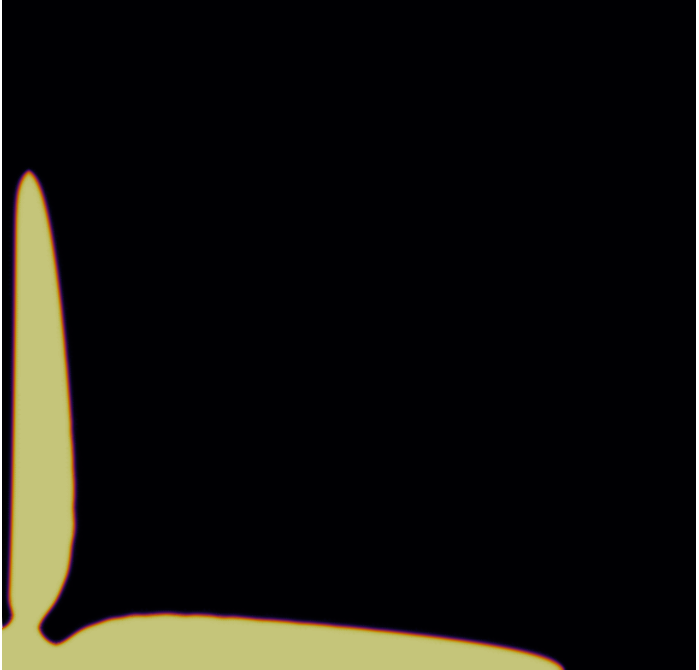}}%
  \quad %
  \subcaptionbox{$\delta = 0.065$\label{orderparameterdelta0.065}}{
    \includegraphics[width=0.3\linewidth,height=0.3\linewidth]{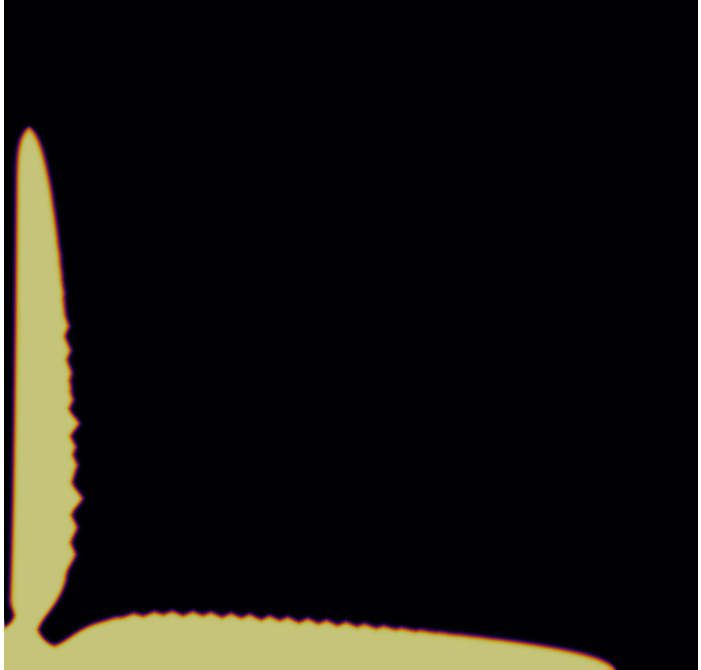}}%
  \quad %
  \subcaptionbox{$\delta = 0.1$\label{orderparameterdelta0.1}}{
    \includegraphics[width=0.3\linewidth,height=0.3\linewidth]{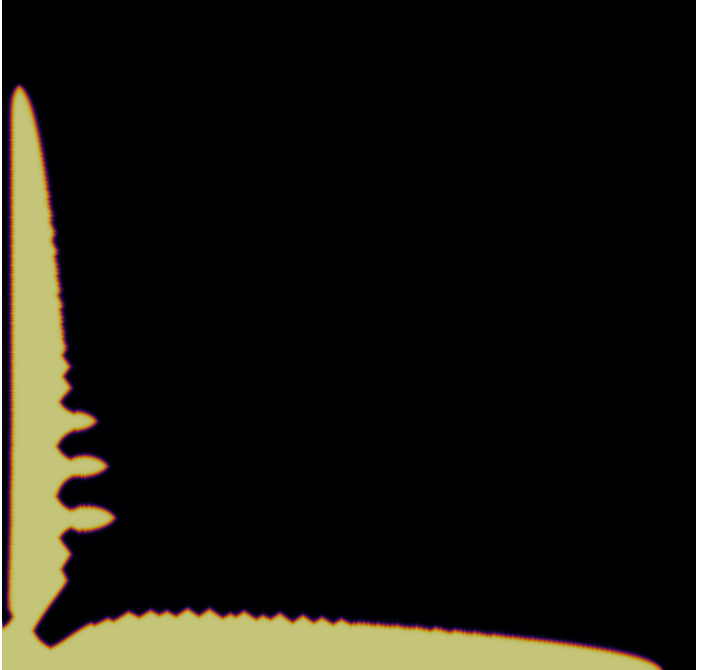}}%
  \quad %
  \caption{Comparison of $u_h^k$ at $t_k=0.366$ with $\mu = 2$ and different values for the anisotropy strength $\delta$, introduced in \eqref{anisotropy_function}.}\label{comparisondeltas}
\end{figure}

\begin{figure}[H]
  \centering
  \subcaptionbox{isotropic case\label{orderparameterisotropic}}{
    \includegraphics[width=0.3\linewidth,height=0.3\linewidth]{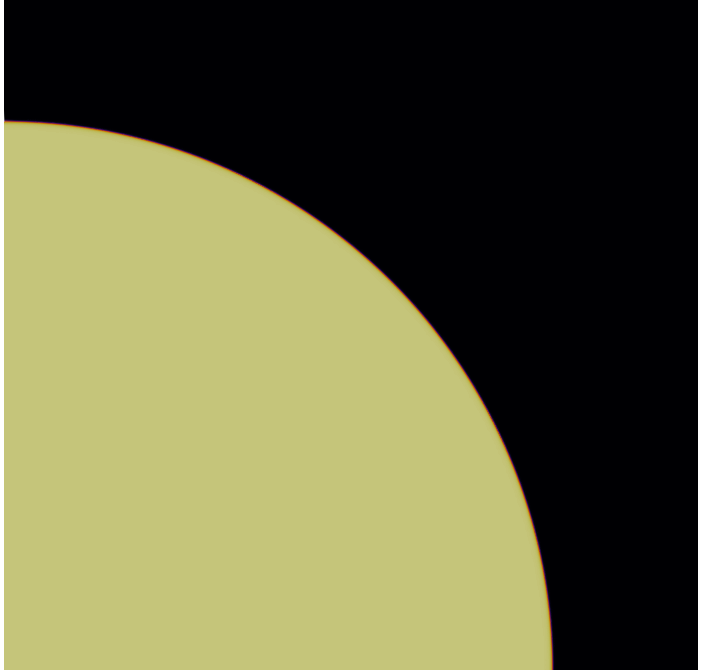}}%
  \quad %
  \subcaptionbox{$\mu = 0$\label{orderparameterbetac0}}{
    \includegraphics[width=0.3\linewidth,height=0.3\linewidth]{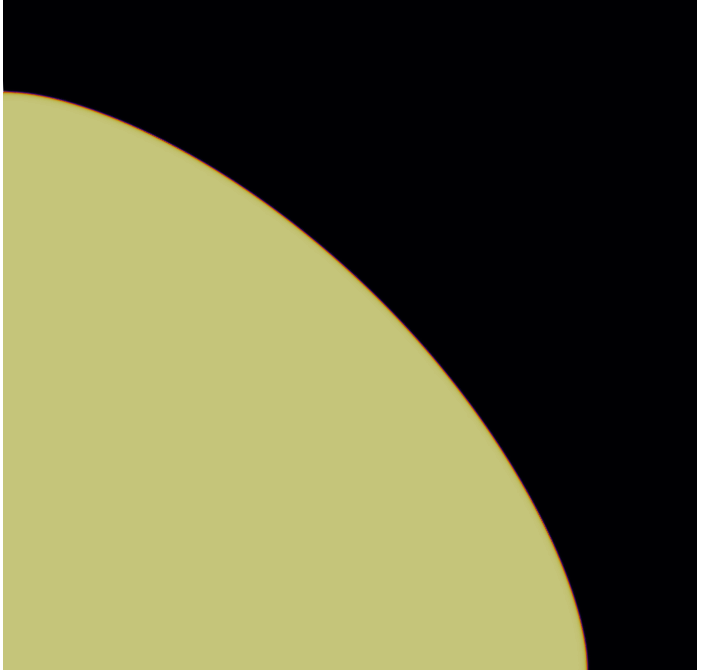}}%
  \quad %
  \subcaptionbox{$\mu = 0.5$\label{orderparameterbetac0.5}}{
    \includegraphics[width=0.3\linewidth,height=0.3\linewidth]{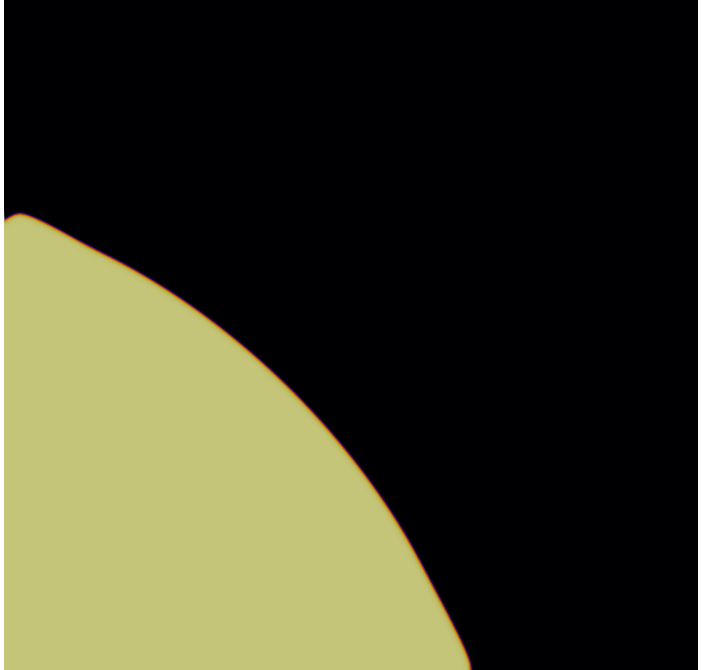}}%
  \quad%
  \subcaptionbox{$\mu = 1.0$\label{orderparameterbetac1.0}}{
    \includegraphics[width=0.3\linewidth,height=0.3\linewidth]{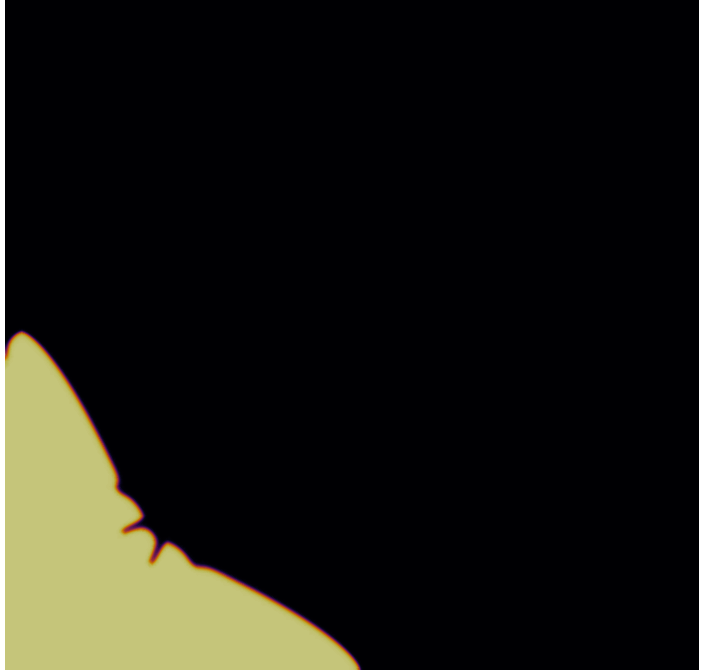}}%
  \quad
  \subcaptionbox{$\mu = 1.5$\label{orderparameterbetac1.5}}{
    \includegraphics[width=0.3\linewidth,height=0.3\linewidth]{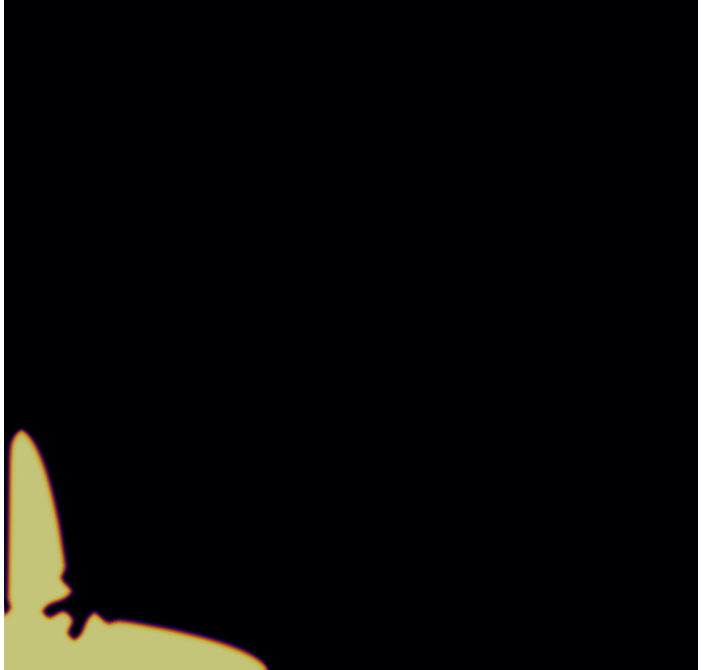}}%
  \quad %
  \subcaptionbox{$\mu = 2.0$\label{orderparameterbetac2.0}}{
    \includegraphics[width=0.3\linewidth,height=0.3\linewidth]{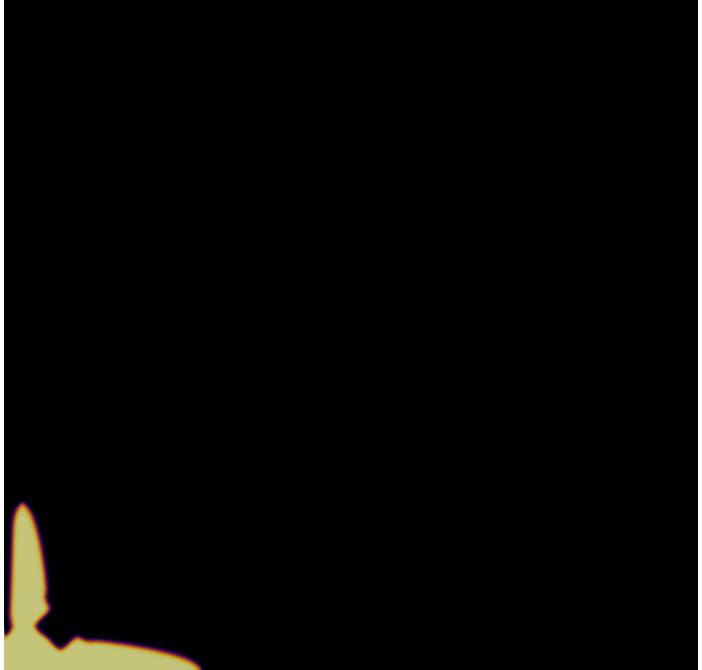}}%
  \quad %
  \caption{Comparison of $u_h^k$ at $t_k=0.122$ with $\delta = 0.05$ for different values of $\mu$.}\label{comparisonbetac}
\end{figure}

\begin{figure}[H]
  \centering
  \subcaptionbox{$\beta =0$\label{orderparameternonoise}}{
    \includegraphics[width=0.4\linewidth,height=0.4\linewidth]{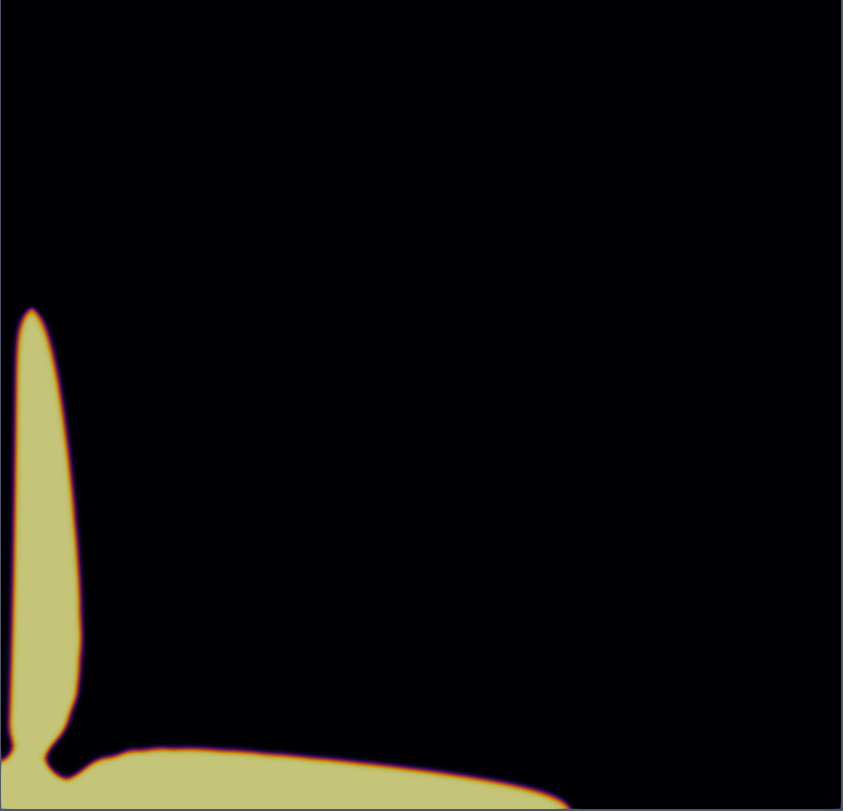}}
  \quad %
  \subcaptionbox{$\beta = 0.01$\label{orderparameternoise0.01}}{
    \includegraphics[width=0.4\linewidth,height=0.4\linewidth]{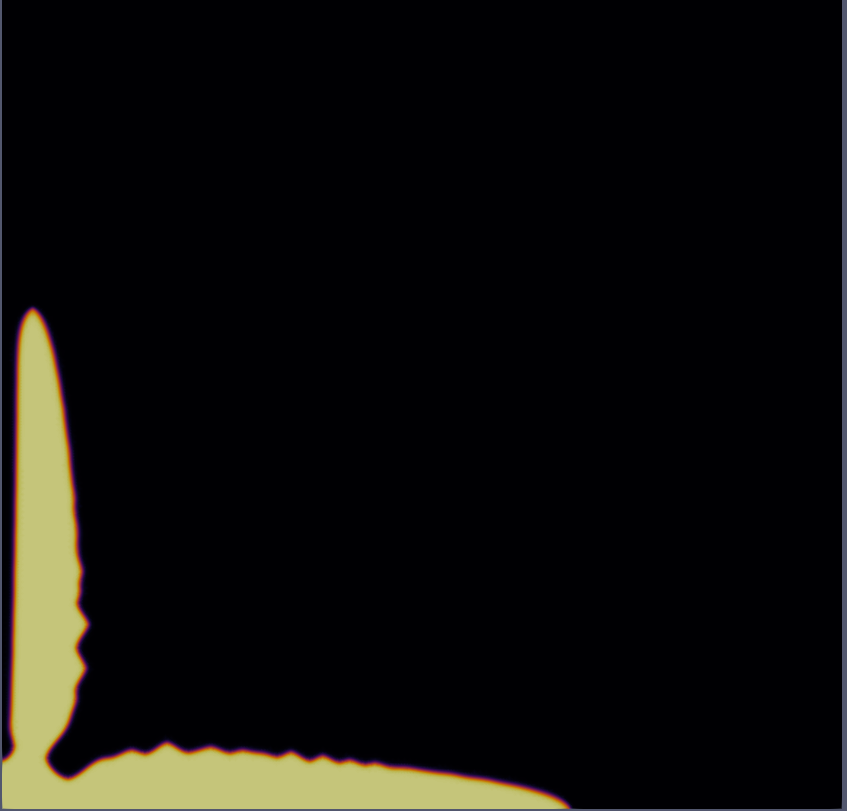}}%
  \quad %
  \subcaptionbox{$\beta = 0.03$\label{orderparameternoise0.03}}{
    \includegraphics[width=0.4\linewidth,height=0.4\linewidth]{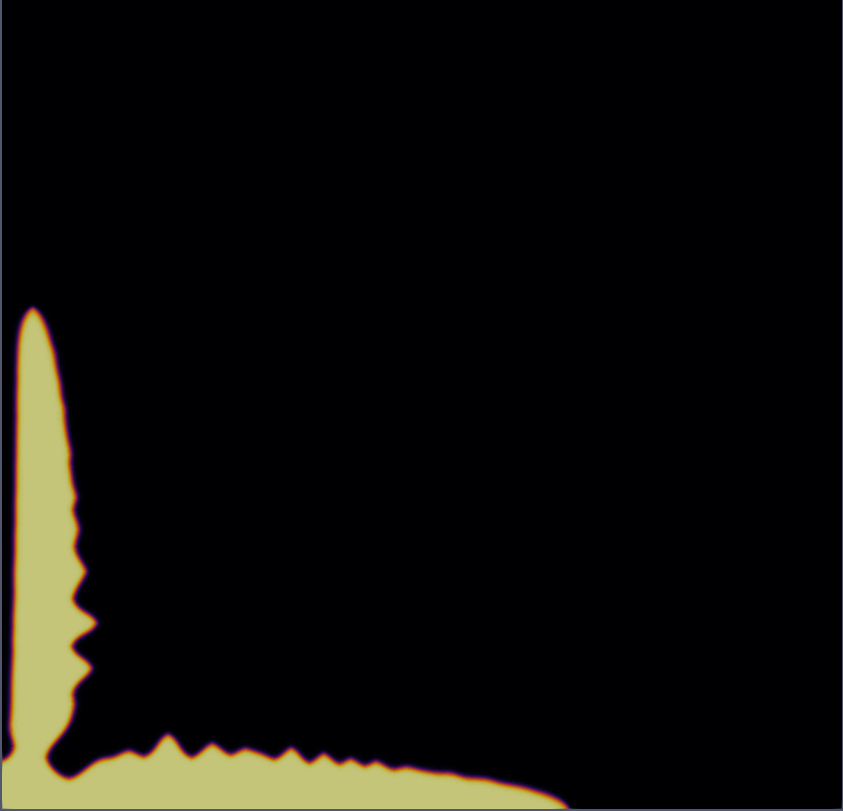}}%
  \quad %
  \subcaptionbox{$\beta = 0.1$\label{orderparameternoise0.1}}{
    \includegraphics[width=0.4\linewidth,height=0.4\linewidth]{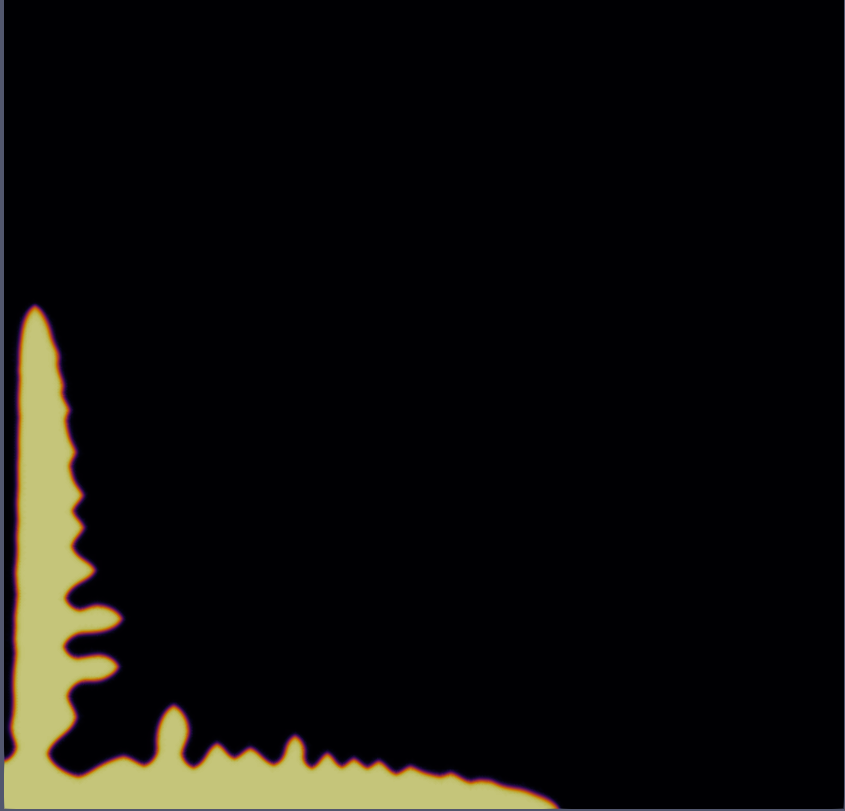}}%
  \quad %
  \caption{Comparison of $u_h^k$ at $t_k=0.305$ with $\delta = 0.05$ and $\mu = 2$, for solutions to the model with \eqref{fullydiscretisedorderparameter} replaced by \eqref{acnoise} for different values of the noise amplitude $\beta$.}\label{fignoise}
\end{figure}

\section*{Acknowledgements}
AS was partially funded by the European Union (ERC Synergy, NEMESIS,
project number 101115663).

Views and opinions expressed are however those of the authors only and
do not necessarily reflect those of the European Union or the European
Research Council Executive Agency. Neither the European Union nor the
granting authority can be held responsible for them.

OL and VS were partly supported in this work by the Dr Perry James
(Jim) Browne Research Centre on Mathematics and its Applications at
the University of Sussex, UK.

\bibliographystyle{abbrvnat}%

\end{document}